\documentclass[11pt, twoside]{article}
\pdfoutput=1
\usepackage{lspaper}
\usepackage{multirow}
\usepackage[normalem]{ulem}

\numberwithin{equation}{section}
\numberwithin{figure}{section}

\newcommand{\Z}{{\mathbb Z}}
\newcommand{\N}{{\mathbb N}}

\DeclareMathOperator{\ind}{ind}
\DeclareMathOperator{\id}{id}
\DeclareMathOperator{\diag}{diag}

\begin{document}

\title{Strong Gelfand subgroups of $F\wr S_n$}

\author{*Mahir Bilen Can, *Yiyang She, \textdagger Liron Speyer\\
\\\normalsize *Department of Mathematics, Tulane University,\\\normalsize
6823 St.~Charles Ave, New Orleans, LA, 70118, USA\\
\texttt{\normalsize mcan@tulane.edu}\\
\texttt{\normalsize yshe@tulane.edu}\\
\\\normalsize
\textdagger Okinawa Institute of Science and Technology Graduate University,\\\normalsize
1919-1 Tancha, Onna-son, Okinawa, Japan 904-0495\\
\texttt{\normalsize liron.speyer@oist.jp}}

\renewcommand\auth{Mahir Bilen Can, Yiyang She, Liron Speyer}

\runninghead{Strong Gelfand subgroups of $F\wr S_n$}

\toptitle

\begin{abstract}

The multiplicity-free subgroups (strong Gelfand subgroups) of wreath products are investigated. 
Various useful reduction arguments are presented. 
In particular, we show that for every finite group $F$, the wreath product $F\wr S_\lambda$,
where $S_\lambda$ is a Young subgroup, is multiplicity-free if and only if $\lambda$ is a partition with at most two parts,
the second part being 0, 1, or 2. 
Furthermore, we classify all multiplicity-free subgroups of hyperoctahedral groups. 
Along the way, we derive various decomposition formulas for the induced representations from some special subgroups of hyperoctahedral groups.
\vspace{.5cm}

\end{abstract}

\section{Introduction}

Let $K$ be a subgroup of a group $G$. 
The pair $(G,K)$ is said to be a {\em Gelfand pair} if the induced trivial representation $\ind^G_K \mathbf{1}$ 
is a multiplicity-free $G$ representation. 
More stringently, if $(G,K)$ has the property that 
\[
\text{$\ind^G_K V$ is multiplicity-free for every irreducible $K$ representation $V$}, 
\] 
then it is called a {\em strong Gelfand pair}. In this case, $K$ is called a {\em strong Gelfand subgroup}
(or a {\em multiplicity-free subgroup}).
Clearly, a strong Gelfand pair is a Gelfand pair, however, the converse need not be true. 
The problem of finding all multiplicity-free subgroups of an algebraic group goes back to Gelfand and Tsetlin's works~\cite{GT1,GT2}, where 
it was shown that $GL_{n-1}(\bbc)$ (resp.~$Spin_{n-1}(\bbc)$) is a multiplicity-free subgroup in $SL_n(\bbc)$ (resp.~in $Spin_n(\bbc)$).
$(SL_n(\bbc), GL_{n-1}(\bbc))$ and $(Spin_n, Spin_{n-1})$ are strong Gelfand pairs. 
It was shown by Kr\"amer in~\cite{Kramer} that for simple, simply connected (complex or real) algebraic groups, 
there are no additional pairs of strong Gelfand pairs.
If the ambient group $G$ is allowed to be a reductive group
and/or the underlying field of definitions is changed, then there are many more strong Gelfand pairs~\cite{Aizenbudetal, SunZhu,KobayashiMatsuki}.

For arbitrary finite groups, there is much less is known about the strong Gelfand pairs. 
In this paper we will be exclusively concerned with representations in characteristic 0. 
Building on Saxl's prior work~\cite{Saxl75,Saxl81}, 
the list of all Gelfand pairs of the form $(S_n,K)$, where $S_n$ is a symmetric group, 
is determined by Godsil and Meagher in~\cite{GodsilMeagher}.
Recently, it was shown by Anderson, Humphries, and Nicholson~\cite{AHN} that, for $n\geq 7$, the only 
strong Gelfand pairs of the form $(S_n,K)$ are given by 
\begin{enumerate}
\item $(S_n, S_n)$,
\item $(S_n,A_n)$,
\item $(S_n,S_1\times S_{n-1})$ (up to interchange of the factors), and 
\item $(S_n,S_2\times S_{n-2})$ (up to interchange of the factors).
\end{enumerate}
Also recently, in~\cite{Tout}, Tout proved that for a finite group $F$, the pair $(F\wr S_n,F\wr S_{n-1})$ 
is a Gelfand pair if and only if $F$ is an abelian group.

In this article, we consider the strong Gelfand pairs of the form $(F\wr S_n, K)$, where $F$ is a finite group. 
Although the strongest results of our paper are about the pairs with $F= \Z/2$, 
we prove some general theorems when $F$ is a finite (abelian) group.  
The main purpose of our article is two-fold.  
First, we give a formula for computing the multiplicities of the irreducible $F\wr S_n$ representations in the induced representations $\ind_{S_n}^{F\wr S_n} V$,
where $V$ is any irreducible representation of $S_n$.
Secondly, we will determine all strong Gelfand pairs $(F\wr S_n, H)$, where $F= \Z/2$.
Along the way, we will present various branching formulas for such pairs.

We are now ready to give a brief outline of our paper and summarize its main results. 
In Section~\ref{S:Preliminaries}, we collect some well-known results from the literature. 

The first novel results of our paper appear in Section~\ref{S:PassiveFactor}, where 1) we prove a key lemma that we use later for describing some branching rules in wreath products, 2) we describe the multiplicities of the irreducible representations in $\ind^{F\wr S_n}_{S_n} S^\lambda$ where $F$ is an abelian group, 
and $S^\lambda$ is a Specht module of $S_n$ labeled by the partition $\lambda$ of $n$.
Roughly speaking, in Theorem~\ref{T:firststep}, we show that the multiplicities are determined by the 
Littlewood--Richardson rule combined with the descriptions of the irreducible representations of wreath products. 
As a corollary, we show that the pair $(F\wr S_n, S_n)$ is not a strong Gelfand pair for $n\geq 6$. 
It is easy to find negative examples for the converse statement.
For example, it is easy to check that if $F= \Z/2$ and $n\in \{1,\dots, 5\}$, then $(F\wr S_n, S_n)$ is a strong Gelfand pair.

It is not difficult to see that for any $S_n$ representation $W$  there is an isomorphism of $F\wr S_n$ representations, 
$\ind_{S_n}^{F\wr S_n} W \cong \bbc[F^n] \otimes W$; see Remark~\ref{R:Jantzen} for some further comments and the reference.
In particular, if $W$ is the trivial representation, and $F$ is an abelian group, 
then $\ind_{S_n}^{F\wr S_n} W$ is a multiplicity-free representation of $F\wr S_n$. 
In other words, $(F\wr S_n, S_n)$ is a Gelfand pair if $F$ is abelian.
From this we find the fact that $(F\wr S_n, \diag(F)\times S_n)$ is a Gelfand pair. 
Now we have two questions about $(F\wr S_n, \diag(F)\times S_n)$ here: 
1) What happens if we take a nonabelian group $F$? 
2) Is $(F\wr S_n, \diag(F)\times S_n)$ a strong Gelfand pair?
The answers of both of these questions are rather intriguing although both of them are negative. 
In~\cite{BensonRatcliff}, Benson and Ratcliff find a range where $(F\wr S_n, \diag(F)\times S_n)$
with $F$ nonabelian fails to be a Gelfand pair. 
In this paper, we show that, for $F= \Z/2$ and $n\geq 6$, $(F\wr S_n, \diag(F)\times S_n)$ is not a strong Gelfand pair (see Lemma~\ref{L:Case4}). 
Our proof can easily be adopted to the arbitrary finite abelian group case.

In Section~\ref{S:Some}, we prove that, for an arbitrary finite group $F$, 
a pair of the form $(F\wr S_n, F\wr (S_{n-k}\times S_k))$ is a strong Gelfand pair
if and only if $k\leq 2$.
In fact, by assuming that $F$ is an abelian group, we prove a stronger statement in one direction: 
$(F\wr S_n, (F\wr S_{n-k})\times S_k)$ is a strong Gelfand pair if $k\leq 2$.

Let $F$ and $K$ be two finite groups, and let $\pi_G: F\wr G \to G$ denote the canonical projection homomorphism.
If $K$ is a subgroup of $F\wr G$, then we will denote by $\gamma_K$ the image of $K$ under $\pi_G$. 
The purpose of our Section~\ref{S:Reduction} is to prove the following important reduction result (Theorem~\ref{T:reduction}): 
If $(F\wr G, K)$ is a strong Gelfand pair, then so is $(G,\gamma_K)$. 
As a consequence of this result, we observe that (Corollary~\ref{C:onlyfour}), for $n\geq 7$, if $(F\wr S_n, K)$ is a strong Gelfand pair, 
then $\gamma_K \in \{S_n, A_n, S_{n-1}\times S_1,S_{n-2}\times S_2\}$.
Moreover, we show a partial converse of Theorem~\ref{T:reduction}: 
Let $n\geq 7$, and let $B$ be a subgroup of $S_n$. 
Then $(F\wr S_n, F\wr B)$ is a strong Gelfand pair if and only if $(S_n,B)$ is a strong Gelfand pair. 
This is our Proposition~\ref{P:Fiff}.

In Sections~\ref{S:Hyperoctahedral} and~\ref{S:lattertwo}, we classify the strong Gelfand subgroups of the hyperoctahedral group, 
$B_n:= F \wr S_n$, where $F= \Z/2$. 
The hyperoctahedral group is a type BC Weyl group, and it contains the type D Weyl group, denoted by $D_n$, 
as a normal subgroup of index 2.  
Our list of strong Gelfand subgroups of $B_n$ is a culmination of a number of propositions. 
In Section~\ref{S:Hyperoctahedral} we handle the groups $K\leqslant B_n$ with $\gamma_K  \in \{ S_n, A_n\}$, 
and in Section~\ref{S:lattertwo}, we handle the case of $K\leqslant B_n$ with $\gamma_K  \in \{ S_{n-1}\times S_1, S_{n-2}\times S_2\}$.
An essential ingredient for our classification is the linear character group $L_n:=\Hom (B_n, \bbc^*)$, which is isomorphic to 
$\Z/2\times \Z/2 = \{ \id, \varepsilon, \delta, \varepsilon\delta\}$. 
Here, $\varepsilon$ and $\delta$ are defined so that $\ker \varepsilon = \Z/2 \wr A_n$, 
where $A_n$ is the alternating subgroup of $S_n$, and $\ker \delta = D_n$. 
The kernel of $\varepsilon\delta$ will be denoted by $H_n$.

Let $\chi$ be a linear character of a group $A$, and let $B$ be another group such that $\chi(A) \leq B$. 
We will denote by $(A\times B)_{\chi}$ the following diagonal subgroup of $A\times B$: 
\[
(A\times B)_{\chi}:=\{(a,\chi(a)) \in A\times B :\ a\in A\}.
\]

\begin{table}[htp]{\small
\begin{center}
\begin{tabular}{| c | c |}
\hline 
$\gamma_K$ & strong Gelfand subgroups of $B_n$  \\  \hline   
\multirow{3}{*} {$S_n$} & $B_n$ \\ & $D_n$ \\ & $H_n$  \\ \hline 
\multirow{2}{*} {$A_n$} & $\Z/2 \wr A_n$ \\ & $\ker \varepsilon \cap \ker \delta$, where $n \not\equiv 2 \mod 4$ \\ \hline
\multirow{9}{*} {$S_{n-1}\times S_1$} & 
	$B_{n-1} \times B_1$ \\ & 
	$B_{n-1} \times \{\id\}$ \\ & 
	$D_{n-1}\times B_1$ \\ &
	$D_{n-1} \times \{\id\}$ if $n$ is odd \\ &
	$H_{n-1}\times B_1$ \\ & 
	$H_{n-1} \times \{\id\}$ if $n$ is odd \\ &
	$(B_{n-1}\times B_1)_\delta = \{ (a,\delta(a)):\ a \in B_{n-1}\}$, if $n$ is odd \\&
	$(B_{n-1}\times B_1)_{\varepsilon \delta} = \{ (a, (\varepsilon\delta)(a)):\ a \in B_{n-1}\}$, if $n$ is odd \\& 
	$(B_{n-1}\times B_1)_\varepsilon = \{ (a,  \varepsilon(a)):\ a\in B_{n-1}\}$  \\&
	$(D_{n-1}\times B_1)_{\varepsilon \delta} = \{ (a, (\varepsilon\delta)(a)):\ a\in D_{n-1}\}$, if $n$ is odd \\&
	$(H_{n-1}\times B_1)_{\delta} = \{ (a, \delta(a)):\ a\in H_{n-1}\}$, if $n$ is odd\\
	\hline
\multirow{14}{*} {$S_{n-2}\times S_2$} &  
$B_{n-2}\times B_2$ \\&
$B_{n-2}\times D_2$ \\&
$B_{n-2}\times \overline{S_2}$
\\&
$B_{n-2}\times H_2$ \\&
$D_{n-2}\times D_2$ if $n$ is odd \\&
$D_{n-2}\times B_2$ if $n$ is odd \\& 
$H_{n-2}\times D_2$ if $n$ is odd \\&
$H_{n-2}\times B_2$ if $n$ is odd \\&
$D_{n-2}\times H_2$ if $n$ is odd \\&
$H_{n-2}\times H_2$ if $n$ is odd \\&
three non-direct product index 2 subgroups of $B_{n-2}\times D_2$ \\&
two non-direct product index 2 subgroups of $B_{n-2}\times H_2$ if $n$ is odd \\&
six non-direct product index 2 subgroups of $B_{n-2}\times B_2$ if $n$ is odd
\\ \hline	
\end{tabular}
\end{center}
\caption{The strong Gelfand subgroups of hyperoctahedral groups.}
\label{F:thelist}}
\end{table}

We will denote the natural copy of $S_n$ in $B_n$,
that is, $\{ (0,\sigma) \in F^n \times S_n:\ \sigma \in S_n\}$ by $\overline{S_n}$.
There is a unique subgroup $Z$ of $B_2$ that is conjugate to $\overline{S_2}$ and $Z\neq \overline{S_2}$.  
We will denote this copy of $S_2$ in $B_2$ by $\overline{S_2}'$. 
In Table~\ref{F:thelist}, 
we list all strong Gelfand subgroups of the hyperoctahedral group, $B_n$, 
collating the results of Propositions~\ref{P:Summary1}, \ref{P:Summary2}, \ref{P:Summary3}, and \ref{P:Summary4}. 
The index $n$ in this table is assumed to be at least 6 for the cases of $\gamma_K \in \{ S_n, A_n\}$, 
at least 7 for the case of $\gamma_K = S_{n-1}\times S_1$, and at least 8 for the case of $\gamma_K = S_{n-2}\times S_2$. 
In fact, there are some additional strong Gelfand subgroups for $n\leq 7$; in Propositions~\ref{P:Summary1} and \ref{P:Summary2} we list them explicitly for $n\leq 5$ and $\gamma_K \in \{ S_n, A_n\}$. 
In proving \cref{P:Summary3} we found two further strong Gelfand subgroups, which are included therein.
Though Table~\ref{F:thelist} does not provide an exhaustive list for $n\leq 7$, we remark that all subgroups in the table are still strong Gelfand if $n\leq 7$.
The strong Gelfand subgroups of $B_2$ are given in \cref{L:sgsB2}; those of $B_3$ are given in \cref{P:smallcases}, in which we also give the number of strong Gelfand subgroups of $B_n$ for each $4\leq n\leq 7$.

\begin{ack}
The first author is partially supported by a grant from the Louisiana Board of Regents. 
The authors thank Roman Avdeev for useful communications about the multiplicity-free representations of reductive algebraic groups, and thank Sin\'ead Lyle for advice on GAP.
We are also grateful for the support and resources provided by the Scientific Computing Section of the Research Support Division at OIST.
Finally, we thank an anonymous referee for providing data that enabled us to fix mistakes in a previous version.
\end{ack}

\section{Preliminaries}\label{S:Preliminaries}

We begin with setting up our conventions.

Throughout our manuscript, we will assume without further mention that our groups are finite. 
By a representation of a group we always mean a finite-dimensional complex representation.
The group-algebra of a group $G$ will be denoted by $\bbc[G]$. 
If $H$ is a subgroup of $G$, then we will write $H\leqslant G$.
The boldface $\mathbf{1}$ will always denote the one-dimensional vector space which 
is the trivial representation of every group. 
When we want to emphasize the group $G$ that acts trivially on $\mathbf{1}$, we will write $\mathbf{1}_G$.

We present some combinatorial notation that we will use in the sequel. 
For a positive integer $n$, a {\em partition of $n$} is a non-increasing sequence of 
positive integers $\lambda := (\lambda_1,\dots, \lambda_r)$ such that $\sum_{i=1}^r \lambda_i = n$. 
In this case, we will write $\lambda \vdash n$. 
Also, we will use the notation $|\lambda|$ for denoting the sum $\sum_{i=1}^r \lambda_i$.

\subsection{Semidirect products} 

Let $G$ be a group, and let $N$ and $H$ be two subgroups in $G$ such that 
\begin{enumerate}
\item $N$ is normal in $G$;
\item $G = HN$; 
\item $H\cap N = \{\id\}$.
\end{enumerate}
In this case, we say that $G$ is the {\em semidirect product of $N$ and $H$}, and we write $G = N\rtimes H$. 
Let $F$ be another group, and let $X$ be a $G$-set. 
The set of all functions from $X$ to $F$ is denoted by $F^X$. 
Since $F$ is a group, this set has the structure of a group with respect to point-wise multiplication.  
As $G$ acts on $X$, it also acts on $F^X$, hence, we can consider the group structure on the direct product $F^X \times G$
defined by 
\begin{align}\label{A:*}
(f,g)* (f',g') = (f g\cdot f',gg') \ \text{ for $(f,g),(f',g')\in F^X\times G$}.
\end{align}
In the sequel, when we think that it will not lead to a confusion, we will skip the multiplication sign $*$ from our notation.
The group $F^X \times G$ whose multiplication is defined in (\ref{A:*}) will be called the {\em wreath product of $F$ and $G$ with respect to $X$}; it will be denoted by $F\wr G$.

Next, we setup some conventions and terminology. 

\begin{enumerate}

\item Let $\id_G$ denote the identity element of $G$. 
The subgroup $F^X\times \{ \id_G\}$, denoted by $\overline{F^X}$, is called the {\em base subgroup} of $F\wr G$. 
In some places in the text, when confusion is unlikely, 
we will write $\overline{F}$ instead of $\overline{F^X}$.

\item The diagonal subgroup of the base group is isomorphic to $F$. 
By abusing the notation, we will denote this copy of $F$ in $F\wr G$ either by $\diag(F)$ or by $F$, 
depending on the context. 
Note that if $F$ is an abelian group, then $\diag(F)$ is a central subgroup in $F\wr G$.

\item Let $\id_F$ denote the identity element of $F$. 
Then the group $\overline{G}:=\{\id_F\}\times G$ is a subgroup of $F\wr G$ as well.
The diagonal copy of $F$ in $F\wr G$ intersects $\overline{G}$ trivially, therefore, $F\overline{G}$ is a subgroup
of $F\wr G$; it is isomorphic to $F\times G$ since both of the subgroups $F$ and $\overline{G}$ are normal subgroups in $F\overline{G}$. 
Some authors refer to $\overline{G}$ as the {\em passive factor} of $F\wr G$. 

\item If $F$ is the trivial group, then $F\wr G \cong G$. If $G$ is the trivial group, then $F\wr G \cong F^X$. 

\item If $G$ is a subgroup of $S_n$, then $F\wr G$ is defined with respect to the set $X:=\{1,\dots, n\}$. 
In particular, we have $F\wr S_1 = F$. 
We set as a convention that $F\wr S_0 = \{\id\}$.

\end{enumerate}

We finish this subsection by reviewing some simple properties of the wreath products.
The proofs of these facts can be found in~\cite[Proposition 2.1.3]{CST14}.

Let $H$ be a subgroup of $G$. Then we have 
\begin{align}\label{A:basicquotient}
(F\wr G)/ (F\wr H) \cong G/H.
\end{align}
Let $G_1$ and $G_2$ be two finite groups, and for $i\in \{1,2\}$, let $X_i$ be a $G_i$-set. 
To form the wreath product $F \wr (G_1\times G_2)$, we use the set $X:=  X_1 \sqcup X_2$ 
with the obvious action of $G_1\times G_2$. 
In this case, it is easy to check that there is an isomorphism of groups, 
$F\wr (G_1\times G_2) \cong (F\wr G_1) \times (F\wr G_2)$.
Let $G$ be a group such that $G_1\times G_2 \leqslant G$. 
In the sequel, we will be concerned with the subgroups of $F\wr G$ of the form $(F \wr G_1) \times \overline{G_2}$,
where the second factor $\overline{G_2}$ is the passive factor of $F\wr G_2$.

\subsection{Basic properties of induced representations.}\label{SS:Induction}

There are many equivalent ways of defining induced representation.
We provide a definition for completeness: 
If $H$ is a subgroup of $G$ and $V$ is a representation of $H$, 
then we view $V$ as a left $\bbc[H]$-module and $\bbc[G]$ as a $(\bbc[G],\bbc[H])$-bimodule. 
Thus we have the following $\bbc[G]$-module:
\begin{align}\label{A:defind}
\ind_H^G V := \bbc[G] \otimes_{\bbc[H]} V.
\end{align}
In loose terms, the $G$ representation afforded by $\ind_H^G V$ is called the {\em induced representation}. 
Notice that for $V= \mathbf{1}$, the r.h.s. of (\ref{A:defind}) is isomorphic to $\bbc[G/H]$.

The following fact, which is referred to as the {\em tensor identity} by some authors, 
is another useful fact that we will refer to later in the text.  
\begin{lem}\label{L:useful}
Let $V$ be a representation of $G$, and let $W$ be a representation of the subgroup $H$.
Then we have the following isomorphism of $G$ representations:
\[
V\otimes \ind_H^G (W) \cong \ind^G_H ( (\res^G_H V)\otimes W).
\]
\end{lem}
If we set $W = \mathbf{1}$ in Lemma~\ref{L:useful}, then we see that 
$V\otimes \bbc[G/H] \cong \ind^G_H \res^G_H V$.

Let $H$ be a proper subgroup of $G$, and let $V$ be an irreducible representation of $H$. 
It is well-known that the dimension of $\text{ind}_H^G V$ is equal to $[G:H]\dim V$, where $[G:H]$ is the index of $H$ in $G$. 
Consequently, for any two representations $V_1$ and $V_2$ of $H$, we have 
\[
\ind_H^G (V_1\otimes V_2) \ncong (\ind_H^G V_1) \otimes (\text{ind}_H^G V_2).
\]
Fortunately, there is a favorable situation where we have a similar factorization.

\begin{lem}\label{L:ellproduct}
Let $G_1,\dots, G_\ell$ be a list of finite groups.
For $i\in \{1,\dots, \ell\}$, let $H_i$ be a subgroup of $G_i$, and let  
$V_i$ be a representation of $H_i$.
Then we have 
\[
\text{ind}_{H_1\times \cdots \times H_\ell}^{G_1\times \cdots \times G_\ell} (V_1\boxtimes \cdots \boxtimes V_\ell) 
\cong (\text{ind}_{H_1}^{G_1} V_1 ) \boxtimes \cdots \boxtimes
(\text{ind}_{H_\ell}^{G_\ell} V_\ell).
\]
\end{lem}

\begin{proof}
Clearly, it suffices to prove our claim for $\ell=2$. 
But this case is proved in~\cite[Theorem 43.2]{CurtisReiner}.
\end{proof}

\subsection{Mackey Theory}\label{SS:Mackey}
 
In this subsection, we will mention some useful results of Mackey describing the relationship between 
the induced representations of two subgroups of $G$.
A good exposition of the main ideas of these results is given in~\cite{CST09}.

Let $H,K\leq G$ be two subgroups with a system $S$ of representatives for the double $(H,K)$-cosets in $G$.
Let $(\sigma,V)$ and $(\rho,W)$ be representations of $H$ and $K$, respectively. 
For $s$ in $S$, let $G_s$ denote $H\cap sKs^{-1}$, and let $W_s$ denote the representation $\rho_s :G_s \to GL(W)$ by setting 
$\rho_s(g) w = \rho (s^{-1}gs) w$ for all $g\in G_s$, and $w\in W$. 
{\em Mackey's formula} states that 
\begin{align}\label{A:Mackey}
\Hom_G( \ind_H^G V , \ind_K^G W ) = \bigoplus_{s\in S} \Hom_{G_s} ( \res_{G_s}^H V ,W_s).
\end{align}  
In this notation, a closely related fact, which is called {\em Mackey's lemma}, states that 
\begin{align}\label{A:MackeyLemma}
\res_{H}^G \ind_K^G W = \bigoplus_{s\in S} \ind^H_{G_s} W_s. 
\end{align}

\subsection{Generalized Johnson schemes}\label{SS:Johnson}

Let $h$ be an element of the set $\{0,\dots, n\}$, and let $\mathbf{1}$ denote the trivial 
representation of the Young subgroup $S_{n-h}\times S_h$ of the symmetric group $S_n$. 
It is well-known that $(S_n, S_h\times S_{n-h})$ is a Gelfand pair.
Indeed, by using Pieri's rule~\cite[Corollary 3.5.14]{CST10}, it is easy to see that 
\begin{align}\label{A:Johnson}
\text{ind}_{S_h\times S_{n-h}}^{S_n} \mathbf{1}= \oplus_{j=0}^h S^{(n-j,j)},
\end{align}
where $S^{(n-j,j)}$ is the Specht module indexed by the partition $(n-h,h)$.

\begin{defn}
Let $(F,H)$ be Gelfand pair. Let $n$ be a positive integer, and let $h$ be an element of the set $\{0,\dots, n\}$.
A pair of the form $(F\wr S_n, H\wr S_h \times F\wr S_{n-h})$ is called a {\em generalized Johnson scheme}. 
\end{defn}

According to~\cite[Theorem 3.2.19 (ii)]{CST14}, every generalized Johnson scheme is a Gelfand pair. 
Evidently, if $F=H$, then $(F,H)$ is a Gelfand pair, hence, we have 
\[
H\wr S_h \times F\wr S_{n-h} = F\wr S_h \times F\wr S_{n-h} = F \wr (S_h\times S_{n-h}).
\] 
In particular, the pair $(F \wr S_n, F \wr (S_h\times S_{n-h}))$ is a Gelfand pair.

\subsection{Characterizations of the Gelfand property.}

We can take quotients by normal subgroups and preserve the Gelfand property.
\begin{lem}\label{L:iff1}
Let $N$ and $H$ be two subgroups of $G$ such that $N$ is normal in $G$ and $N\leqslant H$. 
Then we have $(G,H)$ is a Gelfand pair if and only if $(G/N,H/N)$ is a Gelfand pair. 
\end{lem}

\begin{Remark}\label{R:iff1}
Lemma~\ref{L:iff1} can be restated as follows: Let $\varphi : G \to G'$ be a homomorphism such that $\ker \varphi \leqslant H$.
Then $(G,H)$ is a Gelfand pair if and only if $(\varphi(G),\varphi(H))$ is a Gelfand pair.
\end{Remark}

\begin{cor}\label{C:iff}
Let $G$ and $F$ be two finite groups for which we can define $F\wr G$. 
Let $H$ be a subgroup of $G$. 
Then $(G,H)$ is a Gelfand pair if and only if $(F\wr G, F\wr H)$ is a Gelfand pair.
\end{cor}

\begin{proof}
Let $X$ be a $G$-set such that the semidirect product of $F^X\rtimes G$ is the wreath product $F\wr G$, and $F^X \rtimes H$ is the wreath product $F\wr H$. 
Since $F^X$ is a normal subgroup of $F\wr G$, the proof follows from Lemma~\ref{L:iff1}. 
\end{proof}

In the spirit of Corollary~\ref{C:iff}, we can fix the second factor and choose a Gelfand pair in the first factor. 
\begin{lem}{\cite[Theorem 3.3.18]{CST14}}\label{L:3.3.18}
Let $(F,H)$ be a finite Gelfand pair, and let $G$ be a finite group. 
Then $(F\wr G, H\wr G)$ is a Gelfand pair. 
\end{lem}

If $F$ is an abelian group, then $(F,\{\id\})$ is a Gelfand pair. 
Hence, Lemma~\ref{L:3.3.18} implies that $(F\wr G, G)$ is a Gelfand pair for every finite abelian group $F$,
and for every finite group $G$.

\subsection{A brief review of representations of $F\wr S_n$}\label{SS:Repsofwreaths}

The purpose of this section is to review the construction of the irreducible representations of $F\wr S_n$, 
where $F$ is a finite group.
We loosely follow James and Kerber's book~\cite[Section 4.4]{JamesKerber}.

Let $W_1,\dots, W_r$ be the complete list of pairwise inequivalent and irreducible representations of $F$.
If $n$ denotes a positive integer, then every irreducible representation $D^*$ of $F^n$ is given by an outer tensor product of the form 
\[
D^*:=D_1\boxtimes \cdots \boxtimes D_n, 
\]
where $D_i \in \{W_1,\dots, W_r\}$. 
For $j\in \{1,\dots, r\}$, let $n_j$ denote the number of factors of $D^*$ that are isomorphic to $D_j$. 
Of course, some of these numbers might be equal to zero, nevertheless, the terms of the sequence $\mathbf{n}:= (n_1,\dots, n_r)$ sum to $n$. 
We will call such a sequence of nonnegative integers a {\em composition of $n$}. 
The composition $\mathbf{n}$ will be called the {\em type of $D^*$}.
Since $S_n$ permutes the factors of $F^n$, two irreducible representations of $F^n$ are $S_n$-conjugate 
if and only if they have the same type. 
An important group theoretic invariant of the irreducible representation $D^*$, called the {\em inertia group of $D^*$}, 
is given by 
\begin{align}\label{D:inertia}
F\wr S(\mathbf{n}) := F\wr (S_{n_1} \times \cdots \times S_{n_r}) = F\wr S_{n_1} \times \cdots \times F\wr S_{n_r}.
\end{align}

For every irreducible representation $D$ of $F$, we have a representation of type $\mathbf{n} = (n)$ of $F\wr S_n$,
which is denoted by $D^{(n)}$, and defined as follows: The underlying vector space of $D^{(n)}$ is $V=D^{\boxtimes n}$.
If $v:= v_1\otimes \cdots \otimes v_n$ is a basis element for $V$, 
then the action of an element $((f_1,\dots, f_n),\pi)$ of $F^n\wr S_n$ on $v$ is given by 
\[
((f_1,\dots, f_n),\pi)\cdot v_1\otimes \cdots \otimes v_n := (f_1\cdot v_{\pi(1)} \otimes \cdots \otimes f_n \cdot v_{\pi(n)}).
\]
At the same time, for every irreducible representation $D''$ of $S_n$, we have a corresponding 
irreducible representation of $F\wr S_n$.
It is defined as follows: If $(f,\pi)$ is an element of $F\wr S_n$ and $v$ is a vector from $D''$, then the action of $(f,\pi)$ on $v$ 
is given by   
\begin{align}\label{A:inflationfromSn}
(f,\pi) \cdot v := \pi\cdot v.
\end{align}
\begin{Remark}\label{R:tobeused}
Another name for the representation that is defined in (\ref{A:inflationfromSn}) is {\em inflation}.
More generally, if $N$ is a normal subgroup of a finite group $M$ and $\rho: M/N\to GL(V)$ is a representation,
then we have an associated representation $\bar{\rho}: M\to GL(V)$, $\bar{\rho}(g) := \rho(gN)$, 
which is called the {\em inflation of $\rho$}.
Since the canonical quotient map $M\to M/N$ is a surjective homomorphism, if $\rho$ is irreducible, then so is its inflation. 
\end{Remark}

We now consider the {\em inner tensor product of $D^{(n)}$ and $D''$}, that is, 
\begin{align}\label{A:asin}
(D; D'') :=D^{(n)} \otimes D'',
\end{align}
on which $F\wr S_n$ acts diagonally. 
In general, the inner tensor products of irreducible representations are reducible, however, 
(\ref{A:asin}) is an irreducible representation for $F\wr S_n$. 
Indeed, according to Specht (but see~\cite[Theorem 4.4.3]{JamesKerber})
the complete list of pairwise inequivalent and irreducible representations of 
$F\wr S_n$ is comprised of representations of the form
\begin{align}\label{A:apply}
\ind_{F\wr S(\mathbf{n})}^{F\wr S_n} (D_1;D_1'')\boxtimes \cdots \boxtimes (D_r; D_r''),
\end{align}
where $\mathbf{n}=(n_1,\dots, n_r)$ is a composition of $n$, $D_i$'s ($1\leq i \leq r$) are pairwise
inequivalent irreducible representations of $F$, 
and $D_i''$ ($1\leq i \leq r$) is an irreducible representation of $S_{n_i}$.

\section{A Useful Lemma and Induction From the Passive Factor}\label{S:PassiveFactor}

The main goal of this section is to prove a technical but a useful lemma that we will use repeatedly in the sequel.
Also, we will show that in general the pair $(F\wr S_n, S_n)$ need not be a strong Gelfand pair.

We begin with a lemma which we will use several times in the sequel. 
To keep its statement simple, we introduce some of the notation of its hypothesis here:
$F$ will denote a group, and $D$ will be an irreducible representation of $F$. 
For two nonnegative integers $n$ and $k$ such that $0\leq k \leq n$, we will denote by 
$E$ (resp.~$E'$) an irreducible representation of $S_{n-k}$ (resp.~$S_k$). 
By $U$, we will denote the $S_n$ representation $U:=\ind^{S_n}_{S_{n-k}\times S_k} E\boxtimes E'$.
We assume that the decomposition of $U$ into irreducible $S_n$ representations is given by 
\[
U \cong {m_1} E_1\oplus \cdots \oplus {m_r} E_r.
\]

\begin{lem}\label{L:nail}
We maintain the notation from the previous paragraph. 
If $A$ is the $F\wr S_n$ representation defined by 
$\ind^{ F\wr S_n}_{F\wr S_{n-k}  \times  F\wr S_k}  (D; E) \boxtimes (D; E')$, 
then its decomposition into irreducible subrepresentations is given by 
\begin{align}\label{A:A01}
A \cong   \bigoplus_{i=1}^r m_i  (D;E_i).
\end{align}
In particular, if $U$ is a multiplicity-free $S_n$ representation, then $A$ is a multiplicity-free $F\wr S_n$ representation. 
\end{lem}

\begin{proof} 
By its definition, the induced representation 
$\ind^{ F\wr S_n}_{F\wr S_{n-k}  \times  F\wr S_k}  (D; E) \boxtimes (D; E') $ is given by 
\begin{align}\label{A:bydef}
 \ind^{ F\wr S_n}_{F\wr S_{n-k}  \times  F\wr S_k}  (D; E) \boxtimes (D; E') 
 = \bbc[ F\wr S_n] \otimes_{\bbc[ F\wr S_{n-k}  \times  F\wr S_k]}( (D; E) \boxtimes (D; E') ).
\end{align}
The outer tensor product $ (D; E) \boxtimes (D; E')$, 
as a representation of $F\wr S_{n-k}  \times  F\wr S_k$, can be rewritten as an inner tensor product 
of representations of $F\wr (S_{n-k}\times S_k)$ as follows: 
\begin{align}\label{A:DEDE'}
(D; E) \boxtimes (D; E') &=   (D^{(n-k)}\boxtimes D^{(k)}) \otimes (E \boxtimes E') = D^{(n)} \otimes (E\boxtimes E'),
\end{align}
where $F\wr (S_{n-k}\times S_{k})$ acts on $D^{(n)}$, as a subgroup of $F\wr S_n$, by $F^n$;
it acts on $E \boxtimes E'$ by $S_{n-k}\times S_{k}$. 
By substituting (\ref{A:DEDE'}) in (\ref{A:bydef}), we see that
\begin{align}\label{A:bydef2}
\bbc[ F\wr S_n] \otimes_{\bbc[ F\wr S_{n-k}  \times  F\wr S_k]}   D^{(n)} \otimes (E\boxtimes E')
= \ind^{F\wr S_n}_{F\wr (S_{n-k}\times S_k)} D^{(n)} \otimes (E\boxtimes E').
\end{align}
Since $F\wr (S_{n-k}\times S_k)$ acts on $D^{(n)}$ as a subgroup of $F\wr S_n$, by definition, 
$D^{(n)}$ is the restricted representation $\res^{F\wr S_n}_{F\wr (S_{n-k}\times S_k)} D^{(n)}$. 
Thus, by Lemma~\ref{L:useful}, we have 
\begin{align}\label{A:bydef3}
\ind^{F\wr S_n}_{F\wr (S_{n-k}\times S_k)} D^{(n)} \otimes (E\boxtimes E') 
&=  \ind^{F\wr S_n}_{F\wr (S_{n-k}\times S_k)} ((\res^{F\wr S_n}_{F\wr (S_{n-k}\times S_k)} D^{(n)}) \otimes (E\boxtimes E')) \notag  \\
&\cong D^{(n)} \otimes \ind^{F\wr S_n}_{F\wr (S_{n-k}\times S_k)} E\boxtimes E'.
\end{align}
Since $F^n$ acts trivially on $E\boxtimes E'$, 
we know that $\ind^{F\wr S_n}_{F\wr (S_{n-k}\times S_k)} E\boxtimes E' = {m_1}E_1 \oplus \cdots \oplus m_r E_r $.
Therefore, (\ref{A:bydef3}) is given by 
\begin{align*}
D^{(n)} \otimes ( {m_1}E_1 \oplus \cdots \oplus m_r E_r ) &= 
{m_1} D^{(n)}\otimes E_1 \oplus \cdots \oplus m_r D^{(n)}\otimes E_r   \\ 
&= 
{m_1} (D; E_1) \oplus \cdots \oplus m_r(D; E_r). 
\end{align*}
This finishes the proof of our first claim. 
Our second assertion follows from (\ref{A:A01}) by setting all of the $m_i$'s ($1\leq i \leq r$) to 1. 
This finishes the proof of our lemma.
\end{proof}

Let us present a straightforward consequence of Lemma~\ref{L:nail}.
We will use the following notation: 
Let $D$ be an irreducible representation of $F$.
Let $n$ be a positive integer, and let $\mathbf{b}=(b_1,\dots, b_r)$ be a composition of $n$. 
For $i\in \{1,\dots, r\}$, let $E_i$ be an irreducible representation of $S_{b_i}$,
and let $U$ denote the $S_n$ representation $U:=\ind^{S_n}_{S(\mathbf{b})} (\boxtimes_{i=1}^r E_i)$
whose decomposition into irreducible constituents is given by 
\[
U = \bigoplus_{\lambda \vdash n } m_\lambda  S^\lambda , 
\]
where $S^\lambda$ is the Specht module indexed by the partition $\lambda$, and $m_\lambda \in \Z_{\geq 0}$
is the multiplicity of $S^\lambda$ in $U$.

\begin{cor}\label{C:nail2}
We maintain the notation of the previous paragraph. 
If $A$ is the induced representation $\ind^{ F\wr S_n}_{F\wr S(\mathbf{b})}  (D; E_1) \boxtimes\cdots
\boxtimes (D; E_r)$, 
then its decomposition into irreducible constituents is given by 
\begin{align}\label{A:Pnail2}
A=\bigoplus_{i=1}^r m_\lambda (D;S^\lambda).
\end{align}
In particular, if $U$ is a multiplicity-free $S_n$ representation, then $A$ is a multiplicity-free $F\wr S_n$ representation. 
\end{cor}

\begin{proof}
We apply induction on $r$. The base cases $r=2$ is already proven in Lemma~\ref{L:nail}.
The general case follows from $r-1$ by distributivity of the tensor products over direct sums. 
\end{proof}

\begin{Remark}\label{R:Jantzen}
 
Let $\mathbf{b}$ be the composition $(1,\dots, 1)$ of $n$. 
In this case, we have 
\[
F\wr S(\mathbf{b})=F\wr (S_1\times \cdots \times S_1) = F^n.
\]
Let us denote $S_1\times \cdots \times S_1$ by $\prod^n S_1$. 
Clearly, $\prod^n S_1$ is the trivial subgroup of $S_n$.  
Let $U$ denote the representation $\ind^{S_n}_{\prod^n S_1} \mathbf{1}  \boxtimes\cdots \boxtimes \mathbf{1}=  
\ind^{S_n}_{\{\id\}} \mathbf{1} \cong \bbc[S_n]$.
Hence, every irreducible representation $S^\lambda$ of $S_n$ appears in $U$ with multiplicity 
$m_\lambda = \dim S^\lambda$. 
Corollary~\ref{C:nail2} shows that 
\[
\ind^{ F\wr S_n}_{F\wr \prod^n S_1}  (D;\mathbf{1}) \boxtimes\cdots \boxtimes (D;\mathbf{1})=\bigoplus_{\lambda \vdash n} \dim S^\lambda (D;S^\lambda)
=(D;\bigoplus_{\lambda \vdash n} (\dim S^\lambda) S^\lambda)
 \cong D^{(n)} \otimes \bbc[ S_n]. 
\]
This observation is a special case of a more general, well-known isomorphism that is 
presented in Jantzen's textbook~\cite[Section 3.8]{Jantzen}.
In our special case, it implies that, for any $F^n$ representation $N$, there is an isomorphism of $F\wr S_n$ representations:
\begin{align}\label{A:fromJantzen1}
\ind_{F\wr \prod^n S_1}^{F\wr S_n} N \cong \bbc [S_n]\otimes N. 
\end{align}
We also know from~\cite[Section 3.8]{Jantzen} that, for any $S_n$ representation $W$, there is an isomorphism of $F\wr S_n$ representations:
\begin{align}\label{A:fromJantzen2}
\ind_{S_n}^{F\wr S_n} W \cong \bbc[F^n] \otimes W.
\end{align}
Here, $F^n$ acts on $\bbc[F^n]$ via its left regular representation, $S_n$ acts on $W$ the usual way, 
and it acts on $\bbc[F^n]$ by permuting the factors of $F^n$.  
\end{Remark}

Although the isomorphism in (\ref{A:fromJantzen2}) provides us with 
the general structure of the induced representation $\ind_{S_n}^{F\wr S_n} W$, 
we still want to determine the multiplicities of the irreducible representations in it.  
We resolve this problem by our next theorem.

\begin{thm}\label{T:firststep}
Let $F$ be an abelian group, and let $U$ be an irreducible representation of $F\wr S_n$ of the form 
$U:=\ind_{F \wr S({\mathbf{a}})}^{F\wr S_n} (D_1;D_1'')\boxtimes \cdots \boxtimes (D_s; D_s'')$,
where $D_1,\dots, D_s$ are some pairwise inequivalent irreducible representations of $F$, 
and $D_1'',\dots, D_s''$ are some irreducible representations of $S_{a_1},\dots, S_{a_s}$, respectively.

Under these assumptions, 
if $W$ is an irreducible representation of $S_n$, then the multiplicity of $U$ in $\ind^{F\wr S_n}_{S_n} W$
is equal to the multiplicity of $W$ in $\ind^{S_n}_{S(\mathbf{a})} D_1'' \boxtimes \cdots \boxtimes D_s''$.
\end{thm}

\begin{proof}
Since we will work with a fixed number $n$, and a fixed abelian group $F$, 
to ease our notation, let us set $G:= F\wr S_n$ and $K:=S_n$. 

The multiplicity of $U$ in $\ind_K^G W$ is equal to the dimension of the vector space 
\[
M:= \Hom_G ( U , \ind_K^G W). 
\]
We will use Mackey's formula and Frobenius reciprocity to compute the dimension of $M$.
For brevity, we will denote the inertia group of $V:= (D_1;D_1'')\boxtimes \cdots \boxtimes (D_s; D_s'')$, 
that is $F\wr S(\mathbf{a})$, by $H$.

Let $S$ be a system of representatives for the $(H,K)$-double cosets in $G$. 
Since $\overline{F}$ is a normal subgroup of $G$, and since it is contained in $H$, we see that $HK=G$.
In other words, $S$ has only one element, $S= \{\id\}$. 
Therefore, there is only one local group of the form $G_s= H\cap s K s^{-1}$,
which is given by $G_{\id} = F\wr S(\mathbf{a}) \cap S_n = S(\mathbf{a})$.

Now we apply Mackey's formula~(\ref{A:Mackey}):
\begin{align}\label{A:toapplyMackey}
M&= \Hom_G (\ind_{H}^G V, \ind_K^G W)  \notag \\
&=  \bigoplus_{s\in S} \Hom_{G_s} (\res^H_{G_s} V, W_s) \notag \\
&= \Hom_{S(\mathbf{a})} (\res^{H}_{S(\mathbf{a})} V, W_{\id}), 
\end{align}
where $W_{\id}$ is the copy of $W$ viewed as a representation of $S(\mathbf{a})$,
that is, $W_{\id} = \res^K_{S(\mathbf{a})} W$.
In our abelian case, $D_1,\dots, D_s$ are one-dimensional representations of $F$,
so, the dimension of each factor $(D_i, D_i'')$ of $V$ is equal to the dimension of $D_i''$. 
In particular, the factor $(D_i,D_i'')$ can be identified, as a representation of $S_{a_i}$, with $D_i''$. 
Therefore, $\res^{H}_{S(\mathbf{a})} V$ is equivalent to $D_1'' \boxtimes \cdots \boxtimes  D_s''$. 
Thus (\ref{A:toapplyMackey}) is equivalent to 
\[
M = \Hom_{S(\mathbf{a})} ( D_1'' \boxtimes \cdots \boxtimes  D_s'',   \res^K_{S(\mathbf{a})} W).
\]
This shows that $\dim M$ is given by the multiplicity of $D_1'' \boxtimes \cdots \boxtimes  D_s''$ in 
$ \res^K_{S(\mathbf{a})} W$.
By applying Frobenius reciprocity, we see that 
\[
\dim M = \text{the multiplicity of $W$ in $\ind^{S_n}_{S(\mathbf{a})} D_1'' \boxtimes \cdots \boxtimes  D_s''$}.
\]
This finishes the proof of our theorem. 
\end{proof}

\begin{cor}\label{C:firststep}
Let $B$ be a subgroup of $F\wr S_n$. 
If $n \leq 5$ and $\overline{S_n} \leqslant B$, then the pair $(F\wr S_n, B)$ is a strong Gelfand pair.
If $6\leq n$ and $B\leqslant \overline{S_n}$, then the pair $(F\wr S_n, B)$ is not a strong Gelfand pair.
\end{cor}
\begin{proof}
It suffices to show that $(F\wr S_n,\overline{S_n})$ is a strong Gelfand pair if and only if $n\leq 5$.
In particular, we can now use Theorem~\ref{T:firststep}.

Let $\mathbf{a}$ be a composition of $n$, and let $S(\mathbf{a})$ denote the corresponding subgroup
$S_{a_1}\times \cdots \times S_{a_s}$ of $S_n$. Let $D_1''\boxtimes \cdots \boxtimes D_s''$ 
be an irreducible representation of $S(\mathbf{a})$.
Clearly, if $n\leq 5$, then at most one of the factors $D_i''$ ($1\leq i \leq s$) is of the form
$S^{(1^a)}$ or $S^{(a)}$. 
Then by Pieri's formula, 
$\ind^{S_n}_{S(\mathbf{a})} D_1''\boxtimes \cdots \boxtimes D_s''$ is a multiplicity-free representation. 
For $n\geq 6$, we can use the induced representation $\ind_{S_{n-3}\times S_3}^{S_n} S^{(n-4,1)}\boxtimes S^{(2,1)}$ and note that $S^{(n-3,2,1)}$ appears as a summand with multiplicity 2 -- this is an easy check on the number of Littlewood--Richardson tableaux of skew-shape $(n-3,2,1) / (n-4,1)$ and weight $(2,1)$.
This completes the proof. 
\end{proof}

\section{Some Strong Gelfand Subgroups of Wreath Products}\label{S:Some}

In this section we will prove that, for an arbitrary group $F$, 
a pair of the form $(F\wr S_n, F\wr (S_{n-k}\times S_k))$ is a strong Gelfand pair
if and only if $k\leq 2$.
Furthermore, we will prove that, for an abelian group $F$, 
$(F\wr S_n, (F\wr S_{n-k})\times S_k$ is a strong Gelfand pair if $k\leq 2$. 

\subsection{The nonabelian base group case.}\label{SS:Somenonabelian}

\begin{thm}\label{T:nonabelian}
Let $F$ be a group, and let $n\geq 2$. 
If $k$ is 1 or 2, then the pair $(F\wr S_n, F\wr (S_{n-k}\times S_k))$ is a strong Gelfand pair. 
\end{thm}

\begin{proof}
Let $K$ denote $F\wr (S_{n-k}\times S_k)$. 
Since $K \cong F\wr S_{n-k} \times F \wr S_k$, every irreducible representation of $K$ is of the form $E\boxtimes D$, where 
$E$ is an irreducible representation of $F\wr S_{n-k}$ and $D$ is an irreducible representation of $F\wr S_k$.
Then there exists a composition $\mathbf{c}=(c_1,\dots,c_s)$ of $k$ such that 
\[
D  = \ind_{F \wr S({\mathbf{c}})}^{F\wr S_{k}} (D_1;D_1'')\boxtimes \cdots \boxtimes (D_s; D_s''),
\] 
where $D_1,\dots, D_s$ are some pairwise inequivalent irreducible representations of $F$, 
and $D_1'',\dots, D_s''$ are some irreducible representations of $S_{c_1},\dots, S_{c_s}$, respectively. 
Similarly for $E$, let $\mathbf{b}$ denote the composition of $n-k$ such that  
\[
E  = \ind_{F \wr S({\mathbf{b}})}^{F\wr S_{n-k}} (E_1;E_1'')\boxtimes \cdots \boxtimes (E_r; E_r''),
\] 
where $E_1,\dots, E_r$ are some pairwise inequivalent irreducible representations of $F$, 
and $E_1'',\dots, E_r''$ are some irreducible representations of $S_{b_1},\dots, S_{b_r}$, respectively. 
Of course, if $k=2$, then we have only two possibilities for $\mathbf{c}$; they are given by $\mathbf{c} \in\{ (1,1), (2)\}$. 
If $k=1$, then $\mathbf{c}\in \{ (1)\}$. 
In any case, to ease our notation, let us denote $(D_1;D_1'')\boxtimes \cdots \boxtimes (D_s; D_s'')$ by $D_0$, 
and let us denote $(E_1;E_1'')\boxtimes \cdots \boxtimes (E_r; E_r'')$ by $E_0$.

As far as the irreducible representations $E_1,\dots, E_r,D_1,\dots, D_s$ are concerned, we have two possibilities:
\begin{enumerate}
\item $\{ E_1,\dots, E_r \} \cap \{ D_1,\dots, D_s \} = \emptyset$, or 
\item $\{ E_1,\dots, E_r \} \cap \{ D_1,\dots, D_s \} \neq \emptyset$.
\end{enumerate}
We proceed with the first case. 
In this case, by Lemma~\ref{L:ellproduct}, we have 
\begin{align*}
\ind^{F\wr S_n}_{F\wr (S_{n-k}\times S_k)} (E\boxtimes D) &=
\ind^{F\wr S_n}_{F\wr (S_{n-k}\times S_k)} (\ind^{F\wr S_{n-k}}_{F\wr S(\mathbf{b})} E_0 )\boxtimes (\ind^{F\wr S_k}_{F\wr S(\mathbf{c})} D_0 ) \\
&=
\ind^{F\wr S_n}_{F\wr (S_{n-k}\times S_k)} (\ind^{F\wr S_{n-k} \times F\wr S_k}_{F\wr S(\mathbf{b})\times F\wr S(\mathbf{c})} E_0 \boxtimes D_0) \\
&=
\ind^{F\wr S_n}_{F\wr S(\mathbf{b})\times F\wr S(\mathbf{c})} E_0 \boxtimes  D_0 \\
&
=\ind^{F\wr S_n}_{F\wr S(\mathbf{b} \mathbf{c})} E_0 \boxtimes  D_0,
\end{align*}
where $\mathbf{b} \mathbf{c}$ is the composition obtained by concatanating $\mathbf{b}$ and $\mathbf{c}$.
It is easy to see that the representation 
$\ind^{F\wr S_n}_{F\wr S(\mathbf{b} \mathbf{c})} E_0 \boxtimes  D_0$ is one of the irreducible representations of $F\wr S_n$ as in (\ref{A:apply}).

Next, we will handle the second case. 
Without loss of generality let us assume that $D_1 = E_r$. 
By arguing as before, first, we write  
\begin{align*}
\ind^{F\wr S_n}_{F\wr (S_{n-k}\times S_k)} (E\boxtimes D)
&=
\ind^{F\wr S_n}_{F\wr S(\mathbf{b})\times F\wr S(\mathbf{c})} E_0 \boxtimes  D_0.
\end{align*}
Notice that we have 
\begin{align*}
{F\wr S(\mathbf{b})\times F\wr S(\mathbf{c})} =
F\wr S(\mathbf{b}') \times F\wr S_{b_{r}}  \times  F\wr S_{c_1}\times F\wr S(\mathbf{c}'),
\end{align*}
where 
$F\wr S(\mathbf{b}') = F\wr (S_{b_1}\times \cdots \times S_{b_{r-1}})$ and 
$F\wr S(\mathbf{c}') = F\wr (S_{c_2}\times \cdots \times S_{c_s})$.
Accordingly we write 
\begin{align*}
E\boxtimes D & = E'_0\boxtimes L \boxtimes D'_0,
\end{align*}
where $E'_0 := (E_1;E_1'')\boxtimes \cdots \boxtimes (E_{r-1}; E_{r-1}'')$, 
$D'_0 := (D_2;D_2'')\boxtimes \cdots \boxtimes (D_s; D_s'')$, and 
\begin{align*}
L:= (E_{r}; E_{r}'')\boxtimes  (D_1;D_1'') = E_r^{(b_r)}\otimes E_r'' \boxtimes D_1^{(c_1)}\otimes D_1''= D_1^{(b_r)}\otimes E_r'' \boxtimes D_1^{(c_1)}\otimes D_1''.
\end{align*}
Now, by using by Lemma~\ref{L:ellproduct} and transitivity of the induction, we split our induction:
\begin{align*}
\ind^{F\wr S_n}_{F\wr S(\mathbf{b})\times F\wr S(\mathbf{c})} E_0 \boxtimes  D_0
&= \ind^{F\wr S_n}_{F\wr S(\mathbf{b'})\times F\wr S_{b_r+c_1} \times F\wr S(\mathbf{c'})} \ind^{F\wr S(\mathbf{b'})\times F\wr S_{b_r+c_1} \times F\wr S(\mathbf{c'})}_{F\wr S(\mathbf{b}') \times F\wr S_{b_{r}}  \times  F\wr S_{c_1}\times F\wr S(\mathbf{c}')} E_0 \boxtimes  D_0\\
&= \ind^{F\wr S_n}_{F\wr S(\mathbf{b'})\times F\wr S_{b_r+c_1} \times F\wr S(\mathbf{c'})} \ind^{F\wr S(\mathbf{b'})\times F\wr S_{b_r+c_1} \times F\wr S(\mathbf{c'})}_{F\wr S(\mathbf{b}') \times F\wr S_{b_{r}}  \times  F\wr S_{c_1}\times F\wr S(\mathbf{c}')} E_0' \boxtimes L \boxtimes  D_0' \\
&= \ind^{F\wr S_n}_{F\wr S(\mathbf{b'})\times F\wr S_{b_r+c_1} \times F\wr S(\mathbf{c'})} 
E_0' \boxtimes \left(\ind^{ F\wr S_{b_r+c_1}}_{F\wr S_{b_{r}}  \times  F\wr S_{c_1}}  L \right) \boxtimes  D_0'.
\end{align*}

We continue with the computation of the middle term, 
\begin{align}\label{A:AA}
A:= \ind^{ F\wr S_{b_r+c_1}}_{F\wr S_{b_{r}}  \times  F\wr S_{c_1}}  L  &= 
\ind^{ F\wr S_{b_r+c_1}}_{F\wr S_{b_{r}}  \times  F\wr S_{c_1}}   (D_1^{(b_r)}\otimes E_r') \boxtimes (D_1^{(c_1)}\otimes D_1').
\end{align}
Notice here that we can apply Lemma~\ref{L:nail}. 
Indeed, since $c_1\in \{1,2\}$, $D_1'$ is either a sign representation or the trivial representation of $S_{c_1}$, therefore, 
by Pieri's formula, the induced representation $\ind_{S_{b_r}\times S_{c_1}}^{S_{b_r+c_1}} (E_r' \boxtimes D_1')$ 
is a multiplicity-free representation of ${S_{b_r+c_1}}$. Let $L_1,\dots, L_l$ denote its irreducible constituents.
Then, by Lemma~\ref{L:nail}, 
(\ref{A:AA}) is equivalent to the $F\wr S_{b_r+c_1}$ representation $(D_1;L_1)\oplus \cdots \oplus (D_1;L_l)$,
hence we have 
\begin{align}
\ind^{F\wr S_n}_{F\wr S(\mathbf{b})\times F\wr S(\mathbf{c})} E_0 \boxtimes  D_0
&= \ind^{F\wr S_n}_{F\wr S(\mathbf{b'})\times F\wr S_{b_r+c_1} \times F\wr S(\mathbf{c'})} 
E_0' \boxtimes \left( (D_1;L_1)\oplus \cdots \oplus (D_1;L_l) \right) \boxtimes  D_0'  \notag \\
&= \bigoplus_{i=1}^l \ind^{F\wr S_n}_{F\wr S(\mathbf{b'})\times F\wr S_{b_r+c_1} \times F\wr S(\mathbf{c'})} 
E_0' \boxtimes (D_1;L_i) \boxtimes  D_0'. \label{A:eachsummand}
\end{align}
Evidently, (\ref{A:eachsummand}) is a multiplicity-free representation of $F\wr S_n$ if there is no other term 
$(D_i;D_i'')$ in $D_0$ for $i>1$ such that $D_i\in \{E_1,\dots, E_r\}$. 
In fact, even if there is another such summand, since $\mathbf{c}$ has at most two parts, 
we can apply the same procedure to our decomposition (\ref{A:eachsummand}) by the second (inequivalent) irreducible representation
$D_2$. Since the list of irreducible representations of $F$ that appear in the final direct sum would all be distinct, in this case also, we get a multiplicity-free representation of $F\wr S_n$. 
This finishes the proof of our theorem. 
\end{proof}

\subsection{Abelian base groups.}\label{SS:Somenabelian}

In this subsection, $F$ denotes an abelian group. 
Also, by a slight abuse of notation, the subgroup $(F\wr S_{n-1})\times S_1 \leqslant F\wr S_n$,
where $S_1$ corresponds to the (trivial) subgroup of $F\wr S_1$, will be
denoted by $F\wr S_{n-1}$.

\begin{prop}\label{P:firstcase}
If $n\geq 2$, then $(F\wr S_n, F\wr S_{n-1})$ is a strong Gelfand pair.
\end{prop}

\begin{proof}
Let $K$ denote the subgroup $F \wr S_{n-1} \leqslant F \wr S_n$.
Every irreducible representation of $K$ is of the form $E\boxtimes \mathbf{1}$, where $\mathbf{1}$ is the trivial representation of $S_1$, and $E$ is an irreducible representation of the factor $F\wr S_{n-1}$.
Then there exist a composition $\mathbf{b}=(b_1,\dots, b_r)$ of $n-1$ and an irreducible representation $E_0$
of $F \wr S(\mathbf{b})$ such that $E= \ind^{F\wr S_{n-1}}_{F\wr S(\mathbf{b})} E_0$. 
We want to prove that $\ind^{F\wr S_{n}}_K E\boxtimes \mathbf{1}$ is a multiplicity-free $F\wr S_n$ representation.
Since $\ind^{F\wr S_{n}}_K E\boxtimes \mathbf{1} = \ind^{F\wr S_{n}}_{F\wr (S_{n-1}\times S_1)}
\ind^{F\wr (S_{n-1}\times S_1)}_K E\boxtimes \mathbf{1}$, we will analyze the induced representation 
$\ind^{F\wr (S_{n-1}\times S_1)}_K E\boxtimes \mathbf{1}$. 
By Lemma~\ref{L:ellproduct}, we have 
\begin{align}\label{A:EU}
\ind_K^{F\wr (S_{n-1}\times S_1)} E\boxtimes \mathbf{1} =
\ind_{F\wr S_{n-1}}^{F\wr S_{n-1}} E\boxtimes \ind_{S_1}^{F\wr S_1} \mathbf{1}
= E\boxtimes U = \bigoplus_{j=1}^s E \boxtimes U_j,
\end{align}
where $U:=\ind_{S_1}^{F\wr S_1} \mathbf{1}$, and $U_1\oplus \cdots \oplus U_s = U$ is the decomposition of $U$ 
into irreducible $F\wr S_1$ representations. 
Since $F$ is abelian and $F\wr S_1 = F$, we see that $U_1,\dots, U_s$ is the complete list of pairwise inequivalent and irreducible $F$ representations.  
Finally, since $E_0$ is an irreducible $F\wr S_{n-1}$ representation, 
it is of the form $(E_1;E_1'')\boxtimes \cdots \boxtimes (E_r; E_r'')$, where 
$\{E_1,\dots, E_r\}\subseteq \{U_1,\dots, U_s\}$, 
and $E_i''$ is an irreducible representation of $S_{b_i}$ for $1\leq i \leq r$.

We now look closely at the tensor product $E\boxtimes U_j$ for $j\in \{1,\dots, s\}$. 
Since $U_j$ is an irreducible representation of $F\wr S_1$, 
we have $U_j = \ind^{F\wr S_1}_{F\wr S_1} (U_j;\mathbf{1})$.
Then, by Lemma~\ref{L:ellproduct} once again, we have 
$E\boxtimes  \ind^{F\wr S_1}_{F\wr S_1} (U_j;\mathbf{1}) =
\ind^{F\wr S_{n-1}\times F\wr S_1}_{F\wr S(\mathbf{b})\times F\wr S_1} E_0 \boxtimes  (U_j;\mathbf{1})$.
If $U_j \notin \{E_1,\dots, E_r\}$, then $E_0 \boxtimes  (U_j;\mathbf{1})$ is a typical irreducible representation of 
$F\wr S(\mathbf{b})\times F\wr S_1= F\wr S(\mathbf{b'})$, where $\mathbf{b'} = (b_1,\dots, b_r,1)$.
In this case, $\ind^{F\wr S_n}_{F\wr S_{n-1}\times F\wr S_1} E\boxtimes U_j$ is an irreducible representation of $F\wr S_n$.
If $U_j =E_i$ for some $i\in \{1,\dots, r\}$, then there is exactly one such index $i$.
Without loss of generality, let us assume that this index is $r$. 
Then by another application of Lemma~\ref{L:ellproduct} we get 
\begin{align*}
\ind^{F\wr S_{n-1}\times F\wr S_1}_{F\wr S(\mathbf{b})\times F\wr S_1} E_0 \boxtimes  (U_j;\mathbf{1})  
&= \ind^{F\wr S_{n-1}\times F\wr S_1}_{F\wr S(\mathbf{b})\times F\wr S_1}
(E_1;E_1'')\boxtimes \cdots \boxtimes ((E_r; E_r'') \boxtimes  (U_j;\mathbf{1})) \\ 
&= \ind^{F\wr S_{n-1}\times F\wr S_1}_{F\wr S(\mathbf{b''})}
\ind^{F\wr S(\mathbf{b''})}_{F\wr S(\mathbf{b})\times F\wr S_1}
(E_1;E_1'')\boxtimes \cdots \boxtimes ((E_r; E_r'') \boxtimes  (U_j;\mathbf{1})) \\ 
&= \ind^{F\wr S_{n-1}\times F\wr S_1}_{F\wr S(\mathbf{b''})}
(E_1;E_1'')\boxtimes \cdots \boxtimes (\ind^{F\wr S_{b_r+1}}_{F\wr S_{b_r}\times F\wr S_1}(E_r; E_r'') \boxtimes  (U_j;\mathbf{1})),
\end{align*}
where $\mathbf{b''}=(b_1,\dots, b_{r-1}, b_r+1)$.
Since $\ind^{S_{b_r+1}}_{S_{b_r}\times S_1} E_r''\boxtimes \mathbf{1}$ is a multiplicity-free $S_{b_r+1}$ representation,
by Lemma~\ref{L:nail}, the representation $\ind^{F\wr S_{b_r+1}}_{F\wr S_{b_r}\times F\wr S_1}(E_r; E_r'') \boxtimes  (U_j;\mathbf{1})$ is a multiplicity-free $F\wr S_{b_r+1}$ representation with irreducible summands of the form 
$(E_r; \tilde{E_r})$, where $\tilde{E_r}$ is an irreducible $S_{b_r+1}$ representation.
It follows that $E\boxtimes U_j$ for $j\in \{1,\dots, s\}$ is a multiplicity-free $F\wr S_{n-1}\times F\wr S_1$ representation. 
But since $E_r$ is uniquely determined by $U_j$ (in fact, we assumed that $E_r= U_j$),
the representations $E\boxtimes U_j$, where $j\in \{1,\dots, s\}$, do not have any irreducible constituent in common.
Also, any irreducible representation that appears in $E\boxtimes U_j$ for $j\in \{1,\dots, s\}$ 
induces up to an irreducible representation of $F\wr S_n$.  
Now applying $\ind^{F\wr S_n}_{F\wr S_{n-1}\times F\wr S_1}$ to (\ref{A:EU}) proves our claim at once; 
the representation $\ind^{F\wr S_n}_K E\boxtimes \mathbf{1}$ is multiplicity-free.
This finishes the proof of our proposition.
\end{proof}

We proceed with another example. 

\begin{lem}\label{L:GwrS2}
For every abelian group $F$, the pair $(F\wr S_2, S_2)$ is a strong Gelfand pair. 
\end{lem}

\begin{proof}
The subgroup $S_2$ has two one-dimensional irreducible representations; they are given by 
$\mathbf{1}$ and the sign representation $\mathbf{\epsilon}$. 
On one hand, since $F$ is abelian, $(F\wr S_2, S_2)$ is a Gelfand pair, hence, $\ind_{S_2}^{F\wr S_2} \mathbf{1}$ is multiplicity-free.
On the other hand, we know from the construction of irreducible representations of wreath products $F\wr S_2$ that 
the inflation of any irreducible representation of $S_2$ is an irreducible representation of $F\wr S_2$.
In particular, the (irreducible) linear representation $\mathbf{\epsilon}$ of $S_2$ extends to a linear representation of $F\wr S_2$.
Then we know that the triplet $(F\wr S_2, S_2,\mathbf{\epsilon})$ is an example of a ``twisted Gelfand pair''
~\cite[Ch VII, \S1, Exercise 10]{Macdonald}. Hence, 
$\ind_{S_2}^{F\wr S_2} \mathbf{\epsilon}$ is multiplicity-free representation of $F\wr S_2$~\cite[Ch VII, \S1, Exercise 11]{Macdonald}.
This finishes the proof. 
\end{proof}

\begin{Remark}
Let $D''$ be an irreducible representation of $S_2$. 
Then the inflation of $D''$ to $F\wr S_2$ can be identified with the irreducible representation $(\mathbf{1};D'')$.
\end{Remark}

\begin{prop}\label{P:secondcase}
For every positive integer $n$, the pair $(F\wr S_n, (F\wr S_{n-2})\times S_2)$ is a strong Gelfand pair. 
\end{prop}

\begin{proof}
The of this proposition proof is very similar to the proof of Proposition~\ref{P:firstcase}.
The only difference is that, instead of using the permutation representation $\ind_{S_1}^{F\wr S_1} \mathbf{1}$, 
we use the representations $\ind_{S_2}^{F\wr S_2} V$ where $V\in \{\mathbf{1},\mathbf{\epsilon}\}$.
By Lemma~\ref{L:GwrS2}, we know that they are multiplicity-free representations of $F\wr S_2$. 
Since the rest of the arguments are the same as in the proof of Proposition~\ref{P:firstcase}, we omit the details. 
\end{proof}

\section{A Reduction Theorem}\label{S:Reduction}

We begin with setting up some new notation that will stay in effect in the rest of our paper.

Let $X$ be a finite $G$-set with $n:=|X|$.  
Let $F$ be a finite group. 
Although $F$ is not necessarily an abelian group, for simplifying (the exponents in) our notation, 
the inverse of an element $a$ of $F$ will be denoted by $-a$. 
Accordingly, if $f$ is an element of $F^X$, or equivalently, if it is an element of the subgroup $\overline{F^X}$ in $F\wr G$,
then we will write $-f$ to denote its inverse in $F^X$ (respectively in $\overline{F^X}$).
In this notation, if $(f,g)$ is an element in $F\wr G$, then its inverse is given by 
\begin{align*}
(f,g)^{-1} = (g^{-1} (-f), g^{-1}),
\end{align*}
where $g^{-1}$ is the inverse of $g$ in $G$. 
If $(f',g')$ and $(f,g)$ are two elements from $F\wr G$, then their product is given by 
\begin{align*}
(f,g)  (f',g') = (f + g\cdot f', gg'),
\end{align*}
where $g\cdot f' : F\to X$ is the function defined by 
$g\cdot f'(x) = f'(g^{-1}x)$ for $x\in X$.

Let $\pi_G$ denote the canonical projection homomorphism onto $G$, that is, 
\begin{align*}
\pi_G : F\wr G & \longrightarrow  G\\
(f,g ) &\longmapsto g
\end{align*}
If $K$ is a subgroup of $F\wr G$, then we denote the image of $K$ under $\pi_G$ by $\gamma_K$.
Equivalently, $\gamma_K$ is given by 
\begin{align*}
\gamma_K := \{ g\in G:\ \text{there exists $f\in F^X$ such that $(f,g)\in K$}\}.
\end{align*}
The following remark/notation will be useful in the sequel.

\begin{Remark}\label{R:normalized}
For $g\in \gamma_K$, let $\Gamma_K^g$ denote the preimage $(\pi_G|_K)^{-1}(g)$. 
It is evident that $\Gamma^g_K$ ($g\in \gamma_K$) is a subgroup of $K$ if and only if $g= \id$.
Indeed, we have the following short exact sequence:
$
\{(\id,\id)\} \longrightarrow \Gamma_K^{\id} \longrightarrow K \xrightarrow{\pi_G|_K} \gamma_K \longrightarrow \{\id\}.
$
Let $(f,g)$ be an element from $K$. It is easy to check that $\pi_G((f,g)*\Gamma^{\id}_G)=\{g\}$.
In other words, we have the inclusion 
\begin{align}\label{A:inclusions}
(f,g)*\Gamma^{\id}_K \subseteq \Gamma^g_K.
\end{align}
Since the union of all left cosets of $\Gamma^{\id}_K$ covers $K$, the inclusions (\ref{A:inclusions}) are 
actually equalities of sets; every $\Gamma_K^g$ ($g\in \gamma_K$) is a left coset of $\Gamma^{\id}_K$. 
Therefore, $|\Gamma_K^g | = |\Gamma_K^{\id} |$ for all $g\in \gamma_K$.
\end{Remark}

We now go back to the strong Gelfand pairs. 
The following characterization of the strong Gelfand pairs is easy to prove.

\begin{lem}\label{L:sgiff}
Let $H$ be a subgroup of $G$. 
Then $(G,H)$ is a strong Gelfand pair if and only if $(G\times H,\diag(H))$ is a Gelfand pair. 
\end{lem}

We will apply this result to wreath products.
The main result of this section is the following reduction result.

\begin{thm}\label{T:reduction}
Let $F$ and $G$ be two finite groups, and let $K$ be a subgroup of $F\wr G$.
If $(F\wr G, K)$ is a strong Gelfand pair, then so is $(G,\gamma_K)$. 
\end{thm}

\begin{proof}
We assume that $(F\wr G, K)$ is a strong Gelfand pair. 
By Lemma~\ref{L:sgiff}, we know that $(F\wr G \times K, \diag(K))$ is a Gelfand pair. 

Let $X$ denote the finite $G$-set such that $F\wr G = F^X \rtimes G$, 
and let $H$ denote the following subset of $F\wr G \times K$:
\begin{align*}
H:= \{ ((f,b),(a,b)):\ (a,b)\in K,\ f\in F^X\}.
\end{align*}
We claim that $H$ is a subgroup. 
First, we will show that $H$ is closed under products:
Let $((f_1,b_1),(a_1,b_1))$ and $((f_2,b_2),(a_2,b_2))$ be two elements from $H$. 
Then we have 
\[
((f_1,b_1),(a_1,b_1))*((f_2,b_2),(a_2,b_2)) = ((f_1 + (b_1\cdot f_2), b_1b_2), ( a_1+(b_1 \cdot a_2),b_1b_2)). 
\]
Since the second and the fourth entries are the same, this product is contained in $H$, hence, $H$ is closed under products. 
Next, we will show that the inverses of the elements of $H$ exist: 
For $\sigma:= ((f,b),(a,b))\in H$, let $\tau:=((x,y),(z,y))$ be the element $(( b^{-1}\cdot (-f), b^{-1}), (b^{-1}\cdot (-a), b^{-1}))$ in 
$F\wr G \times K$. 
Clearly, $\tau$ is an element of $H$.
The product of $\sigma$ and $\tau$ is given by 
\begin{align*}
\sigma * \tau &=((f,b),(a,b)) * ((x,y),(z,y))\\ 
&=( (f,b) ( b^{-1}\cdot (-f), b^{-1}),  (a,b)  (b^{-1}\cdot (-a), b^{-1})) \\
&=( (f +(b\cdot (b^{-1}\cdot (-f))), bb^{-1}),  (a + (b\cdot (b^{-1}\cdot (-a))),bb^{-1}) )\\
&= ((\id, \id), (\id, \id)),
\end{align*}
hence, $((x,y),(z,y))$ is the inverse of $((f,b),(a,b))$. 
These computations show that $H$ is a subgroup of $F\wr G\times K$.

Evidently, the diagonal subgroup $\diag(K)$ in $F\wr G \times K$ is a subgroup of $H$. 
Since $(F\wr G \times K, \diag(K))$ is a Gelfand pair, it follows that $(F\wr G \times K, H)$ is a Gelfand pair as well. 
Now we will identify a normal subgroup of $F\wr G \times K$ by the help of the following map: 
\begin{align*}
\phi : F\wr G \times K \longrightarrow  G\times \gamma_K \\
((f,g),(a,b)) \longmapsto (g,b)
\end{align*}
It easy to verify that $\phi$ is a homomorphism.
It is also evident that, if an element $((x,y),(z,w))$ from $F\wr G \times K$ lies in the kernel of $\phi$, then $y=w=\id$. 
In particular, we see that $N:=\ker(\phi) \leqslant H$.
This is the normal subgroup that we were seeking. 

By Remark~\ref{R:iff1}, now we know that the pair $((F\wr G \times K)/N, H/N)$
is a Gelfand pair. But $(F\wr G \times K)/N$ is isomorphic to $G\times \gamma_K$. 
Also, it is easy to check that the diagonal subgroup $\diag(\gamma_K)$ of $F\wr G\times K$ 
is isomorphic $H/ N$ under the restriction of $\phi$.
Therefore, we have the following identification of Gelfand pairs: 
$((F\wr G \times K)/N, H/N) \cong (G\times \gamma_K, \diag(\gamma_K))$.
Finally, by using Lemma~\ref{L:iff1} once again, we conclude that $(G,\gamma_K)$ is a strong Gelfand pair.
This finishes the proof of our theorem.
\end{proof}

We mentioned earlier that, for $n\geq 7$, there are only four (minimal) strong Gelfand subgroups in $S_n$~\cite{AHN}.
As a simple consequence of Theorem~\ref{T:reduction}, we deduce a similar statement for the strong Gelfand subgroups in $F\wr S_n$.

\begin{cor}\label{C:onlyfour}
Let $n\geq 7$, and let $K$ be a subgroup of $F\wr S_n$. 
If $(F\wr S_n, K)$ is a strong Gelfand pair, then $\gamma_K \in \{S_n, A_n, S_{n-1}\times S_1,S_{n-2}\times S_2\}$.
\end{cor}

We record a partial converse of Theorem~\ref{T:reduction}.

\begin{prop}\label{P:Fiff}
Let $n\geq 7$, and let $B$ be a subgroup of $S_n$. 
Then $(F\wr S_n, F\wr B)$ is a strong Gelfand pair if and only if $(S_n,B)$ is a strong Gelfand pair. 
\end{prop}

\begin{proof}
Let $K$ be a subgroup of the form $F\wr B$ for some subgroup $B\leqslant S_n$.
Then $\gamma_K = B$.
Now Corollary~\ref{C:onlyfour} gives the $\Rightarrow$ direction. 
For the converse, by the main result of~\cite{AHN}, 
we have four cases: $B= S_n$, $B=A_n$, $B= S_{n-1}\times S_{1}$, and $B= S_{n-2}\times S_2$. 
In the first case there is nothing to do. 
The second case follows from the fact that $F\wr A_n$ is an index 2 subgroup of $F\wr S_n$. 
The last two cases follow from Theorem~\ref{T:nonabelian}. 
\end{proof}

\section{Hyperoctahedral Groups}\label{S:Hyperoctahedral}

From now on we will denote by $F$ the cyclic group of order 2. 
To simplify our notation, we will use the additive notation, so, $F=\Z/2 = \{ 0,1\}$. 
If there is no danger for confusion, the identity element of $F^n$ ($n\in \N$) will be denoted by $0$ as well. 
The wreath product $F\wr S_n$ will be denoted by $B_n$.
For $i\in \{1,\dots, n\}$, the element $x \in F^n$ which has 1 at its $i$-th entry and 0's elsewhere 
will be denoted by $e_i$. The set $\{e_1,\dots, e_n\}$ will be called the {\em standard basis for $F^n$}. 
In this notation, for $i\in \{1,\dots, n\}$, if $f\in F^n$, then $f_i$ will denote the coefficient of $e_i$ in $f$. 
When there is no danger for confusion, we will use $1$ to denote the sum of the standard basis elements, 
\begin{align}\label{A:sum}
1 = e_1+\cdots + e_n.
\end{align}

The wreath product $B_n$ is called the {\em $n$-th hyperoctahedral group}. 
It follows from the general description of the irreducible representations of wreath products that 
every irreducible linear representation of $B_n$ is equivalent to one of the induced representations, 
\begin{align}\label{A:Bnirrs}
S^{\lambda,\mu}:=\ind^{B_n}_{B_{n-k}\times B_k} (\mathbf{1}; S^\lambda)\boxtimes (\mathbf{\epsilon}; S^\mu),
\end{align}
where $0 \leq k \leq n$, 
and $S^\lambda$ (resp.~$S^\mu$) is the Specht module indexed by the partition $\lambda$ of $n-k$ (resp.~by the partition $\mu$ of $k$).
The character of $S^{\lambda,\mu}$ will be denoted by $\chi^{\lambda,\mu}$.

Our goal in the rest of this section is to determine the strong Gelfand subgroups of $B_n$.
In light of Corollary~\ref{C:onlyfour}, it will suffice to determine the Gelfand subgroups $K$ with 
$\gamma_K \in \{S_n,A_n,S_{n-1}\times S_1,S_{n-2}\times S_2\}$ only.
Before going through these cases, we will point out some well-known facts about the structures of certain subgroups of $B_n$. 
Also, we will describe a branching rule for the subgroup $B_{n-1}$ in $B_n$.

\subsection{Some special subgroups of $B_n$.}\label{SS:Bnsubgroups}

First of all, we want to point out that $B_n$ has a distinguished index 2 subgroup, denoted $D_n$. 
Actually, it is the Weyl group of type $\text{D}_n$. 
To describe it, we will view $B_n$ as a subgroup of $S_{2n}$. 
Let $(f,\sigma)$ be an element of $B_n$. 
For $i\in \{1,\dots, n\}$, we will construct a permutation $\tilde{f_i}$ of $\{1,\dots, 2n\}$ as follows: 
For $j\in \{1,\dots, 2n\}$, the value of $\tilde{f_i}$ at $j$ is defined by 
\[
\tilde{f_i}( j) =
\begin{cases} 
j & \text{ if $j\notin \{2i-1,2i\}$ or $f_i=0$},\\
2i & \text{ if $j=2i-1$ and $f_i=1$},\\
2i-1 & \text{ if $j=2i$ and $f_i=1$}.
\end{cases}
\]
Likewise, by using $\sigma$ we define a permutation $\tilde{\sigma}$ of $\{1,\dots, 2n\}$ by
\[
\tilde{\sigma} (2i-1) = 2\sigma(i)-1 \quad \text{and} \quad \tilde{\sigma} (2i) = 2\sigma(i), \quad \text{for $i\in \{1,\dots, n\}$}.
\]
We set $x= x(f,\sigma):= \tilde{f_1} \cdots \tilde{f_n} \tilde{\sigma} \in S_{2n}$.
Then the map $(f,\sigma) \mapsto x(f,\sigma)$
is an injective group homomorphism $B_n \hookrightarrow S_{2n}$. 
Abusing notation, we will denote the image of $B_n$ in $S_{2n}$ by $B_n$ as well. 
Then our distinguished subgroup is given by the intersection $D_n = A_{2n} \cap B_n$.
In the sequel, we will identify $D_n$ in $B_n$ in a different way.

It is going to be important for our purposes that we know all index 2 subgroups of $B_n$.
Fortunately, they are fairly easy to find once we give the Coxeter generators of $B_n$. 
Let $s_1,\dots, s_{n-1}$ denote the (simple) transpositions $(0,(1\,2)),(0,(2\,3)),\dots, (0,(n-1\,n))$;
these elements are the Coxeter generators of the passive factor $\overline{S_n}$ in $B_n$. 
Let $t$ denote the element $((1,0,\dots,0),\id) \in B_n$. 
Then we have the following Coxeter relations: 
\begin{itemize}
\item $s_i^2 = (s_i s_{i+1})^3 = (0,\id)$ for every $i\in \{1,\dots, n-1\}$;
\item $t^2 = (0,\id)$;
\item $(s_is_j)^2 = (0,\id)$ for every $i,j\in \{1,\dots, n-1\}$ with $|i-j| \geq 2$;
\item $(s_it)^2 = (0,\id)$ for every $i\in \{2,\dots, n-1\}$;
\item $(s_1t)^4 = (0,\id)$.
\end{itemize}

The group of linear characters of $B_n$, that is $L_n:=\Hom (B_n, \bbc^*)$, is generated by the characters $\varepsilon$
and $\delta$ defined by
\begin{align}\label{A:varepsilondelta}
\varepsilon(s_i) = -1,\ \varepsilon(t) = +1,\qquad   \delta(s_i) = +1,\ \delta(t) = -1.
\end{align}
Thus, $L_n$ is isomorphic to $F\times F$. 
Explicitly, the four linear characters $\mathbf{1}, \varepsilon, \delta, \varepsilon \delta$ correspond to $B_n$ representations as follows.

\begin{itemize}
\item The trivial character $1$ is the character of $S^{(n),\varnothing} = (\mathbf{1}; \mathbf{1}) = \mathbf{1}^{(n)} \otimes \mathbf{1}$.

\item The character $\varepsilon$ is the character of $S^{(1^n),\varnothing} = (\mathbf{1}; S^{(1^n)}) = \mathbf{1}^{(n)} \otimes \mathbf{\epsilon}$.

\item The character $\delta$ is the character of $S^{\varnothing,(n)} = (\mathbf{\epsilon}; \mathbf{1}) = \mathbf{\epsilon}^{(n)} \otimes \mathbf{1}$.

\item The character $\varepsilon \delta$ is the character of $S^{\varnothing,(1^n)} = (\mathbf{\epsilon}; S^{(1^n)}) = \mathbf{\epsilon}^{(n)} \otimes \mathbf{\epsilon}$.
\end{itemize}

These facts allow us to conclude the following useful statements:
\begin{enumerate}
\item $B_n$ has exactly three subgroups of index 2, corresponding to the kernels of the homomorphisms 
$\varepsilon,\delta$, and $\varepsilon \delta$. 
\item $B_n$ has exactly one normal subgroup of index 4, denoted by $J_n$, that is given by the intersection $\ker \varepsilon \cap \ker \delta$.
\item The kernel of $\delta$ is $D_n$.
Indeed, if $s_0$ denotes $ts_1 t$, then, as a Coxeter group, $D_n$ is generated by $s_0,s_1,\dots, s_{n-1}$.
Note that $ts_1 t = ((1,1,0,\dots, 0),\id)$.
\item The kernel of $\varepsilon$ is $F\wr A_n$. 
\end{enumerate}

\begin{Remark}
Let $G$ be a finite group. 
If there is a unique normal index 4 subgroup $J$ of $G$, 
then we think that it would be appropriate to call it the {\em Stembridge subgroup of $G$} because of
John Stembridge's seminal work on the projective representations~\cite{Stembridge92} where such a subgroup is extensively used. 
\end{Remark}

\subsubsection{The associators of index 2 subgroups of $B_n$.}\label{associators}

Our references for this subsection are the two papers~\cite{Stembridge89,Stembridge92} of Stembridge.

The linear character group $L_n$ of $B_n$ acts on the isomorphism classes of irreducible representations of $B_n$ 
via $V \mapsto \tau \otimes V$, where $\tau$ is a one-dimensional representation corresponding to an element of $L_n$. 
We continue with the assumption that $V$ is an irreducible representation of $B_n$. 
If $V \cong \tau \otimes V$, then we will say that $V$ is {\em self-associate with respect to $\tau$}; otherwise, $V$ and $\tau\otimes V$ are said to be {\em associate representations with respect to $\tau$}. 
In the sequel, when there is no danger for confusion, it will be convenient to 
denote $\tau \otimes V$ by $\chi_\tau V$, where $\chi_\tau$ is the linear character of $\tau$.

Let $H$ denote the kernel of $\chi_\tau : B_n \to \bbc^*$, 
and let $V$ be a self-associate representation with respect to $\tau$.
Then there exists an endomorphism $S\in GL(V)$ such that $g Sv = \tau(g) Sgv$ for all $g\in B_n$ and $v\in V$.
Furthermore, as a consequence of Schur's lemma, one knows that $S^2 = 1$, hence, $S$ has two eigenvalues, $\pm 1$. 
Any of the two endomorphisms $\pm S$ is called the {\em $\tau$-associator of $V$}.
Let $V^+$ (resp.~$V^-$) denote the $S$-eigenspace of eigenvalue $+1$ (resp.~eigenvalue $-1$). 
Then $V^+$ and $V^-$ are irreducible pairwise inequivalent $H$ representations. 
If $V$ and $\tau \otimes V$ are associate representations with respect to $\tau$, 
then both of them are irreducible and isomorphic as $H$ representations
~\cite[Lemma 4.1]{Stembridge89}.
Let us translate these statements to the `induced/restricted representation' language. 
Let $V$ be a self-associate representation with respect to $\tau$.
Then the restriction of $V$ to the subgroup $H$ splits into two inequivalent irreducible representations. 
If $V$ is not a self-associate representation with respect to $\tau$, then 
the restriction of $V$ to $H$ is an irreducible representation of $H$, and furthermore, 
$\ind_H^{B_n} \res^{B_n}_H V = V \oplus \tau\otimes  V$.

Let $S^{\lambda,\mu}$ be an irreducible representation of $B_n$, and let $\chi^{\lambda,\mu}$ 
denote the corresponding character. 
Then we have 
\begin{enumerate}
\item $\delta \chi^{\lambda,\mu} = \chi^{\mu, \lambda}$, hence, $S^{\lambda,\mu}$ is 
self-associate with respect to $\delta$ if and only if $\lambda = \mu$; 
\item $\varepsilon  \chi^{\lambda,\mu} = \chi^{\lambda', \mu'}$, hence, $S^{\lambda,\mu}$ is 
self-associate with respect to $\varepsilon$ if and only if $\lambda = \lambda'$ and $\mu = \mu'$; 
\item $\varepsilon \delta \chi^{\lambda,\mu} = \chi^{\mu', \lambda'}$, hence, $S^{\lambda,\mu}$ is 
self-associate with respect to $\varepsilon \delta$ if and only if $\lambda = \mu'$. 
\end{enumerate}

\subsection{$\gamma_K = S_n$.}\label{SS:1st}

If $f$ is an element of $F^n$, then we will denote by $\#f$ the number of 1's in $f$. 
For a subgroup $K$ of $B_n$, we define 
\begin{align}
m_K:= \min_{f\in \Gamma^{\id}_{K} \setminus\{ (0,\id)\}} \#f .
\end{align}
Note that $m_K$ may not exist, as we may sometimes have $\Gamma^{\id}_{K} = \{(0,\id)\}$.
Clearly, if it exists, then $m_K$ is an element of the set $\{1,\dots, n\}$. 
We have five major cases for $m_K$:
\begin{enumerate}
\item $m_K=1$,
\item $m_K=2$,
\item $3\leq m_K \leq n-1$
\item $m_K=n$,
\item $m_K$ does not exist.
\end{enumerate}

Although the above five cases are defined for $K$ with $\gamma_K=S_n$, in the sequel the same cases will be considered for 
the subgroups $K \leqslant B_n$ where $\gamma_K \in \{A_n, S_1\times S_{n-1}, S_2\times S_{n-2}\}$. 
We recall our notation from Remark~\ref{R:normalized}:
For $g\in \gamma_K$, $\Gamma^g_K$ is the preimage $(\pi_G|_K)^{-1}(g)$.

\begin{lem}\label{L:Case1}
If $m_K=1$, then $\Gamma^{\id}_K= \{ (f, \id) \in B_n:\ f\in F^n\} = \overline{F}$.
In this case, $K$ is equal to $B_n$, hence, it is a strong Gelfand subgroup.
\end{lem}
\begin{proof}
Since $m_K=1$, we know that $\Gamma^{\id}_K$ contains an element of the form $(e_k, \id)$ ($1\leq k \leq n$). 
Also, we know from Remark~\ref{R:normalized} that $\Gamma^{\id}_K$ is a normal subgroup of $K$. 
Let $(f,(i\, k))$ be an element in $K$, where $(i\, k)$ is the transposition that interchanges $i$ and $k$.
Then we have 
\begin{align*}
(f,(i\, k)) * (e_k, \id) *(f,(i\, k))^{-1} &=  
(f,(i\, k)) * (e_k, \id) *  ((i\, k)\cdot (-f), (i\, k))\\ 
&= (f + e_i, (i\,k)) *  ((i\, k)\cdot (-f), (i\, k)) \\
&= (e_i , \id) \in \Gamma^{\id}_K.
\end{align*}
But this implies that $\Gamma^{\id}_K = \overline{F}$.
Since $K/\Gamma^{\id}_K \cong S_n$, the cardinality of $K$ is equal to that of $B_n$, hence, $K=B_n$. 
This finishes the proof of our assertion.
\end{proof}

\begin{Remark}\label{R:itisnormalized}
The computation in the proof of Lemma~\ref{L:Case1} can be generalized as follows.
If $(f,\sigma)$ and $(g,\id)$ are two elements from $K$ and $\Gamma^{\id}_K$, respectively, then 
\begin{align*}
(f,\sigma)* (g,\id) * (\sigma^{-1}\cdot (-f) , \sigma^{-1}) &= (f + \sigma\cdot g, \sigma) *(\sigma^{-1}\cdot (-f) , \sigma^{-1})\\
&= (f+\sigma\cdot g -f, \id)\\
&= (\sigma\cdot g, \id).
\end{align*}
Since $\Gamma^{\id}_K$ is a normal subgroup of $K$, this computation shows that $(\sigma\cdot g,\id)\in\Gamma^{\id}_K$.
We interpret this as follows: Although $S_n$ need not be a subgroup of $K$, it still normalizes $\Gamma^{\id}_K$. 
Also, let us point out that, in coordinates, if $g = (g_1,\dots, g_n)$, then the action of $\sigma$ on $g$ is given by $\sigma\cdot g = (g_{\sigma^{-1}(1)},\dots, g_{\sigma^{-1}(n)})$.
\end{Remark}

\begin{lem}\label{L:Case2}
If $K$ is a subgroup of $B_n$ with $\gamma_K = S_n$ and $m_K=2$, then we have 
\begin{align}\label{L:Aasin}
\Gamma^{\id}_K= \{ (f, \id) \in B_n:\ \text{$\#f$ is even}\}.
\end{align}
There are two subgroups $K\leqslant B_n$ with $\Gamma^{\id}_K$ as in (\ref{L:Aasin}):
\begin{enumerate}
\item $D_n= \ker \delta$,
\item $H_n:= \ker (\varepsilon \delta)$. 
\end{enumerate}
In particular, in both of these cases, $(B_n,K)$ is a strong Gelfand pair. 
\end{lem}

\begin{proof}
Since $m_K=2$, we know that $n\geq 2$ and $\Gamma^{\id}_K$ contains an element of the form $(e_i+e_j,\id)$ for some 
$i,j\in \{1,\dots, n\}$ with $i\neq j$. 
By Remark~\ref{R:itisnormalized}, $\Gamma^{\id}_K$ is normalized by the permutation action of $S_n$, therefore, we have 
$(e_i+e_j,\id) \in \Gamma^{\id}_K$ for every $i,j\in \{1,\dots, n\}$ with $i\neq j$. 
To simplify our notation, let us denote an element $(e_i+e_j,\id)$ with $1\leq i\neq j \leq n$ by $(e_{i,j},\id)$. 
In this notation, let $f$ be a product of the form
\begin{align}\label{A:suitable}
f=(e_{i_1,j_1},\id)*\cdots * (e_{i_r,j_r},\id)= ( e_{i_1,j_1} +\cdots + e_{i_r,j_r},\id).
\end{align} 
Clearly, $f$ is an element of $\Gamma^{\id}_K$, and the number of nonzero coordinates of $ e_{i_1,j_1} +\cdots + e_{i_r,j_r}$ is 
divisible by 2. 
More generally, for $(g,\id)\in \Gamma^{\id}_K$, the first entry $g \in F^n$ cannot have an odd number of elements. 
Otherwise, by adding a suitable element of the form (\ref{A:suitable}) to it, we would have $e_i \in \Gamma^{\id}_K$ 
for some (hence for every) $i\in \{1,\dots, n\}$. 
This argument shows that every element of $\Gamma^{\id}_K$ is of the form (\ref{A:suitable}).
An easy inductive argument shows that $|\Gamma^{\id}_K | = 2^{n-1}$. 
Since $K/\Gamma^{\id}_K \cong S_n$, it follows that $| K | = 2^{n-1} n!$, hence that, $K$ is an index 2 subgroup of $B_n$.
In particular, $K$ is a strong Gelfand subgroup of $B_n$. 

For our second claim, we look at the kernels of the three nontrivial linear characters of $B_n$ (Subsection~\ref{SS:Bnsubgroups}).
It is easy to verify from the descriptions of $\varepsilon$ and $\delta$ that 
$\Gamma^{\id}_K \leqslant \ker \delta$ and $\Gamma^{\id}_K \leqslant \ker \varepsilon \delta$.
Hence, $K$ is equal to either $\ker \delta$ or $\ker \varepsilon \delta$.
In the former case, we already noted that $\ker \delta = D_n$. 
In the latter case, we observe that $(f,\sigma) \in \ker \varepsilon \delta$ if and only if 
the parities of $\varepsilon ( (0,\sigma))$ and $\delta((f,\id))$ are the same. 
This finishes the proof of our lemma. 
\end{proof}

\begin{Remark}
We see from Lemma~\ref{L:Case2} that among the many equivalent descriptions of $D_n$ 
we have $D_n = \{ (f, \sigma) \in B_n:\ \text{$\#f$ is even}\}$.
\end{Remark}

\begin{lem}\label{L:Case3}
If $K$ is a subgroup of $B_n$ with $\gamma_K = S_n$, then $m_K \notin\{3,\dots, n-1\}$. 
In other words, there is no subgroup $K\leqslant B_n$ such that $\gamma_K = S_n$ and $3\leq m_K  \leq n-1$. 
\end{lem}

\begin{proof}
Assume towards a contradiction that there exists a subgroup $K$ in $B_n$ such that $m_K$ is an element of $\{3,\dots, n-1\}$. 
Let $(f,\id)$ be an element of $\Gamma^{\id}_K$. 
Then there are some basis elements $e_{i_1},\dots, e_{i_r}$ ($r\in \{3,\dots, n-1\}$) 
with $1\leq i_1 <\cdots < i_r \leq n$ such that $f = e_{i_1}+\cdots + e_{i_r}$. 
Without loss of generality we assume that $i_r \neq n$. 
Let $\sigma$ denote the transposition $(i_r\, n)$ so that $\sigma (i_r ) = n$. 
Since $\gamma_K = S_n$, the element $(0,\sigma)$ is contained in $K$. 
In particular, $(0,\sigma)*(f,\id) = (\sigma \cdot f, \sigma)$ is an element of $K$. 
Then, $(\sigma \cdot f,\sigma)* (\sigma \cdot f, \sigma) = (f+ \sigma \cdot f, \id)$ is an element of $K$. 
But this last element is equal to $(e_{i_r}+e_n,\id)$. 
Since $\# (e_{i_r}+e_n,\id) = 2$, we obtained a contradiction to our initial assumption $m_K \geq 3$. 
This finishes the proof of our assertion.
\end{proof}

\begin{Remark}\label{R:conjbyaut}
Let $H_1$ and $H_2$ be two subgroups of a group $G$.
The assumption that $H_1$ is isomorphic to $H_2$ does not guarantee that the following implications hold: 
\begin{align}\label{A:guarantees}
\text{$(G,H_1)$ is a strong Gelfand pair $\iff$ $(G,H_2)$ is a strong Gelfand pair}.
\end{align}
For example, let $(G,H_1,H_2)$ be the triplet $(G,H_1,H_2):=(B_2,\diag(F),\overline{S_2})$.
It is easy to verify that $(G,H_1)$ is not a strong Gelfand pair but $(G,H_2)$ is a strong Gelfand pair. 
Nevertheless, if $H_1$ and $H_2$ are isomorphic via an automorphism of $G$, then Remark~\ref{R:iff1} shows that 
(\ref{A:guarantees}) hold.
\end{Remark}

\begin{lem}\label{L:Case4}
If for a subgroup $K\leqslant B_n$ with $\gamma_K = S_n$ the integer $m_K$ is $n$, then $K$ is conjugate-isomorphic to a subgroup of $\diag(F)\times S_n$. 
Moreover, these subgroups are strong Gelfand subgroups of $B_n$ if and only if $n\leq 5$.
\end{lem}

\begin{proof}
For $m_K=n$, the fact that $\Gamma_K^{\id} = \{ (0,\id), (1,\id)\}$ follows from definitions. 
Therefore, we have a copy of the central subgroup $\diag(F)\times \{ \id_{S_n}\}$ in $K$. 
Since any element of $K$ is of the form $(f,\sigma)$, where $f\in \Gamma_K^{\id}$ and $\sigma \in S_n$, 
we see that $K$ is conjugate-isomorphic to a subgroup of $\diag(F)\times S_n$. 
Furthermore, since $\gamma_K = S_n$, we know that the index of $K$ in $\diag(F)\times S_n$ is at most 2. 
As the group of linear characters of $\diag(F)\times S_n$ is isomorphic to $\Z/2\times \Z/2$, we see that 
$K$ can be a conjugate of one of the following four subgroups: 1) $\diag(F)\times S_n$, 2) $\{ id\}\times S_n$, 3) $\diag(F)\times A_n$, 
4) a non-direct product subgroup of index 2. 
Notice that, in our case, 
the subgroups in 3) cannot occur since $\gamma_K = S_n$, and the subgroups in 2) cannot occur because it needs $m_K=n$.

Next, we will show that, for $n\geq 6$, 
$K:=\diag(F)\times S_n$ is not a strong Gelfand subgroup of $B_n$;
applying Remark~\ref{R:conjbyaut} will then handle the other possible subgroups in item 1).
Those in item 4) follow by transitivity of induction.

Let $U$ be an irreducible representation of $B_n$ of the form 
$U:=\ind_{F \wr (S_a \times S_b)}^{B_n} (\mathbf{1};D_1'')\boxtimes  (\mathbf{\epsilon}; D_2'')$,
where $D_1'',D_2''$ are some irreducible representations of $S_{a}$ and $S_{b}$, respectively, and $a+b=n$.
Let us set $G:= B_n$, and $H:=F\wr (S_a\times S_b)$. 
Let $W$ be an irreducible representation of $K$; it is of the form $D\boxtimes D''$, where 
$D \in \{\mathbf{1},\epsilon\}$, and $D''$ is an irreducible $S_n$-representation.
The multiplicity of $U$ in $\ind_K^G W$ is equal to the dimension of the vector space 
\[
M:= \Hom_G ( U , \ind_K^G W). 
\]
As in the proof of Theorem~\ref{T:firststep}, 
we will use Mackey's formula and Frobenius reciprocity to compute the dimension of $M$.
Let $V$ denote the representation $(\mathbf{1};D_1'') \boxtimes (\mathbf{\epsilon}; D_2'')$.
Then $H$ is the inertia group of $V$. 
Let $S$ be a system of representatives for the $(H,K)$-double cosets in $G$. 
Since $\overline{F}$ is a normal subgroup of $G$, and since it is contained in $H$, we see that $HK=G$.
In other words, $S= \{\id\}$. 
Therefore, there is only one local group of the form $G_s= H\cap s K s^{-1}$,
which is given by $G_{\id} = \diag(F)\times(S_a\times S_b)$. 
By Mackey's formula~(\ref{A:Mackey}),
\begin{align}\label{A:M}
M&= \Hom_G (\ind_{H}^G V, \ind_K^G W)  \notag \\
&=  \bigoplus_{s\in S} \Hom_{G_s} (\res^H_{G_s} V, W_s) \notag \\
&= \Hom_{\diag(F)\times (S_a\times S_b)} (\res^{H}_{\diag(F)\times (S_a\times S_b)} V, W_{\id}), 
\end{align}
where $W_{\id}$ is the copy of $W$ viewed as a representation of $\diag(F)\times (S_a\times S_b)$, 
that is, 
\begin{align}\label{A:subsback1}
W_{\id} = \res^K_{\diag(F)\times (S_a\times S_b)} W =
 \res^{\diag(F)\times S_n}_{\diag(F)\times (S_a\times S_b)} D\boxtimes D'' 
 = D\boxtimes \res^{S_n}_{S_a\times S_b} D''.
\end{align}
Recall from Subsection~\ref{SS:Repsofwreaths} that for a representation $M$ of $F$, and for $r\in \N$,
the notation $M^{(n)}$ stands for the $r$-fold outer tensor product $M\boxtimes \cdots \boxtimes M$. 
In this notation, we have 
\begin{align}\label{A:subsback2}
\res^{H}_{\diag(F)\times (S_a\times S_b)} V
&=\res^{(F\wr S_a)\times (F\wr S_b)}_{ S_a \times (\diag(F)\times S_b)} (\mathbf{1}; D_1'') \boxtimes (\epsilon;D_2'') \notag \\
&=\res^{F\wr S_a}_{ S_a} (\mathbf{1}; D_1'') \boxtimes \res^{F\wr S_b}_{ \diag(F)\times S_b}(\epsilon;D_2'')
\notag \\
&= D_1'' \boxtimes (\res^{F^b}_{\diag(F)} \epsilon^{(b)} ) \boxtimes D_2'' \notag \\
&=   (\res^{F^b}_{\diag(F)} \epsilon^{(b)} )  \boxtimes  D_1'' \boxtimes D_2''.
\end{align}
Note that since $F= \Z/2$, the restricted representation $\res^{F^b}_{\diag(F)} \epsilon^{(b)}$ is either $\mathbf{1}$ or $\epsilon$,
depending on the parity of $b$. 
Substituting (\ref{A:subsback2}) and (\ref{A:subsback1}) in (\ref{A:M}), and using 
the Frobenius reciprocity, we see that 
\begin{align*}
\dim M = 
\begin{cases}
0  & \text{ if $\res^{F^b}_{\diag(F)} \epsilon^{(b)} \neq D$; }\\
\dim \Hom_{S_a\times S_b} (D'', \ind^{S_n}_{S_a\times S_b} D_1''\boxtimes D_2'') & \text{ if $\res^{F^b}_{\diag(F)} \epsilon^{(b)} = D$.}
\end{cases}
\end{align*}
In particular, we see that if $W= (\res^{F^b}_{\diag(F)} \epsilon^{(b)}) \boxtimes D''$, then the multiplicity of $U$ in $\ind^{B_n}_K W$ is equal to 
the multiplicity of the irreducible $S_n$-representation $D''$ in $\ind^{S_n}_{S_a\times S_b} D_1''\boxtimes D_2''$.
For $n\geq6$, we have examples of such representations with multiplicity at least 2.

When $n\leq 5$, the result for subgroups conjugate to $\diag(F)\times S_n$ follows from the fact that the passive factor $\overline{S_n}$ is strong Gelfand, 
which we prove in Lemma~\ref{L:Case5}.
For the subgroups conjugate to an index 2 subgroup of this, the result can easily be checked by computer.
This completes the proof.
\end{proof}

\begin{Remark}
There is an easier, alternative proof of the second part of Lemma~\ref{L:Case4} for $n\geq 10$.
Indeed, since $n\geq10$, we always have an irreducible representation $V$ of $\overline{S_n}$ that induces up with multiplicity at least 3.
For example, by using Theorem~\ref{T:firststep}, we see that multiplicity of $S^{((n-7,2,1),(3,1))} = \ind_{F\wr(S_{n-4}\times S_4)}^{B_n} (\mathbf{1}; S^{(n-7,2,1)}) \boxtimes (\mathbf{\epsilon}; S^{(3,1)})$ in $\ind_{\overline{S_n}}^{B_n} S^{(n-6,3,2,1)}$ is equal to the multiplicity of $S^{(n-6,3,2,1)}$ in $\ind_{S_{n-4}\times S_4}^{S_n} S^{(n-7,2,1)} \boxtimes S^{(3,1)}$.
As in the proof of Corollary~\ref{C:firststep}, we can easily count that there are three Littlewood--Richardson tableaux of skew-shape $(n-6,3,2,1)/(n-7,2,1)$ and weight $(3,1)$.
Now, since $\overline{S_n}$ is a subgroup of index 2 in $Y_n$, $\ind_{\overline{S_n}}^{Y_n} V = V_1 \oplus V_2$, where $V_1$ and $V_2$
are two irreducible representations of $Y_n$. Let $W$ be an irreducible representation of $B_n$ such that the multiplicity of $W$ in $\ind_{\overline{S_n}}^{B_n} V$ is at least 3. Then the multiplicity of $W$ in one the induced representations $\ind^{B_n}_{Y_n} V_1$ or $\ind^{B_n}_{Y_n} V_2$ 
is at least 2. Hence, we conclude that $(B_n, \diag(F) \times S_n)$
is not a strong Gelfand pair.
\end{Remark}

\begin{lem}\label{L:Case5}
Let $Y_n$ denote the subgroup
\[
Y_n:= \left\{(f,\sigma) \in B_n :\ 
f=
\begin{cases}
0 & \text{ if $\sigma$ is an even permutation};\\
1 & \text{ if $\sigma$ is an odd permutation}.
\end{cases} \ \
\right\}.
\]
Then $Y_n$ is isomorphic to the passive factor $\overline{S_n}$ by an automorphism of $B_n$. 
Furthermore, if for a subgroup $K \leqslant B_n$ with $\gamma_K = S_n$ the integer $m_K$ does not exist, 
then $K$ is either conjugate to $\overline{S_n}$ or it is conjugate to $Y_n$.
In this case, $(B_n,K)$ is a strong Gelfand pair if and only if $n\leq 5$.   
\end{lem}

\begin{proof}
If $m_K$ does not exist, then $\Gamma^{\id}_K$ has only one element, $\Gamma^{\id}_K= \{ (0,\id) \}$. 
Therefore, $K$ is isomorphic to $\gamma_K$.
Clearly, $\overline{S_n}$ is such a subgroup, and the conjugate subgroups $(f,\id)* \overline{S_n} * (f,\id)^{-1}$, 
where $f\in F^n \setminus \{1\}$, are all different from each other.
Likewise, it is easy to check that $Y_n$ is a subgroup of $B_n$ such that $\gamma_{Y_n} = S_n$ and $m_{Y_n}$ does not exist. 
Furthermore, $(f,\id)* Y_n * (f,\id)^{-1}$, where $f\in F^n \setminus \{1\}$, are all different from each other.
It is not difficult to show that these are all possible subgroups of $B_n$ with $\gamma_K = S_n$ and $m_K$ does not exist. 
We omit the details of this part of the proof.

Next, we will show that the subgroups $\overline{S_n}$ and $Y_n$ are strong Gelfand subgroups of $B_n$ if and only if $n\leq 5$. 
We already proved the former case in Corollary~\ref{C:firststep}. 
To prove that $Y_n$ is not a strong Gelfand subgroup, we introduce the following map: 
\begin{align*}
\psi: B_n &\longrightarrow B_n \\
(f,\sigma) &\longmapsto 
\begin{cases}
(f,\sigma) & \text{ if $\sigma \in A_n$};\\
(f+1,\sigma) & \text{ if $\sigma \notin A_n$}.
\end{cases}
\end{align*}
We claim that $\psi$ is an automorphism of $B_n$. 
Clearly, $\psi$ is well-defined and one-to-one. 
Hence, it is a bijection. 
Let $(f,\sigma)$ and $(g,\tau)$ be two elements from $B_n$. 
Then 
$(f,\sigma) * (g,\tau) = (f+ \sigma \cdot g , \sigma \tau)$. 
We have four cases: 1) $\sigma, \tau \in A_n$;
2) $\sigma \in A_n, \tau \notin A_n$;
3) $\sigma, \tau \notin A_n$;
4) $\sigma \notin A_n, \tau \in A_n$.
In each case, it is easily checked that 
$\psi ((f,\sigma)) * \psi((g,\tau)) = \psi((f+ \sigma \cdot g , \sigma \tau))$.
It follows that $\psi$ is an automorphism. 
Moreover, we see from the description of $\psi$ that it is an outer automorphism of $B_n$. 
Evidently, we have $\psi(\overline{S_n}) = Y_n$. 
But then by Remark~\ref{R:iff1} we know that 
\[
\text{$(B_n , \overline{S_n})$ is a strong Gelfand pair} \iff 
\text{$(B_n , Y_n)$ is a strong Gelfand pair}. 
\]
Therefore, by Corollary~\ref{C:firststep}, $Y_n$ is a strong Gelfand subgroup of $B_n$ if and only if $n\leq 5$.
Applying Remark~\ref{R:conjbyaut} finishes the proof of our lemma.
\end{proof}

We now paraphrase the conclusions of the above lemmas in a single proposition.

\begin{prop}\label{P:Summary1}
Let $n\geq 6$, and let $K$ be a strong Gelfand subgroup of $B_n$.
Let $\varepsilon:B_n \to \bbc^*$ and $\delta: B_n\to \bbc^*$ be the linear characters as defined in (\ref{A:varepsilondelta}).
If $\gamma_K=S_n$, then (up to a conjugation) $K$ is one of the following subgroups:
\begin{enumerate}
\item $B_n$;
\item $D_n = \ker \delta$;
\item $H_n= \ker (\varepsilon \delta)$.
\end{enumerate}
For $n\leq 5$, in addition to these three cases, we have the following possibilities:
$\overline{S_n}$, $Y_n$ from Lemma~\ref{L:Case5} and $\diag(F) \times S_n$.
\end{prop}

\subsection{$\gamma_K = A_n$.}\label{SS:2nd}

\begin{Assumption}
Unless otherwise noted, we assume that $n\geq 3$. 
\end{Assumption}

We maintain our notation from the previous subsections.

For $n\geq 3$, $A_n$ is generated by the following 3-cycles:
\begin{align}\label{A:gens}
(1\, 2\, 3),(1\, 2\, 4),\dots,(1\, 2\, n).
\end{align}
Indeed, we know that the transpositions $(1\, 2),\dots, (1\, n)$ generate $S_n$. 
If $\sigma$ is element of $A_n$, then we write it as a product of these transpositions, 
\[
\sigma = (1\, l_1) (1\, l_2) \cdots (1\, l_{2j}).  
\]
Between every consecutive subproduct $(1\, l_{2i-1})(1\, l_{2i})$ for $1\leq i \leq j$, 
we insert the trivial product $\id = (1\, 2)(1\, 2)$.
Then we observe that 
\[
(1\, l_{2i-1})(1\,2) =  (1\, 2\, l_{2i-1})\ \text{ and }\  (1\, 2)(1\, l_{2i}) =  (1\, 2\, l_{2i})^{-1},
\]  
hence that  
\[
(1\, l_{2i-1})(1\, l_{2i}) =  (1\, 2\, l_{2i-1})(1\, 2\, l_{2i})^{-1}.
\]
This observation shows that the 3-cycles in (\ref{A:gens}) generate $A_n$.

Now we proceed to determine all strong Gelfand subgroups $K$ with $\gamma_K = A_n$. 
As in the case of $S_n$, we will analyze the following five cases for $m_K$: 
1) $m_K=1$; 2) $m_K=2$; 3) $3\leq m_K \leq n-1$; 4) $m_K=n$; and 5) $m_K$ does not exist.

\begin{lem}\label{L:AnCase1}
If $m_K=1$, then $\Gamma^{\id}_K= \{ (f, \id) \in B_n:\ f\in F^n\} = \overline{F}$.
In this case, $K$ is equal to $F\wr A_n$, hence, we have a strong Gelfand subgroup. 
\end{lem}
\begin{proof}
We will argue as in Lemma~\ref{L:Case1}: 
Since $m_K=1$, $\Gamma^{\id}_K$ contains an element of the form $(e_k,\id)$ for some $k$ in $\{1,\dots, n\}$. 
Let $(f,(1\, 2\, k))\in K$.
Without loss of generality, we will assume that $k>2$.
Since $\Gamma^{\id}_K$ is a normal subgroup of $K$, we know that  
\begin{align*}
(f, (1\, 2\, k))* (e_k,\id) * (f, (1\, 2\, k))^{-1} &= ((1\, 2\, k)\cdot e_k,\id) = (e_1,\id) \in \Gamma^{\id}_K.
\end{align*}
Then it is not difficult to see that $(e_l,\id) \in \Gamma^{\id}_K$ for every $l\in \{1,\dots, n\}$. 
It follows that, $\Gamma^{\id}_K$ is equal to $\overline{F}$.
This means that, for every $(f,\sigma) \in K$, we have $(0,\sigma) = (-f,\id)*(f,\sigma) \in K$.
In other words, the alternating subgroup of the passive factor $\overline{S_n}$ is a subgroup of $K$.
But this shows that $F^n \rtimes A_n$ is a subgroup of $K$. 
Since this an index 2 subgroup in $B_n$, and since $K$ is a proper subgroup, we have the equality
$F\wr A_n = K$. 
In particular, $K$ is a strong Gelfand subgroup of $B_n$ by Proposition~\ref{P:Fiff}. 
\end{proof}

\begin{lem}\label{L:AnCase2}
If $K$ is a subgroup of $B_n$ with $\gamma_K = A_n$ and $m_K=2$, then 
\begin{align}\label{L:AnCase2Gamma}
\Gamma^{\id}_K= \{ (f, \id) \in B_n:\ \text{$\#f$ is even}\}, \text{ and hence }\ K = J_n =  \ker \varepsilon \cap \ker \delta.
\end{align}
Then $(B_n,K)$ is a strong Gelfand pair if and only if $n \not\equiv 2 \mod 4$. 
\end{lem}

\begin{proof}
First assume that $n\geq4$.
Since $m_K=2$, we know that $\Gamma^{\id}_K$ contains an element of the form $(e_i+e_j,\id)$ for some 
$i,j\in \{1,\dots, n\}$ with $i < j$. 
Let $u$ be a subset $u:=\{ k,l\}$ of $\{1,\dots, n\}$ with $k<l$ and $u\cap \{i,j\} = \emptyset$.
Let $\sigma$ denote the (even) permutation $(k\, i) (l\, j)$, and let $x$ be an element from $\Gamma^{(k\, i) (l\, j)}_K$. 
Then $x= (f,\sigma)$ for some $f\in F^n$. 
Since conjugating by $x$ gives 
\begin{align*}
x* (e_i+e_j,\id) * x^{-1} &=  ((k\, i) (l\, j)\cdot (e_i+e_j), \id) = (e_k + e_l ,\id ), 
\end{align*}
every element of the form $(e_r+e_s,\id)$, where $1\leq r < s \leq n$ and $\{r,s\} \cap \{i,j\} = \emptyset$ is contained in $\Gamma^{\id}_K$. 
Note that, we already have $(e_i+e_j,\id) \in \Gamma^{\id}_K$.
Now by the argument that we used in the proof of Lemma~\ref{L:Case2}, 
we see that if $(f,\id)$ is an element of $\Gamma^{\id}_K$, then $\#f$ is even. 
In particular, we see that $|\Gamma^{\id}_K | = 2^{n-1}$. 
Since $K/\Gamma^{\id}_K \cong A_n$, we see that $|K| = 2^{n-1} n!/2 = 2^{n-2}n!$. 
Thus, $K$ is an index 4 subgroup of $B_n$. 
Finally, it is easy to check that $K$ is contained in both of the subgroups $\ker \varepsilon$ and $\ker \delta$. 
Therefore, $K$ is equal to $\ker \varepsilon \cap \ker \delta$.
This finishes the proof of (\ref{L:AnCase2Gamma}).

We now proceed to prove our second claim.
As in Section~\ref{associators}, we let $L_n$ denote the linear character group of $B_n$.
Let $V$ be a finite-dimensional irreducible representation of $B_n$, 
and let $L_V$ denote the stabilizer of $V$ in $L_n$, that is, $L_V = \{\tau \in L_n:\  V \cong \tau \otimes V\}$.
In (the proof of)~\cite[Theorem 3.1]{Stembridge92}, Stembridge describes in detail the decomposition of the restriction $\res^{B_n}_K V$ into $K$-representations. 
In particular, Stembridge's theorem shows that $\res^{B_n}_K V$ is {\em not} multiplicity-free if and only if 
\begin{enumerate}
\item $L_V = L_n = \{\id, \varepsilon,\delta, \varepsilon \delta\}$,
\item the $\varepsilon$-associator $S$ of $V$ and the $\delta$-associator of $V$ anti-commute, $ST = -TS$. 
\end{enumerate}
Since these conditions require the existence of a $\delta$-associator, which is possible only if $n$ is even, 
we conclude that $\res^{B_n}_KV$ is a multiplicity-free representation if $n$ is odd.
We now proceed with the assumption that $n$ is even. 
An irreducible representation $V$ with the correspond character $\chi^{\lambda,\mu}$ is self-associate with respect to all of the three 
linear characters $\varepsilon,\delta$ and $\delta \varepsilon$ if and only if $\lambda=\mu = \mu'$,
where $\mu'$ is the partition conjugate to $\mu$. In other words, $\chi^{\lambda,\mu} = \chi^{\lambda,\lambda}$ and
$\lambda$ is a self-conjugate partition of $n/2$. 
In this case, it is known that, the associators $S=S^{\lambda,\lambda}$ and $T=T^{\lambda}$ 
anti-commute if and only if $n/2$ is an odd integer (see the paragraph after~\cite[Theorem 6.4]{Stembridge92}).

If $n=3$, we can explicitly check that $K = J_3$, which is a strong Gelfand subgroup of $B_3$.
\end{proof}

\begin{lem}\label{L:AnCase3}
Let $K$ be a subgroup of $B_n$ with $\gamma_K = A_n$.
Then $m_K \notin \{3,\dots, n-1\}$. 
In other words, there is no subgroup $K\leqslant B_n$ with $\gamma_K = A_n$ and $3\leq m_K \leq n-1$. 
\end{lem}

\begin{proof}
Suppose that there did exist a subgroup $K$ of $B_n$ such that $m_K \in \{3,\dots, n-1\}$. 
Then there are some basis elements $e_{i_1},\dots, e_{i_r}$ ($r\in \{3,\dots, n-1\}$) 
with $1\leq i_1 <\cdots < i_r \leq n$ such that $(f,\id) := (e_{i_1}+\cdots + e_{i_r},\id)\in \Gamma^{\id}_K$. 
Without loss of generality we assume that $i_r \neq n$. 
Let $(g, (i_1\, i_2)(i_r\, n))$ be an element from $\Gamma^{(i_1\, i_2)(i_r\, n)}_K$. 
Then we have 
\[
(f,\id)*(g, (i_1\, i_2)(i_r\, n)) = (f+g, (i_1\, i_2)(i_r\, n) ) \in K.
\]
In particular, the following elements are contained in $\Gamma^{\id}_K$:
\begin{align*}
(x,\id) &:= (f+g, (i_1\, i_2)(i_r\, n) ) * (f+g, (i_1\, i_2)(i_r\, n)) = (f+g + (i_1\, i_2)(i_r\, n) \cdot (f+g), \id), \\
(y,\id) &:= (g, (i_1\, i_2)(i_r\, n))* (g, (i_1\, i_2)(i_r\, n))  = (g+  (i_1\, i_2)(i_r\, n)\cdot g , \id).
\end{align*}
Notice that for $j\notin \{i_1,i_2,i_r, n\}$, we have $x_j = y_j = 0$.
If $j\in \{i_1, i_2\}$, then we have $x_j = y_j$; if $j\in \{ i_r, n\}$, then we have $x_j = y_j + 1$.
Now we consider the product 
\[
(x,\id) * (y,\id) = ( x + y, \id) = (e_{i_r}+e_n, \id) \in K.
\]
Since $\# (e_{i_r}+e_n,\id) = 2$, we obtained a contradiction to our initial assumption that $m_K \geq 3$.
This finishes the proof of our lemma.
\end{proof}

\begin{lem}\label{L:AnCase4}
If $K$ is a subgroup of $B_n$ with $\gamma_K = A_n$ and $m_K=n$, then $K$ is conjugate to a subgroup of $\diag(F) \times A_n$.
In particular, $K$ is not strong Gelfand, unless $n\leq 5$. 
\end{lem}

\begin{proof}
Since $m_K = n$, we have $\Gamma^{\id}_K = \{ (0,\id),(1,\id)\}$, which is a central subgroup of $B_n$. 
Since $\gamma_K = A_n$, by the exact sequence in Remark~\ref{R:normalized}, we see that $K$ is conjugate to a subgroup of $\diag(F) \times A_n$. 
In particular, the index of $K$ in (a conjugate of) $\diag(F)\times A_n$ is at most 2. 
Without loss of generality, we assume that $K$ is a subgroup of $\diag(F)\times A_n$. 
Then since the group of linear characters of $\diag(F)\times A_n$ has order 2, and since $K\neq A_n$, we see that $K= \diag(F)\times A_n$. 
Since $K\leqslant \diag(F)\times S_n$, by Lemma~\ref{L:Case4}, we find that $K$ is not a strong Gelfand subgroup for $n\geq 6$. 
For $3\leq n \leq 5$, we verified in GAP that $\diag(F)\times A_n$ is a strong Gelfand subgroup in $B_n$ if and only if $n =3$.
\end{proof}

\begin{lem}\label{L:AnCase5}
If for a subgroup $K \leqslant B_n$ with $\gamma_K = A_n$ the integer $m_K$ does not exist, 
then $K$ is conjugate to the alternating subgroup of the passive factor $\overline{S_n}$.
In this case, $(B_n, K)$ is not a strong Gelfand pair for $n\geq 4$. 
If $n=3$, then $(B_n, K)$ is a strong Gelfand pair. 
\end{lem}

\begin{proof}
The proof proceeds in a similar way to that of Lemma~\ref{L:AnCase4}.

Since $m_K$ does not exist, we have $\Gamma^{\id}_K = \{ (0,\id)\}$.
Since $\pi_{B_n} ((0,(1\,2))) = (1\,2)$ is not contained in $A_n$, the element $(0,(1\,2))$ is not contained in $K$.
Let $H$ denote the subgroup of $B_n$ that is generated by $(0,(1\,2))$ and $K$. 
Then it is easy to check that $\gamma_H = S_n$ and that $\Gamma^{\id}_H = \Gamma^{\id}_K$, hence, $m_H$ does not exist.  
It follows that $H$ is one of the subgroups that we considered in Lemma~\ref{L:Case5}.

By conjugating $H$ we may assume that $H = \overline{S_n}$ or the group $Y_n$ defined therein, and that $K\leqslant  \overline{S_n}$ or $K\leqslant Y_n$.
But $|K| = n!/2$ and we know that $\gamma_K=A_n$, hence we infer that $K = \overline{A_n}$ (in both $\overline{S_n}$ and $Y_n$). 
When $n\geq 6$, since $H$ is not a strong Gelfand subgroup, nor is $K$.
For $n\leq 5$ we checked in GAP that $K$ is not a strong Gelfand subgroup unless $n=3$.
This finishes the proof. 
\end{proof}

We paraphrase the conclusions of the above lemmas in a single proposition. 

\begin{prop}\label{P:Summary2}
Let $K$ be a strong Gelfand subgroup of $B_n$. 
If $\gamma_K=A_n$, then $K$ is one of the following subgroups:
\begin{enumerate}
\item $F\wr A_n$;
\item the Stembridge subgroup of $B_n$, that is, 
$J_n=\ker \varepsilon \cap \ker \delta$, where $\varepsilon$ and $\delta$ are two inequivalent nontrivial linear characters of $B_n$,
where $n \not\equiv 2 \mod 4$ for $n\geq 4$,
\item $\diag(F) \times A_3$ or $A_3$ if $n=3$.
\end{enumerate}
\end{prop}

\section{The Cases of $\gamma_K=S_{n-1}\times S_1$ and $\gamma_K = S_{n-2}\times S_2$}\label{S:lattertwo}

In this last part of our paper, we analyze the strong Gelfand subgroups $K$ in $B_n$, where $\gamma_K \in \{ S_{n-1}\times S_1, S_{n-2}\times S_2\}$.
These two cases provide us with the most diversity. 
For our proofs we heavily use analogs of the ``Pieri's formulas'' for the hyperoctahedral groups. 

Throughout this section also, $F$ will denote the cyclic group $\Z/2$.

\subsection{Some Pieri rules.}\label{SS:Appendix}

Our goal in this section is to explicitly compute the decomposition formulas for 
induced representations from various subgroups of $B_n$. 
While some of these formulas are known~\cite{Pushkarev}, 
we could not locate all of the decomposition rules that we need for our purposes.

Let $k$ be an element of $\{0,1,\dots, n-1\}$, let $\lambda$ be a partition of $n-1-k$, and let $\mu$ be a partition of $k$. 
Let $S^{\lambda,\mu}$ denote the corresponding $B_{n-1}$-Specht module. 
We begin our computations by studying the decomposition of 
$\ind^{B_n}_{B_{n-1}\times S_1} S^{\lambda,\mu}\boxtimes \mathbf{1}$ into its irreducible constituents.

By using 1) the transitivity of the induction, 2) Lemma~\ref{L:ellproduct}, 3) the additivity of the induction, we see that 
\begin{align}
\ind^{B_n}_{B_{n-1}\times S_1} S^{\lambda,\mu}\boxtimes \mathbf{1} &= 
\ind^{B_n}_{B_{n-1}\times B_1}\ind^{B_{n-1}\times B_1}_{B_{n-1}\times S_1} S^{\lambda,\mu}\boxtimes \mathbf{1} \notag \\
&= \ind^{B_n}_{B_{n-1}\times B_1} (S^{\lambda,\mu}\boxtimes \ind^{B_1}_{S_1}\mathbf{1}) \notag \\
&= \ind^{B_n}_{B_{n-1}\times B_1} (S^{\lambda,\mu}\boxtimes ( (\mathbf{1}\boxtimes \mathbf{1}) \oplus (\mathbf{\epsilon}\boxtimes
\mathbf{1}))) \notag \\
&= 
(\ind^{B_n}_{B_{n-1}\times B_1} S^{\lambda,\mu}\boxtimes (\mathbf{1};\mathbf{1}))\oplus  
(\ind^{B_n}_{B_{n-1}\times B_1} S^{\lambda,\mu}\boxtimes (\mathbf{\epsilon};\mathbf{1})). \label{A:focus0}
\end{align}

Recall that the $B_{n-1}$-Specht module $S^{\lambda,\mu}$ is given by 
$\ind^{B_{n-1}}_{B_{n-1-k}\times B_k} (\mathbf{1};S^\lambda)\boxtimes (\mathbf{\epsilon};S^\mu)$.
Note also that $(\mathbf{1};\mathbf{1}) = \ind^{B_1}_{B_1} (\mathbf{1};\mathbf{1})$.
We apply these observations to the first summand in (\ref{A:focus0}):
\begin{align}
\ind^{B_n}_{B_{n-1}\times B_1} S^{\lambda,\mu}\boxtimes (\mathbf{1};\mathbf{1}) &= 
\ind^{B_n}_{B_{n-1}\times B_1}  (\ind^{B_{n-1}}_{B_{n-1-k}\times B_k} (\mathbf{1};S^\lambda)\boxtimes (\mathbf{\epsilon};S^\mu))
\boxtimes  (\ind^{B_1}_{B_1} (\mathbf{1};\mathbf{1})) \notag \\
&= 
\ind^{B_n}_{B_{n-1}\times B_1}  (\ind^{B_{n-1}\times B_1}_{B_{n-1-k}\times B_k\times B_1} 
(\mathbf{1};S^\lambda)\boxtimes (\mathbf{\epsilon};S^\mu) \boxtimes (\mathbf{1};\mathbf{1})) \  \text{ (Lemma~\ref{L:ellproduct})}\notag \\
&= 
\ind^{B_n}_{ B_{n-1-k}\times B_1 \times B_{k}} 
(\mathbf{1};S^\lambda)\boxtimes (\mathbf{1};\mathbf{1}) \boxtimes
(\mathbf{\epsilon};S^\mu). \label{A:last111}
\end{align}
By the transitivity of the induced representations, we have 
\[
\ind^{B_n}_{B_{k}\times B_{n-1-k}\times B_1}  = \ind^{B_n}_{B_{k}\times B_{n-k}} \ind^{B_k\times B_{n-k}}_{B_{k}\times B_{n-1-k}\times B_1}.
\]
Therefore, by Lemma~\ref{L:ellproduct}, the following induced representation is equal to (\ref{A:last111});
\begin{align}\label{A:last1111}
 \ind^{B_n}_{B_{k}\times B_{n-k}} 
 \left( \ind^{B_k}_{B_k}(\mathbf{\epsilon};S^\mu)  \right) \boxtimes \left( \ind^{B_{n-k}}_{B_{n-1-k}\times B_1}
(\mathbf{1};S^\lambda)\boxtimes (\mathbf{1};\mathbf{1}) \right).
\end{align}
Now by applying Lemma~\ref{L:nail} to $\ind^{B_{n-k}}_{B_{n-k}\times B_1}(\mathbf{1};S^\lambda)\boxtimes (\mathbf{1};\mathbf{1})$,
we re-express the formula (\ref{A:last1111}), hence, the formula (\ref{A:last111}), more succinctly as follows:
\begin{align}\label{A:focus1}
\ind^{B_n}_{B_{n-1}\times B_1} S^{\lambda,\mu}\boxtimes (\mathbf{1};\mathbf{1})&=
\ind^{B_n}_{B_{k}\times B_{n-k}} (\mathbf{\epsilon};S^\mu) \boxtimes (\mathbf{1};S^\lambda \boxtimes \mathbf{1}). 
\end{align}

Next, we focus on $\ind^{B_n}_{B_{n-1}\times B_1} S^{\lambda,\mu}\boxtimes (\mathbf{\epsilon};\mathbf{1})$. 
By adapting the above arguments to this case, we find that 
\begin{align}\label{A:focus2}
\ind^{B_n}_{B_{n-1}\times B_1} S^{\lambda,\mu}\boxtimes (\mathbf{\epsilon};\mathbf{1})&=
\ind^{B_n}_{B_{k}\times B_{n-k}} (\mathbf{\epsilon};S^\mu\boxtimes \mathbf{1}) \boxtimes (\mathbf{1};S^\lambda). 
\end{align}

\begin{Notation}
If $\lambda$ is a partition of $n-1$, then $\overline{\lambda}$ denotes the set of all partitions obtained from 
$\lambda$ by adding a ``box'' to its Young diagram. 
Equivalently, we have 
\[
\overline{\lambda} =
\{ \tau \vdash n :  
\text{$S^\tau$ is a summand of $\ind^{S_n}_{S_{n-1}\times S_1} S^\lambda \boxtimes \mathbf{1}$}
\}.
\]
\end{Notation}

\begin{lem}\label{L:Branching}
If $S^{\lambda, \mu}$ is an irreducible representation of $B_{n-1}$, then we have the following decomposition rules:
\begin{enumerate}
\item $\ind^{B_n}_{B_{n-1}\times B_1} S^{\lambda, \mu} \boxtimes (\mathbf{1};\mathbf{1}) = 
\bigoplus_{\tau \in \overline{\lambda}} S^{\tau, \mu}$,
\item $\ind^{B_n}_{B_{n-1}\times B_1} S^{\lambda, \mu} \boxtimes (\mathbf{\epsilon};\mathbf{1}) = 
\bigoplus_{\rho \in \overline{\mu}} S^{{\lambda}, \rho}$, 
\item $\ind^{B_n}_{B_{n-1}\times S_1} S^{\lambda, \mu} \boxtimes \mathbf{1} = 
\bigoplus_{\tau \in \overline{\lambda}} S^{\tau, \mu}  \oplus  \bigoplus_{\rho \in \overline{\mu}} S^{{\lambda}, \rho}$.
\end{enumerate}
\end{lem}
\begin{proof}
The first and the second identities follow from (\ref{A:focus1}) and (\ref{A:focus2}), respectively, 
where we decompose the $S_{n-k}$ (resp.~$S_{k+1}$) representation $\ind_{S_{n-k-1} \times S_1}^{S_{n-k}} S^\lambda \boxtimes \mathbf{1}$ (resp.~$\ind_{S_k \times S_1}^{S_{k+1}}S^\mu\boxtimes \mathbf{1}$)
into irreducible constituents. 
We then use the additivity property for the induced representations. 
In light of the decomposition (\ref{A:focus0}), the third identity follows from the first two identities.
\end{proof}

\begin{Remark}
Lemma~\ref{L:Branching} combined with the fact that $(S_n,S_{n-1}\times S_1)$ is a strong Gelfand pair
 gives a second proof of the fact that $(B_n,B_{n-1}\times S_1)$ is a strong Gelfand pair. 
\end{Remark}

Our goal is to extend Lemma~\ref{L:Branching} to certain subgroups of $B_{n-2}\times B_2$, 
so, we setup some relevant notation: 
\begin{Notation}
If $\lambda$ be a partition of $n-2$, then $\bar{\bar{\lambda}} $ and $\tilde{\bar{\lambda}}$ are the following sets of partitions of $n$:
\begin{align*}
\bar{\bar{\lambda}} 
&:= \{ \tau \vdash n :  \text{$S^\tau$ is a summand of $\ind^{S_n}_{S_{n-2}\times S_2} S^\lambda \boxtimes \mathbf{1}_{S_2}$} \},\\
\tilde{\bar{\lambda}}
&:= \{ \tau \vdash n :  \text{$S^\tau$ is a summand of $\ind^{S_n}_{S_{n-2}\times S_2} S^\lambda \boxtimes \mathbf{\epsilon}_{S_2}$}\}, 
\end{align*}
where $\mathbf{\epsilon}_{S_2}$ is the sign representation of $S_2$. 
\end{Notation}

The irreducible representations of $B_2$ are easy to list. They are given by 
\begin{align}\label{A:irrrepsofB2}
S^{(2),\emptyset} ,S^{(1^2),\emptyset}, S^{(1),(1)}, S^{\emptyset,(1^2)},S^{\emptyset,(2)}.
\end{align}
In Table~\ref{F:charvaluesB2}, we present the values of the characters of these representations;
they are computed by using the formula (5.5) in~\cite{Stembridge92}.
\begin{figure}[htp]
\begin{center}
  \begin{tabular}{ c | c | c | c | c | c }
    & $(1^2),\emptyset$   & $(2),\emptyset$ & $(1),(1)$ & $\emptyset, (1^2)$ &  $\emptyset, (2)$    \\    \hline
   $\chi^{(2),\emptyset}$ & 1 & 1 & 1 & 1 & 1  \\ \hline
      $\chi^{(1^2),\emptyset}$ & 1 & -1 & 1 & 1 & -1  \\ \hline
          $\chi^{(1),(1)}$ & 2 & 0 & 0 & -2 & 0  \\ \hline
              $\chi^{\emptyset,(2)}$ & 1 & 1 & -1 & 1 & -1  \\ \hline
                  $\chi^{\emptyset,(1^2)}$ & 1 & -1 & -1 & 1 & 1 
  \end{tabular}
\end{center}
  \caption{Character table of $B_2$.}
\label{F:charvaluesB2}
\end{figure}

We fix an integer $n\geq 4$.
For $k\in \{0,1,\dots, n-2\}$, let $S^{\lambda,\mu}$ be an irreducible representation of $B_{n-2}$, 
where $\lambda \vdash n-k-2$ and $\mu \vdash k$.
Our goal is to compute the irreducible constituents of $\ind^{B_n}_{B_{n-2}\times B_2} S^{\lambda,\mu} \boxtimes W$, 
where $W$ is one of the representations in (\ref{A:irrrepsofB2}).
The method of computation for all these five cases are similar, so, we will present details only in the first case.

\textbf{The case of $W = S^{(2),\emptyset}$.}
Recall that $S^{\lambda,\mu}$ is equal to $\ind^{B_{n-2}}_{B_{n-2-k}\times B_k} (\mathbf{1};S^\lambda)\boxtimes (\mathbf{\epsilon};S^\mu)$, and that 
$S^{(2),\emptyset}=(\mathbf{1};\mathbf{1}_{S_2}) = \ind^{B_2}_{B_2} (\mathbf{1};\mathbf{1}_{S_2}) $.
By using Lemma~\ref{L:ellproduct} repeatedly, we transform our induced representation to another form:
\begin{align*}
\ind^{B_n}_{B_{n-2}\times B_2} S^{\lambda,\mu}\boxtimes  (\mathbf{1};\mathbf{1}_{S_2})
&=  \ind^{B_n}_{B_{n-2}\times B_2}  
\left( \ind^{B_{n-2}}_{B_{n-2-k}\times B_k} (\mathbf{1};S^\lambda)\boxtimes (\mathbf{\epsilon};S^\mu)\right) 
\boxtimes \left(\ind^{B_2}_{B_2} (\mathbf{1};\mathbf{1}_{S_2}) \right) \\
&=  \ind^{B_n}_{B_{n-2}\times B_2}  
\left( \ind^{B_{n-2} \times B_2}_{B_{n-2-k}\times B_k\times B_2} (\mathbf{1};S^\lambda)\boxtimes (\mathbf{\epsilon};S^\mu) 
\boxtimes (\mathbf{1};\mathbf{1}_{S_2}) \right) \\
&=  \ind^{B_n}_{B_{n-2-k}\times B_k\times B_2} (\mathbf{1};S^\lambda)\boxtimes (\mathbf{\epsilon};S^\mu) 
\boxtimes (\mathbf{1};\mathbf{1}_{S_2}) \\
&=  \ind^{B_n}_{B_{n-k}\times B_k} 
\ind^{B_{n-k}\times B_k}_{B_{n-2-k}\times B_2\times B_k} 
(\mathbf{1};S^\lambda) \boxtimes (\mathbf{1};\mathbf{1}_{S_2}) \boxtimes (\mathbf{\epsilon};S^\mu) \\
&= 
 \ind^{B_n}_{B_{n-k}\times B_{k}} \left( \ind^{B_{n-k}}_{B_{n-k-2}\times B_2}
(\mathbf{1};S^\lambda)\boxtimes (\mathbf{1};\mathbf{1}_{S_2}) \right) 
\boxtimes \left( \ind^{B_k}_{B_k}(\mathbf{\epsilon};S^\mu)  \right).
\end{align*}
By Lemma~\ref{L:nail}, $\ind^{B_{n-k}}_{B_{n-k-2}\times B_2} (\mathbf{1};S^\lambda)\boxtimes (\mathbf{1};\mathbf{1}_{S_2})$
is equal to $(\mathbf{1};S^\lambda\boxtimes \mathbf{1}_{S_2})$.
Since $ \ind^{B_k}_{B_k}(\mathbf{\epsilon};S^\mu) = (\mathbf{\epsilon};S^\mu)$, we proved that 
\begin{align}\label{A:B21}
\ind^{B_n}_{B_{n-2}\times B_2} S^{\lambda,\mu}\boxtimes  (\mathbf{1};\mathbf{1}_{S_2}) 
&=  \ind^{B_n}_{B_{n-k}\times B_{k}}  
(\mathbf{1};S^\lambda\boxtimes \mathbf{1}_{S_2}) \boxtimes  (\mathbf{\epsilon};S^\mu).
\end{align}

\textbf{The case of $W= S^{(1^2),\emptyset}$.}
In this case we have
$W=(\mathbf{1};\mathbf{\epsilon}_{S_2}) = \ind^{B_2}_{B_2} (\mathbf{1};\mathbf{\epsilon}_{S_2})$ and
\begin{align}\label{A:B22}
\ind^{B_n}_{B_{n-2}\times B_2} S^{\lambda,\mu}\boxtimes  (\mathbf{1};\mathbf{\epsilon}_{S_2})&=  
\ind^{B_n}_{B_{n-k}\times B_{k}}  (\mathbf{1};S^\lambda\boxtimes \mathbf{\epsilon}_{S_2}) \boxtimes (\mathbf{\epsilon};S^\mu).
\end{align}

\textbf{The case of $W= S^{(1),(1)}$.}
In this case, we have $W= \ind^{B_2}_{B_1\times B_1} (\mathbf{1};\mathbf{1})\boxtimes (\mathbf{\epsilon};\mathbf{1})$ 
and 
\begin{align}\label{A:B23}
\ind^{B_n}_{B_{n-2}\times B_2} S^{\lambda,\mu}\boxtimes W &= 
\ind^{B_n}_{B_{n-k-1}\times B_{k+1}}   (\mathbf{1};S^\lambda\boxtimes \mathbf{1})
\boxtimes 
(\mathbf{\epsilon};S^\mu\boxtimes \mathbf{1}).
\end{align}

\textbf{The case of $W= S^{\emptyset, (1^2)}$.}
In this case, we have $W=(\mathbf{\epsilon};\mathbf{\epsilon}_{S_2}) = \ind^{B_2}_{B_2} (\mathbf{\epsilon};\mathbf{\epsilon}_{S_2}) $
and 
\begin{align}\label{A:B24}
\ind^{B_n}_{B_{n-2}\times B_2} S^{\lambda,\mu}\boxtimes  (\mathbf{\epsilon};\mathbf{\epsilon}_{S_2})&= 
\ind^{B_n}_{B_{n-k-2}\times B_{k+2}} (\mathbf{1};S^\lambda) \boxtimes (\mathbf{\epsilon};S^\mu \boxtimes \mathbf{\epsilon}_{S_2}).
\end{align}

\textbf{The case of $W= S^{\emptyset, (2)}$.}
In this case, we have $W=(\mathbf{\epsilon};\mathbf{1}_{S_2}) = \ind^{B_2}_{B_2} (\mathbf{\epsilon};\mathbf{1}_{S_2})$
and  
\begin{align}\label{A:B25}
\ind^{B_n}_{B_{n-2}\times B_2} S^{\lambda,\mu}\boxtimes  (\mathbf{\epsilon};\mathbf{1}_{S_2}) &=  
\ind^{B_n}_{B_{n-k-2}\times B_{k+2}}   (\mathbf{1};S^\lambda) \boxtimes (\mathbf{\epsilon};S^\mu \boxtimes \mathbf{1}_{S_2}).
\end{align}

We are now ready to present our decomposition rules for the induced representations from $B_{n-2}\times B_2$ to $B_n$.

\begin{lem}\label{L:Branching2}
If $S^{\lambda, \mu}$ is an irreducible representation of $B_{n-2}$, then we have the following decomposition formulas: 
\begin{enumerate}
\item $\ind^{B_n}_{B_{n-2}\times B_2} S^{\lambda, \mu} \boxtimes S^{(2),\emptyset} = 
\bigoplus_{\tau \in \bar{\bar{\lambda}}} S^{\tau, \mu}$,
\item $\ind^{B_n}_{B_{n-2}\times B_2} S^{\lambda, \mu} \boxtimes S^{(1^2),\emptyset} = 
\bigoplus_{\tau \in \tilde{\bar{\lambda}}} S^{\tau, \mu}$,
\item $\ind^{B_n}_{B_{n-2}\times B_2} S^{\lambda, \mu} \boxtimes S^{(1),(1)} = 
\bigoplus_{\tau \in \bar{\lambda},\rho \in \bar{\mu} } S^{\tau, \rho}$,
\item $\ind^{B_n}_{B_{n-2}\times B_2} S^{\lambda, \mu} \boxtimes S^{\emptyset,(1^2)} = 
\bigoplus_{\rho \in \tilde{\bar{\mu}} } S^{\lambda,\rho}$,
\item $\ind^{B_n}_{B_{n-2}\times B_2} S^{\lambda, \mu} \boxtimes S^{\emptyset,(2)} = 
\bigoplus_{\rho \in \bar{\bar{\mu}} } S^{\lambda,\rho}$.
\end{enumerate}
\end{lem}
\begin{proof}
All five of these formulas follow from Lemma~\ref{L:nail} and the decompositions we described in (\ref{A:B21})--(\ref{A:B25}).
\end{proof}

\begin{Remark}
Lemma~\ref{L:Branching2} combined with the fact that $(S_n,S_{n-2}\times S_2)$ is a strong Gelfand pair  
gives a second proof of the fact that $(B_n,B_{n-2}\times B_2)$ is a strong Gelfand pair, see~Theorem~\ref{T:nonabelian}. 
\end{Remark}

\begin{cor}\label{A:C:D2}
Let $S^{\lambda, \mu}$ be an irreducible representation of $B_{n-2}$, and let $W$ be an irreducible representation of $D_2$.
Then we have
\begin{enumerate}
\item if $\ind^{B_2}_{D_2} W=S^{(2),\emptyset}\oplus S^{\emptyset, (2)}$, then $\ind^{B_n}_{B_{n-2}\times D_2} S^{\lambda, \mu} \boxtimes W = 
\bigoplus_{\tau \in \bar{\bar{\lambda}}} S^{\tau, \mu} \oplus \bigoplus_{\rho \in \bar{\bar{\mu}} } S^{\lambda,\rho}$; 
\item if $\ind^{B_2}_{D_2} W=S^{(1^2),\emptyset}\oplus S^{\emptyset, (1^2)}$, 
then $\ind^{B_n}_{B_{n-2}\times D_2} S^{\lambda, \mu} \boxtimes W = \bigoplus_{\tau \in \tilde{\bar{\lambda}}} S^{\tau, \mu}
\oplus \bigoplus_{\rho \in \tilde{\bar{\mu}} } S^{\lambda,\rho}$;
\item if $\ind^{B_2}_{D_2} W=S^{(1),(1)}$, then $\ind^{B_n}_{B_{n-2}\times D_2} S^{\lambda, \mu} \boxtimes W =  
\bigoplus_{\tau \in \bar{\lambda},\rho \in \bar{\mu} } S^{\tau, \rho}$.
\end{enumerate}
\end{cor}

\begin{proof}
As a subgroup of $B_2$, $D_2$ is given by $\diag(F)\times S_2$, which is isomorphic to $\Z/2\times \Z/2$. 
Hence, $D_2$ has four irreducible representations.
However, two of these irreducible representations induce up the same representation of $B_2$.
Indeed, by Clifford theory, and the tools of Section~\ref{associators}, the irreducible representations of $D_2$ are found as follows. 
An irreducible representation $V$ of $D_2$ is given by  
either $\res^{B_2}_{D_2} S^{\lambda,\mu}$, where $\lambda$ and $\mu$ are two distinct partitions such that $|\lambda | + |\mu| = 2$, 
or by one of the two irreducible constituents of $\res^{B_2}_{D_2} S^{(1),(1)}$.
By Frobenius reciprocity, we see that $\ind^{B_2}_{D_2} V$ is one of the following three representations of $B_2$: 
$S^{(1),(1)}, S^{(2),\emptyset}\oplus S^{\emptyset, (2)}$, or $S^{(1^2),\emptyset}\oplus S^{\emptyset, (1^2)}$.
The rest of the proof follows from Lemma~\ref{L:Branching2}.
\end{proof}

\begin{cor}\label{A:C:H2}
Let $S^{\lambda, \mu}$ be an irreducible representation of $B_{n-2}$, and let $W$ be an irreducible representation of $H_2$.
Then we have 
\begin{enumerate}
\item if $\ind^{B_2}_{H_2} W=S^{(2),\emptyset}\oplus S^{\emptyset, (1^2)}$, 
then $\ind^{B_n}_{B_{n-2}\times H_2} S^{\lambda, \mu} \boxtimes W = 
\bigoplus_{\tau \in \bar{\bar{\lambda}}} S^{\tau, \mu} \oplus \bigoplus_{\rho \in \tilde{\bar{\mu}} } S^{\lambda,\rho}$;

\item if $\ind^{B_2}_{H_2} W=S^{\emptyset, (2)}\oplus S^{(1^2),\emptyset}$, 
then $\ind^{B_n}_{B_{n-2}\times H_2} S^{\lambda, \mu} \boxtimes W = 
\bigoplus_{\rho \in \bar{\bar{\mu}} } S^{\lambda,\rho} \oplus \bigoplus_{\tau \in \tilde{\bar{\lambda}}} S^{\tau, \mu}$; 

\item if $\ind^{B_2}_{H_2} W=S^{(1),(1)}$, then $\ind^{B_n}_{B_{n-2}\times H_2} S^{\lambda, \mu} \boxtimes W =  
\bigoplus_{\tau \in \bar{\lambda},\rho \in \bar{\mu} } S^{\tau, \rho}$.
\end{enumerate}
\end{cor}

\begin{proof}
The group $H_2$ is isomorphic to $\Z/4$, and hence it has four inequivalent irreducible representations. 
Arguing as in Corollary~\ref{A:C:D2}, these representations can be described in terms of the irreducible representations of $B_2$. 
Any irreducible representation $V$ of $H_2$ is either equal to 
$\res^{B_2}_{H_2} S^{\lambda,\mu}$, where $\lambda$ and $\mu$ are two distinct partitions 
such that $\lambda \neq \mu'$ and $|\lambda | + |\mu| = 2$, or it is one of the two irreducible constituents of 
the representation $\res^{B_2}_{H_2} S^{(1),(1)}$.
By Frobenius reciprocity, $\ind^{B_2}_{H_2} V$ is one of the following representations of $B_2$: 
$S^{(1),(1)}, S^{(2),\emptyset}\oplus S^{\emptyset, (1^2)},S^{\emptyset, (2)}\oplus S^{(1^2),\emptyset}$.
The rest of the proof follows from Lemma~\ref{L:Branching2}.
\end{proof}

\begin{cor}\label{A:C:D21}
Let $S^{\lambda, \mu}$ be an irreducible representation of $B_{n-2}$, and let $W$ be an irreducible representation of $\overline{S_2}$.
Then we have 
\begin{align*}
\ind^{B_2}_{\overline{S_2}}  W= 
\begin{cases}
S^{(2),\emptyset}\oplus  S^{\emptyset,(2)} \oplus S^{(1),(1)} & \text{ if $W = \mathbf{1}$,}\\
S^{(1^2),\emptyset}\oplus S^{\emptyset,(1^2)} \oplus S^{(1),(1)} & \text{ if $W = \mathbf{\epsilon}$.}
\end{cases}
\end{align*}
Consequently, we have 
\begin{align*}
\ind^{B_n}_{B_{n-2}\times \overline{S_2}} S^{\lambda, \mu} \boxtimes W= 
\begin{cases}
\bigoplus_{\tau \in \bar{\bar{\lambda}} } S^{\tau,\mu} 
\oplus 
\bigoplus_{\rho \in \bar{\bar{\mu}} } S^{\lambda,\rho} 
\oplus 
\bigoplus_{\tau \in \bar{\lambda},\rho \in \bar{\mu} } S^{\tau, \rho}
 & \text{ if $W = \mathbf{1}$,}\\
\bigoplus_{\tau \in \tilde{\bar{\lambda}}} S^{\tau, \mu} 
\oplus 
\bigoplus_{\rho \in \tilde{\bar{\mu}}} S^{\lambda, \rho}
\oplus 
\bigoplus_{\tau \in \bar{\lambda},\rho \in \bar{\mu} } S^{\tau, \rho} & \text{ if $W = \mathbf{\epsilon}$.}
\end{cases}
\end{align*}

\end{cor}

\begin{proof}
Our first claim follows from a direct computation by using Table~\ref{F:charvaluesB2}.
(Alternatively, one can use Theorem~\ref{T:firststep}.) 
The rest of the proof follows from Lemma~\ref{L:Branching2}.
\end{proof}

\subsection{$\gamma_K = S_{n-1}\times S_1$.}\label{SS:3rd}

Let $H$ be an index 2 subgroup of a finite group $G$, and let $\eta$ denote the sign representation of the quotient $G/H \cong \Z/2$. 
If $W$ is an irreducible representation of $G$ with character $\chi$, then we have two cases:
\begin{enumerate}
\item[(a)] $\chi \eta = \chi \implies \text{$\res^G_H W$ has two irreducible constituents $V_1$ and $V_2$}$.
The corresponding induced representations $\ind_H^G V_i$ ($i\in \{1,2\}$) are equal to $W$.
\item[(b)] $\chi \eta \neq \chi \implies \text{$\res^G_H W$ is an irreducible representation $V$ of $H$}$.
The corresponding induced representation $\ind^G_H V$ is given by $W\oplus W'$, where the character of $W'$ is $\chi\eta$. 
\end{enumerate}
We will apply this standard fact in the following special situation: $G=A\times B$, where $A$ and $B$ are two finite groups.

Let $\phi : G\to \Z/2$ be a surjective homomorphism.
We denote the restriction of $\phi$ to the subgroup $A\times \{1\}$ by $\phi_1$. 
Likewise, the restriction of $\phi$ onto $\{\id\}\times B$ is denoted by $\phi_2$. 
Then for every $(a,b)\in G$, we have $\phi (a,b) = \phi_1(a,1)\phi_2(\id,b)$. 
In particular, we have three distinct possibilities for the kernel $H$ of $\phi$:
\begin{enumerate}
\item[] (H1) $\phi_2$ is the trivial homomorphism. In this case, $\phi_1$ must be surjective.
Otherwise we have both of the subgroups $A\times \{1\}$ and $\{1\}\times B$ as subgroups of $H$,
hence, we have $H=G$, which is absurd. 
Now, since $\phi_1$ is surjective, $\ker \phi_1$ is an index 2 subgroup of $A\times \{1\}$. 
In particular, $H= \ker \phi_1B$.

\item[] (H2) $\phi_2$ is a surjective homomorphism and $\phi_1$ is the trivial homomorphism.
This case is similar to (H1), hence, we have $H=  A\times \ker \phi_2$.
\item[] (H3) Both of the homomorphisms $\phi_1$ and $\phi_2$ are surjective. 
Then the kernel of $\phi$ is given by 
$H = \{(a,b) \in A\times B:\  \phi_1(a)=\phi_2(b) \}$.
\end{enumerate}

Clearly, the most complicated case is the case of (H3).
We will refer to this case as the {\em non-obvious index 2 subgroup case}.
Nevertheless, if $B$ is $\Z/2 = \{ 1,-1\}$ (in multiplicative notation), then we can describe $H$ quite explicitly,
\[
H=\ker\phi_1 \times \{1\}\cup \overline{\ker\phi_1}\times \{-1\},
\]
where $\overline{\ker\phi_1}= \{(a,-1):\  \phi_1(a)=-1 \}$.

\begin{Example}\label{E:initialexample}
Let us consider the following factors: $A= B_n$ and $B= F = \Z/2$.
Recall that $B_n$ has three index 2 subgroups corresponding to three nontrivial linear characters $\varepsilon,\delta$, and $\varepsilon\delta$ (see~\ref{A:varepsilondelta}).
Thus, for $A=B_n$, $H$ is one of the following subgroups in $B_n\times F$:
\begin{itemize}
\item[] (H1) $H= B_n\times \{ 1\}$;
\item[] (H2.a) $H= \ker(\delta) \times F = D_n\times F$;
\item[] (H2.b) $H= \ker(\varepsilon\delta) \times F= H_n\times F$;
\item[] (H2.c) $H= \ker(\varepsilon)\times F$;
\item[] (H3.a) $H = \{(a,b) \in B_n\times F:\  \varepsilon(a)=\phi_2(b) \}$;
\item[] (H3.b) $H = \{(a,b) \in B_n\times F:\  \delta(a)=\phi_2(b) \}$;
\item[] (H3.c) $H = \{(a,b) \in B_n\times F:\  \varepsilon\delta(a)=\phi_2(b) \}$.
\end{itemize}
Let $V$ be an irreducible representation of $H$. 
\begin{enumerate}

\item Case 1. Let $H$ be as in (H1). 
Then $V= V'\boxtimes \mathbf{1}$, where $V'$ is an irreducible representation of $A$, hence,
$\ind^G_H V = V' \boxtimes (\ind^F_{\id} \mathbf{1}) = V'\boxtimes (\mathbf{1}\oplus \mathbf{\epsilon})$. 

\item Case 2. Let $H$ be as in (H2.a)--(H2.c). Then $V= V'\boxtimes V''$, where $V'$ (resp.~$V''$) is an irreducible representation of $\ker \phi_1$
(resp.~of $F$), hence, $\ind^G_H V = (\ind^A_{\ker \phi_1} V') \boxtimes V''$.

\item Case 3. Let $H$ be as in (H3.a)--(H3.c). 
In particular, $\ker\phi_1\times \{1\}$ is an index 2 subgroup of $H$; we have 
\[
H=\ker\phi_1\times \{1\} \cup \overline{\ker\phi_1}\times \{-1\},
\] 
where $\phi_1 \in \{\varepsilon\times1,\delta\times1,\varepsilon\delta\times 1\}$. 
Therefore, $\ker\phi_1\times \{1\}$ is an index 4 subgroup of $G=B_n\times F$.
In fact, it is a normal subgroup of $G$, so, the irreducible representations of $\ker\phi_1\times \{1\}$ are easy to describe. 
Consequently, we can effectively analyze $\ind^G_HV$ in relation with the induced 
representations $\ind^G_{\ker\phi_1\times \{1\}} V'$, where $V'$ is an irreducible representation of ${\ker\phi_1\times \{1\}}$.
We will do this in the sequel for the cases (H3.b) and (H3.c). 

\end{enumerate}
\end{Example}

\begin{Example}\label{E:initialexample2}
Now we consider $A= D_n$ and $B= F$. 
Then $D_n$ has a unique subgroup of index 2, namely, the Stembridge subgroup, $J_n=D_n \cap H_n$. 
Therefore, $H$ is one of the following subgroups in $D_n\times F$:
\begin{itemize}
\item[] (H1) $H= D_n\times \{ 1\}$;
\item[] (H2) $H= J_n \times F$;
\item[] (H3) $H = \{(a,b) \in D_n\times F:\  \varepsilon(a)=\phi_2(b) \}$, where $\phi_2 : F\to \{1,-1\}$ is the sign representation. 
\end{itemize}
\end{Example}

\begin{Example}\label{E:initialexample3}
Now we consider $A= H_n$ and $B= F$. 
Then $H_n$ has a unique subgroup of index 2, namely, the Stembridge subgroup, $J_n=D_n \cap H_n$. 
Therefore, $H$ is one of the following subgroups in $H_n\times F$:
\begin{itemize}
\item[] (H1) $H= H_n\times \{ 1\}$;
\item[] (H2) $H= J_n \times F$;
\item[] (H3) $H = \{(a,b) \in H_n\times F:\  \delta(a)=\phi_2(b) \}$, where $\phi_2 : F\to \{1,-1\}$ is the sign representation. 
\end{itemize}
\end{Example}

\begin{Assumption}
In the rest of this subsection, $K$ will denote a subgroup of $B_n$ such that $\gamma_K=S_{n-1}\times S_1$, hence,  
$K\leqslant  F \wr (S_{n-1}\times S_1)$. 
\end{Assumption}

\begin{Notation}\label{N:alphabeta}
We denote by $\phi$ the natural isomorphism $\phi : F\wr (S_{n-1}\times S_1) \to B_{n-1}\times B_1$. 
For $\lambda \in F\wr (S_{n-1}\times S_1)$ and $(a,b)\in B_{n-1}\times B_1$ such that $\phi(\lambda) = (a,b)$, 
we denote by $\lambda_\alpha$ the element of $F\wr (S_{n-1}\times S_1)$ such that 
$\phi (\lambda_\alpha) = (a,\id_{B_1})$.
Similarly, we will denote by $\lambda_\beta$ the element of $F\wr (S_{n-1}\times S_1)$ 
such that $\phi(\lambda_\beta) = (\id_{B_{n-1}},b)$. 
\end{Notation}

In this notation, we now have the following two subgroups of $B_n$:
\begin{align*}
\Lambda^\alpha_K  := \{ \lambda_\alpha :\ \lambda \in K \}\qquad \text{and}\qquad 
\Lambda^\beta_K  := \{ \lambda_\beta :\ \lambda \in K\}.
\end{align*} 
Clearly these two subgroups commute with each other.

\begin{lem}\label{L:Kissubgroup}
We maintain the notation from the previous paragraph. 
Then we have 
\begin{enumerate}
\item $K \leqslant \Lambda^\alpha_K \Lambda^\beta_K \cong \Lambda^\alpha_K \times \Lambda^\beta_K$; 
\item $\gamma_{\Lambda^\alpha_K} \cong S_{n-1}$ in $S_n$ and $\gamma_{\Lambda^\beta_K} \cong S_1$ in $S_n$.
\end{enumerate}
\end{lem}

\begin{proof}
The first item follows from the definitions of $\Lambda^\alpha_K$ and $\Lambda^\beta_K$. 
For the second item, we observe that the restriction $\pi_{S_n} |_{F\wr (S_{n-1}\times S_1)}$ 
of the canonical projection $\pi_{S_n} :B_n \to S_n$ factors through $\phi$. 
Since $\pi_{S_n}(\Lambda^\alpha_K)\cup \pi_{S_n}( \Lambda^\beta_K) \subseteq \pi_{S_n}(K)$ and 
since we have $\phi(\Lambda^\alpha_K) \leqslant B_{n-1}\times \{\id\}$ and $\phi(\Lambda^\beta_K) \leqslant \{\id\}\times B_1$, the inclusion $K\leqslant \Lambda^\alpha_K \Lambda^\beta_K$ implies that 
\[
\pi_{S_n}(\Lambda^\alpha_K) = S_{n-1}\times  \{\id\} \qquad \text{and}\qquad \pi_{S_n}(\Lambda^\beta_K) = \{\id\}\times S_1.\qedhere
\]
\end{proof}

\begin{Remark}\label{R:indexofK}
It follows from definitions that $|\Lambda^\alpha_K | \leq |K|$. 
Since $K$ is a subgroup of $\Lambda^\alpha_K \Lambda^\beta_K$ and since $|\Lambda^\beta_K| \leq 2$, we have 
$|\Lambda^\alpha_K | \leq |K| \leq 2|\Lambda^\alpha_K | = | \Lambda^\alpha_K \Lambda^\beta_K|$.
In particular, we have the inequality $[\Lambda^\alpha_K \Lambda^\beta_K:K] \leq 2$.
\end{Remark}

By abuse of notation, in our next result, which we call the {\em second reduction theorem},   
we will identify $\Lambda^\alpha_K$ with its image under $\phi$.
Similarly, we will view $ \Lambda^\beta_K$ as a subgroup of $B_1$.

\begin{thm}\label{T:reductionII}
If $(B_n, K)$ is a strong Gelfand pair, then 
so is $(B_{n-1}, \Lambda_K^\alpha)$.
\end{thm}

\begin{proof}
Let us denote by $H$ the following subgroup of $B_{n-1}\times B_1 \times \Lambda^\alpha_K\times \Lambda^\beta_K$:
\[
H=\{ (a,b',a,b) | \ a \in \Lambda_{K}^{\alpha}, \ b \in \Lambda_{K}^{\beta}, b'\in B_1 \}.
\]
Clearly, $H$ contains the diagonal copy of $\Lambda^\alpha_K\times \Lambda^\beta_K$
in $B_{n-1}\times B_1 \times \Lambda^\alpha_K\times \Lambda^\beta_K$ as 
a subgroup,
\[
\diag(\Lambda^\alpha_K\times \Lambda^\beta_K) \leqslant H.
\]
Next, we will show the following logical implications and equivalences:

\begin{center}

$(B_n, K)$ is a strong Gelfand pair.\\
$\Downarrow(1)$\\
$(B_{n-1} \times B_1, K)$ is a strong Gelfand pair.\\
$\Downarrow(2)$\\
$(B_{n-1} \times B_1, \Lambda_{K}^{\alpha} \times \Lambda_{K}^{\beta})$ is a strong Gelfand pair.\\
$\Updownarrow(3)$\\
$((B_{n-1} \times B_1) \times (\Lambda_{K}^{\alpha} \times \Lambda_{K}^{\beta}),\diag(\Lambda_{K}^{\alpha} \times \Lambda_{K}^{\beta}))$ is a Gelfand pair.\\
$\Downarrow(4)$\\
$((B_{n-1} \times B_1) \times (\Lambda_{K}^{\alpha} \times \Lambda_{K}^{\beta}),H)$ is a Gelfand pair. \\
$\Updownarrow(5)$\\
$(B_{n-1} \times \Lambda_{K}^{\alpha},\diag(\Lambda_{K}^{\alpha}))$ is a Gelfand pair.\\
$\Updownarrow(6)$\\
$(B_{n-1},\Lambda_{K}^{\alpha})$ is a strong Gelfand pair.

\end{center}
The equivalences $(3)$ and $(6)$ hold by Lemma~\ref{L:sgiff}.
The implications $(1)$ and $(2)$ hold since we have the subgroup inclusions, 
$K \leqslant \Lambda_{K}^{\alpha} \times \Lambda_{K}^{\beta} \leqslant B_{n-1} \times B_1\leqslant B_n$.
Likewise, (4) holds since we have $\diag(\Lambda_{K}^{\alpha} \times \Lambda_{K}^{\beta}) \leqslant H$.
To prove (5) we will show that the pair 
$(B_{n-1} \times \Lambda_{K}^{\alpha},\diag(\Lambda_{K}^{\alpha}))$ is obtained from 
the pair $((B_{n-1} \times B_1) \times (\Lambda_{K}^{\alpha} \times \Lambda_{K}^{\beta}),H)$ by a quotient construction. 
To this end, we define the map 
\begin{align*}
\varphi : (B_{n-1} \times B_1) \times (\Lambda_{K}^{\alpha} \times \Lambda_{K}^{\beta}) &\longrightarrow B_{n-1} \times \Lambda_{K}^{\alpha} \\
(a,b,c,d) &\longmapsto (a,c).
\end{align*}
Then we have $\ker \varphi =\{ (\id,b,\id,d):\ b\in B_1,\ d\in \Lambda^\beta_K \} \leqslant H$.
Clearly, $\varphi$ is surjective. Now (5) follows from Remark~\ref{R:iff1}.
This completes the proof.
\end{proof}

\begin{cor}\label{C:reductionII}
Let $n\geq 7$. 
We maintain the notation of Theorem~\ref{T:reductionII}. 
If $K$ is a strong Gelfand subgroup of $B_n$ such that $\gamma_K = S_{n-1} \times S_1$, then, up to a conjugation by an element of $B_n$, 
$\Lambda^\alpha_K$ is one of the subgroups $B_{n-1}\times \{\id\}, D_{n-1}\times \{\id\}$, or $H_{n-1}\times \{\id\}$.
\end{cor}

\begin{proof}
By Theorem~\ref{T:reductionII}, 
$\Lambda^\alpha_K$ (resp.~$\Lambda^\beta_K$) is a strong Gelfand subgroup of $B_{n-1}$ (resp.~of $B_1$). 
Since $\gamma_{\Lambda^\alpha_K} \cong S_{n-1}$, for the pair $(B_{n-1},\Lambda^\alpha_K)$, 
we are in the situation of Subsection~\ref{SS:1st}. 
By Proposition~\ref{P:Summary1}, we know that $\Lambda^\alpha_K$ is conjugate-isomorphic to one of the subgroups 
$B_{n-1}\times \{\id\}, D_{n-1}\times \{\id\}$, or $H_{n-1}\times \{\id\}$ in $B_n$ if $n\geq 7$.
\end{proof}

\begin{lem}\label{L:11cases}
Let $n\geq 7$.
If $K$ be a strong Gelfand subgroup of $B_n$ such that $\gamma_K = S_{n-1}\times S_1$, 
then $K$ is conjugate to one of the following subgroups: 
\begin{enumerate}
\item $K= B_{n-1}\times B_1$,
\item $K=B_{n-1} \times \{\id\}$, 
\item $K=D_{n-1}\times B_1$, 
\item $K=D_{n-1}\times \{\id\}$,
\item $K=H_{n-1}\times B_1$, 
\item $K=H_{n-1}\times \{\id\}$,
\item $(B_{n-1}\times B_1)_{\delta}$,
\item $(B_{n-1}\times B_1)_{\varepsilon \delta}$,
\item $(B_{n-1}\times B_1)_{\varepsilon }$,
\item $(D_{n-1}\times B_1)_{\varepsilon \delta}$, 
\item $(H_{n-1}\times B_1)_{\delta}$. 
\end{enumerate}
\end{lem}

\begin{proof}
We already know that $B_{n-1}\times B_1$ is a strong Gelfand subgroup of $B_n$, so, let us assume that $K\neq  B_{n-1}\times B_1$. 
Recall that $K$ is an index 1 or 2 subgroup of $\Lambda^\alpha_K\Lambda^\beta_K$. 
If $\Lambda^\beta_K = \{ \id_K\}$, then by Corollary~\ref{C:reductionII} $K$ is as in 2., 4., or 6.
We proceed with the assumption that $\Lambda^\beta_K \neq \{ \id_K \}$.
In this case $K$ can be a subgroup of the form $K'\times \Lambda^\beta_K$, where $K'$ is an index 2 subgroup of $\Lambda^\alpha_K$. 
However, in this case, $\Lambda^\alpha_K$ can only be $B_{n-1}$; otherwise, if $\Lambda^\alpha_K = D_{n-1}$ or $H_{n-1}$, we would have that $K' = J_{n-1}$ and thus that, $\gamma_{K'} = A_{n-1}$, which would contradict 
our assumption that $\gamma_K = S_{n-1}\times S_1$. 
Thus, once again by Corollary~\ref{C:reductionII}, $K$ can be as in 3. or 5.
These options for $K$ are the obvious options. 
For the non-obvious index 2 subgroups of $\Lambda^\alpha_K \Lambda^\beta_K$, 
we apply our discussion from the beginning of this subsection.

The remaining possibilities for $K$ are given by the non-obvious index 2 subgroups of $G\times B_1$, 
where $G \in \{B_{n-1}, D_{n-1},H_{n-1}\}$. 
We already encountered them in Examples~\ref{E:initialexample},~\ref{E:initialexample2}, and~\ref{E:initialexample3}. 
They account for the possibilities that we listed in the items 7., 8., 9. for $G=B_{n-1}$; 10. for $G=D_{n-1}$; and 11. for $G=H_{n-1}$. 
This finishes the proof of our lemma. 
\end{proof}

We now proceed to check the strong Gelfand property of the subgroups that we listed in Lemma~\ref{L:11cases}.
We will make use of several elementary results from Subsection~\ref{SS:Appendix}.

\subsubsection{$(B_n, B_{n-1}\times B_1)$ and $(B_n, B_{n-1}\times S_1)$.}

We already showed that the pairs $(B_n, B_{n-1}\times B_1)$ and $(B_n, B_{n-1}\times S_1)$ are strong Gelfand pairs.

\subsubsection{$(B_n, D_{n-1}\times B_1)$ and $(B_n, D_{n-1}\times S_1)$.}

By the discussion in Section~\ref{associators},
Every irreducible representation of $D_{n-1}$ is either equal to 
$\res^{B_{n-1}}_{D_{n-1}} S^{\lambda,\mu}$, where $\lambda$ and $\mu$ are two distinct partitions with $|\lambda | + |\mu| = n-1$, or it is one of the two irreducible constituents of $\res^{B_{n-1}}_{D_{n-1}} S^{\lambda,\lambda}$, where $2|\lambda| = n-1$. 
By Frobenius reciprocity, for every irreducible representation $V$ of $D_{n-1}$, we have exactly one of the following two cases: 
\begin{enumerate}
\item $\ind^{B_{n-1}}_{D_{n-1}} V = S^{\lambda,\lambda}$ is an irreducible representation of $B_{n-1}$ if $V$ is one of the two irreducible constituents of 
$\res^{B_{n-1}}_{D_{n-1}} S^{\lambda,\lambda}$ for some partition $\lambda$ such that $2|\lambda|= n-1$;
\item $\ind^{B_{n-1}}_{D_{n-1}} V = S^{\lambda,\mu} \oplus S^{\mu,\lambda}$ if $V=\res^{B_{n-1}}_{D_{n-1}} S^{\lambda,\mu}$, where 
$\lambda$ and $\mu$ are distinct partitions with $|\lambda | + |\mu| = n-1$.
\end{enumerate}

We will analyze the induced representations $\ind^{B_{n}}_{D_{n-1}\times B_1} V\boxtimes \mathbf{1}_{B_1}$. 
By the transitivity of the induction, we have 
\begin{align}\label{A:BnDn-1}
\ind^{B_{n}}_{D_{n-1}\times B_1} V\boxtimes  \mathbf{1}_{B_1} &= \ind^{B_{n}}_{B_{n-1}\times B_1}
\ind^{B_{n-1}\times B_1}_{D_{n-1}\times B_1} V\boxtimes  \mathbf{1}_{B_1}.
\end{align}
If $V$ is as in item 1., then (\ref{A:BnDn-1}) gives 
$\ind^{B_{n}}_{D_{n-1}\times B_1} V\boxtimes  \mathbf{1}_{B_1} = \ind^{B_{n}}_{B_{n-1}\times B_1}
S^{\lambda, \lambda} \boxtimes  \mathbf{1}_{B_1}$.
Since $B_{n-1}\times B_1$ is a strong Gelfand subgroup of $B_n$, the resulting induced representation is multiplicity-free. 
Likewise, if $V$ is as in item 2., then (\ref{A:BnDn-1}) together with part 1. of Lemma~\ref{L:Branching} give 
\begin{align}
\ind^{B_{n}}_{D_{n-1}\times B_1} V\boxtimes  \mathbf{1}_{B_1} &= \ind^{B_{n}}_{B_{n-1}\times B_1}
S^{\lambda,\mu}\boxtimes  \mathbf{1}_{B_1} \oplus \ind^{B_{n}}_{B_{n-1}\times B_1} S^{\mu,\lambda} \boxtimes  \mathbf{1}_{B_1} \notag \\
&= \bigoplus_{\tau \in \overline{\lambda}} S^{\tau, \mu}  \oplus  \bigoplus_{\rho \in \overline{\mu}} S^{\rho,\lambda}. \label{A:BnDn-2}
\end{align}
In this case also, since $\lambda \neq \mu$, we see that the irreducible representations of $B_n$ that appear in 
(\ref{A:BnDn-2}) are inequivalent. 
Thus, we proved that $(B_n, D_{n-1}\times B_1)$ is a strong Gelfand pair.

We now analyze the pair $(B_n, D_{n-1}\times S_1)$. 
Once again, by the transitivity of the induction, we have 
\begin{align}\label{A:BnDn-3}
\ind^{B_{n}}_{D_{n-1}\times S_1} V\boxtimes  \mathbf{1}_{S_1} &= \ind^{B_{n}}_{B_{n-1}\times S_1}
\ind^{B_{n-1}\times S_1}_{D_{n-1}\times S_1} V\boxtimes  \mathbf{1}_{S_1}.
\end{align}
If $V$ is as in item 2.~above, then (\ref{A:BnDn-3}) and part 3. of Lemma~\ref{L:Branching} give
\begin{align}
\ind^{B_{n}}_{D_{n-1}\times S_1} V\boxtimes  \mathbf{1}_{S_1} &= \ind^{B_{n}}_{B_{n-1}\times S_1}
S^{\lambda,\mu}\boxtimes  \mathbf{1}_{S_1} \oplus \ind^{B_{n}}_{B_{n-1}\times S_1} S^{\mu,\lambda} \boxtimes  \mathbf{1}_{S_1} \notag \\
&= \left(\bigoplus_{\tau \in \overline{\lambda}} S^{\tau, \mu}  \oplus  \bigoplus_{\rho \in \overline{\mu}} S^{\lambda,\rho}  \right)
\oplus \left(\bigoplus_{\rho \in \overline{\mu}} S^{\rho, \lambda}  \oplus  \bigoplus_{\tau \in \overline{\lambda}} S^{\mu,\tau}  \right).
\label{A:BnDn-1S1}
\end{align}
Here, $\lambda$ and $\mu$ are two distinct partitions such that $|\lambda|+|\mu | =n-1$.
If $n-1 = 2m+1$ for some $m\in \N$, then we consider the partitions $\lambda = (m)$ and $\mu = (m+1)$. 
It is easily checked that the multiplicity of $S^{(m+1),(m+1)}$ in (\ref{A:BnDn-1S1}) is 2. 
On the other hand, if $n-1 = 2m$ for some $m\in \N$, then it is easy to check that (\ref{A:BnDn-1S1}) is a multiplicity-free $B_n$ representation. 

Next, we consider the  irreducible representations $V$ of $D_{n-1}$ as in item 1. 
In particular, $n-1$ is even. 
Let $\lambda$ be a partition such that $2|\lambda| = n-1$, let $S^{\lambda, \lambda}$ denote the corresponding irreducible representation of $B_{n-1}$.
By part 3. of Lemma~\ref{L:Branching}, we get 
\begin{align}
\ind^{B_{n}}_{D_{n-1}\times S_1} V\boxtimes  \mathbf{1}_{S_1} &= \ind^{B_{n}}_{B_{n-1}\times S_1}
S^{\lambda,\lambda}\boxtimes  \mathbf{1}_{S_1} \notag \\
&= \bigoplus_{\tau \in \overline{\lambda}} S^{\tau, \lambda}  \oplus  \bigoplus_{\rho \in \overline{\lambda}} S^{\lambda,\rho}.
\label{A:BnDn-4}
\end{align}
Clearly, the irreducible constituents of (\ref{A:BnDn-4}) are inequivalent, hence, $\ind^{B_{n}}_{D_{n-1}\times S_1} V\boxtimes  \mathbf{1}_{S_1} $
is a multiplicity-free representation of $B_n$. 
In summary, we proved that $(B_n, D_{n-1}\times S_1)$ is a strong Gelfand pair if and only if $n$ is odd.

\subsubsection{$(B_n, H_{n-1}\times B_1)$ and $(B_n, H_{n-1}\times S_1)$.}

Next, we proceed to analyze the pair $(B_n, H_{n-1}\times B_1)$. 
Since $[B_{n-1}:H_{n-1}]=2$, the irreducible representations of $H_{n-1}$ are described by Clifford theory (c.f.~Section~\ref{associators}):
1.~$S^{\lambda,\mu}$ is self-associate with respect to $\varepsilon\delta$ if and only if $\lambda=\mu'$, in which case $\res^{B_{n-1}}_{H_{n-1}} S^{\lambda,\mu}$ is the direct sum of two irreducible $H_{n-1}$ representations of the same degree. 
2.~If $S^{\lambda,\mu}$ and $\varepsilon\delta S^{\lambda,\mu}$ are associate representations with respect 
to $\varepsilon\delta$, then $\res^{B_{n-1}}_{H_{n-1}} S^{\lambda,\mu}$ is an irreducible representation of $H_{n-1}$. 
By Frobenius reciprocity, if $V$ is an irreducible representation of $H_{n-1}$, then:

\begin{enumerate}
\item $\ind^{B_{n-1}}_{H_{n-1}} V = S^{\lambda,\lambda'}$ is an irreducible representation of $B_{n-1}$ 
if $V$ is one of the two irreducible components of $\res^{B_{n-1}}_{H_{n-1}} S^{\lambda,\lambda'}$ for some partition $\lambda$ of $n-1$;
\item $\ind^{B_{n-1}}_{H_{n-1}} V = S^{\lambda,\mu} \oplus S^{\mu',\lambda'}$ if $V=\res^{B_{n-1}}_{H_{n-1}} S^{\lambda,\mu}$, 
where 
$\lambda \neq \mu'$ and $|\lambda | + |\mu| = n-1$.
\end{enumerate}
Now let $V$ be an irreducible representation of $H_{n-1}$ as in 1. 
Since $\ind^{B_{n-1}}_{H_{n-1}} V$ is an irreducible representation of $B_{n-1}$, 
$\ind^{B_n}_{H_{n-1}\times B_1} V\boxtimes \mathbf{1}_{B_1}$ is a multiplicity-free representation of $B_n$. For $V$ as in 2., we get 
\begin{align}
\ind^{B_{n}}_{H_{n-1}\times B_1} V\boxtimes  \mathbf{1}_{B_1} &= \ind^{B_{n}}_{B_{n-1}\times B_1}
S^{\lambda,\mu}\boxtimes  \mathbf{1}_{B_1} \oplus \ind^{B_{n}}_{B_{n-1}\times B_1} S^{\mu',\lambda'} \boxtimes  \mathbf{1}_{B_1} \notag \\
&= \bigoplus_{\tau \in \overline{\lambda}} S^{\tau, \mu}  \oplus  \bigoplus_{\rho \in \overline{\mu'}} S^{\rho',\lambda'}. \label{A:BnHn-2}
\end{align}
In this case also, it is easy to verify that (\ref{A:BnHn-2}) is a multiplicity-free representation of $B_n$. 
Therefore, $(B_n, H_{n-1}\times B_1)$ is a strong Gelfand pair.

We now proceed to the case of $(B_n, H_{n-1}\times S_1)$. 
In this case, if $V$ is as in 2., then by part 3. of Lemma~\ref{L:Branching} we get 
\begin{align}
\ind^{B_{n}}_{H_{n-1}\times S_1} V\boxtimes  \mathbf{1}_{S_1} &= 
\left( \ind^{B_{n}}_{B_{n-1}\times S_1}
S^{\lambda,\mu}\boxtimes  \mathbf{1}_{S_1} \right) \oplus \left( \ind^{B_{n}}_{B_{n-1}\times S_1} S^{\mu',\lambda'} \boxtimes  \mathbf{1}_{S_1} \right) \notag \\
&= \left(\bigoplus_{\tau \in \overline{\lambda}} S^{\tau, \mu}  \oplus  \bigoplus_{\rho \in \overline{\mu}} S^{\lambda,\rho}  \right)
\oplus \left(\bigoplus_{\rho \in \overline{\mu'}} S^{\rho, \lambda'}  \oplus  \bigoplus_{\tau \in \overline{\lambda'}} S^{\mu',\tau}  \right).
\label{A:BnHn-3}
\end{align}

Let $n-1 = 2m+1$ for some $m\in \N$, and set $\lambda = (m+1)$ and $\mu = (1^{m})$. 
Then by Pieri's rule we see that the multiplicity of $S^{(m+1),(1^{m+1})}$ in (\ref{A:BnHn-3}) is 2. 
Thus, if $n$ is even, then $(B_n, H_{n-1}\times S_1)$ is not a strong Gelfand pair. 
If $n$ is odd, then it is easy to check that the induced representation (\ref{A:BnHn-3}) is a multiplicity-free $B_n$ representation.

Next, we consider the irreducible representations $V$ of $H_{n-1}$ as in 1. 
Then $n$ is odd. 
By part 3. of Lemma~\ref{L:Branching}, we get 
\begin{align}
\ind^{B_{n}}_{H_{n-1}\times S_1} V\boxtimes  \mathbf{1}_{S_1} &= \ind^{B_{n}}_{B_{n-1}\times S_1}
S^{\lambda,\lambda'}\boxtimes  \mathbf{1}_{S_1} \notag \\
&= \bigoplus_{\tau \in \overline{\lambda}} S^{\tau, \lambda'}  \oplus  \bigoplus_{\rho \in \overline{\lambda'}} S^{\lambda,\rho}.
\label{A:BnHn-4}
\end{align}
Clearly, the irreducible constituents of (\ref{A:BnHn-4}) are inequivalent, and hence $\ind^{B_{n}}_{H_{n-1}\times S_1} V\boxtimes  \mathbf{1}_{S_1} $
is multiplicity-free. 
In summary, we proved that $(B_n, H_{n-1}\times S_1)$ is a strong Gelfand pair if and only if $n$ is odd.

\subsubsection{$(B_n, (B_{n-1}\times B_1)_{\delta})$ and $(B_n, (B_{n-1}\times B_1)_{\varepsilon \delta})$.}

We start with the case $(B_n,(B_{n-1}\times B_1)_{\delta })$.
To ease our notation, let us denote $(B_{n-1}\times B_1)_{\delta}$ by $M$. 
Let $\nu$ denote the linear character of $B_{n-1}\times B_1$ such that $\ker \nu= M$.
Then it is easy to check that $\nu |_{B_{n-1}\times \{1\}} = \delta_{B_{n-1}}$ and $\nu |_{\{1\} \times B_1}=\delta_{B_1}$.
(Indeed, $D_{n-1}\times \{1\}$ is contained in the kernel of $\nu$.)

Let $W = S^{\lambda,\mu} \boxtimes (D;\mathbf{1})$ be an irreducible representation of $B_{n-1}\times B_1$.
There are two possibilities: 
1) $W$ is a self-associate representation with respect to $\nu$, or 
2) $W$ and $\nu W$ are associate representations with respect to $\nu$. 
However, since $\delta_{B_{n-1}} S^{\lambda, \mu}= S^{\mu, \lambda}$ and $\delta_{B_1}^2 = \mathbf{1}_{B_1}$, we have 
\[
\nu ( S^{\lambda,\mu} \boxtimes (D;\mathbf{1})) = S^{\mu, \lambda}\boxtimes (\tilde{D};\mathbf{1}),
\]
where $\{ D,\tilde{D} \} = \{\mathbf{1},\mathbf{\epsilon}\}$.
Since $\{ D,\tilde{D} \} = \{\mathbf{1},\mathbf{\epsilon}\}$, the representations $S^{\lambda,\mu} \boxtimes (D;\mathbf{1})$
and $S^{\mu, \lambda}\boxtimes (\tilde{D};\mathbf{1})$ are inequivalent. 
Hence, we conclude that there is no self-associate irreducible representation with respect to $\nu$. 
Now let $V$ be an irreducible representation of $M$.
Then by Frobenius reciprocity we have 
\[
\ind^{B_{n-1}\times B_1}_M V = S^{\lambda,\mu}\boxtimes (\mathbf{1};\mathbf{1}) \oplus S^{\mu, \lambda}\boxtimes (\mathbf{\epsilon};\mathbf{1})
\]
for some irreducible representation $S^{\lambda,\mu}$ of $B_{n-1}$. 
By Lemma~\ref{L:Branching}, $\ind^{B_n}_M V$ must be 
\begin{align}\label{A:bothBD}
\bigoplus_{\tau \in \overline{\lambda}} S^{\tau, \mu}  \oplus  \bigoplus_{\rho \in \overline{\lambda}} S^{\mu,\rho}.
\end{align}
If $n-1=2m+1$ for some $m\in \N$, then we fix a pair of partitions $(\lambda,\mu)$ such that $\lambda$ is obtained from $\mu$ by removing a box from its Young diagram. 
Then it is easy to check that $S^{\mu,\mu}$ appears twice (\ref{A:bothBD}).
If $n$ is odd, then it is easy to check that (\ref{A:bothBD}) is multiplicity-free for any $\lambda$ and $\mu$.
Therefore, we proved that $(B_n,(B_{n-1}\times B_1)_{\delta})$ is a strong Gelfand pair if and only if $n$ is odd.

Next, we consider the pair $(B_n,(B_{n-1}\times B_1)_{\varepsilon \delta})$.
To ease notation, we denote $(B_{n-1}\times B_1)_{\varepsilon \delta}$ by $N$. 
We know that $H_{n-1}\times \{1\}$ is an index 2 subgroup of $N$, and $N$ is an index 2 subgroup of $B_{n-1}\times B_1$. 
We will describe the irreducible representations of $N$.
Let $\nu$ denote the linear character of $B_{n-1}\times B_1$ such that $\ker \nu= N$.
Then the restrictions of $\nu$ to the factors are given by $\nu |_{B_{n-1}\times \{1\}} = \varepsilon\delta$ and $\nu |_{\{1\} \times B_1}=\delta_{B_1}$.
Let $W = S^{\lambda,\mu} \boxtimes (D;\mathbf{1})$ be an irreducible representation of $B_{n-1}\times B_1$.
We have two possibilities here: 
1) $W$ is a self-associate representation with respect to $\nu$, 
2) $W$ and $\nu W$ are associate representations with respect to $\nu$. 
However, since $\delta S^{\lambda, \mu}= S^{\mu', \lambda'}$ 
and $\delta_{B_1}^2 = \mathbf{1}_{B_1}$, we have 
\[
\nu ( S^{\lambda,\mu} \boxtimes (D;\mathbf{1})) = S^{\mu', \lambda'}\boxtimes (\tilde{D};\mathbf{1}),
\]
where $\{ D,\tilde{D} \} = \{\mathbf{1},\mathbf{\epsilon}\}$.
Since $\{ D,\tilde{D} \} = \{\mathbf{1},\mathbf{\epsilon}\}$, the representations $S^{\lambda,\mu} \boxtimes (D;\mathbf{1})$
and $S^{\mu', \lambda'}\boxtimes (\tilde{D};\mathbf{1})$ are inequivalent. 
Thus, similarly to the previous case, there is no self-associate irreducible representation with respect to $\nu$. 
Now let $V$ be an irreducible representation of $N$.
Then we have
\[
\ind^{B_{n-1}\times B_1}_N V = S^{\lambda,\mu}\boxtimes (\mathbf{1};\mathbf{1}) \oplus S^{\mu', \lambda'}\boxtimes (\mathbf{\epsilon};\mathbf{1})
\]
for some irreducible representation $S^{\lambda,\mu}$ of $B_{n-1}$. 
By Lemma~\ref{L:Branching}, $\ind^{B_n}_N V$ must be
\begin{align}\label{A:bothBH}
\bigoplus_{\tau \in \overline{\lambda}} S^{\tau, \mu}  \oplus  \bigoplus_{\rho \in \overline{\lambda'}} S^{\mu',\rho}.
\end{align}
If $n-1=2m+1$ for some $m\in \N$, then we fix a pair of partitions $(\lambda,\mu)$ such that $\lambda$ is obtained from $\mu'$ by removing a box from its Young diagram. 
Then we see that the multiplicity of $S^{\mu',\mu}$ in (\ref{A:bothBH}) is 2. 
If $n$ is odd, then it is easy to check that (\ref{A:bothBH}) is multiplicity-free for any $\lambda$ and $\mu$.
Therefore, we proved that $(B_n,(B_{n-1}\times B_1)_{\varepsilon \delta})$ is a strong Gelfand pair if and only if $n$ is odd.

\subsubsection{$(B_n, (B_{n-1}\times B_1)_{\varepsilon })$.}\label{S:case9} 

To ease our notation, let us denote $(B_{n-1}\times B_1)_{\varepsilon}$ by $Z$. 
We know that $(F\wr A_{n-1})\times \{1\} $ is an index 2 subgroup of $Z$, and $Z$ is an index 2 subgroup of $B_{n-1}\times B_1$. 
We will describe the irreducible representations of $Z$.
Let $\nu$ denote the linear character of $B_{n-1}\times B_1$ such that $\ker \nu= Z$.
Then the restrictions of $\nu$ to the factors are given by $\nu |_{B_{n-1}\times \{\id\}} = \varepsilon_{B_{n-1}}$ 
and $\nu |_{\{\id\} \times B_1}=\delta_{B_1}$.

Let $W = S^{\lambda,\mu} \boxtimes (D;\mathbf{1})$ be an irreducible representation of $B_{n-1}\times B_1$.
We have two possibilities: 
1) $W$ is a self-associate representation with respect to $\nu$, 
2) $W$ and $\nu W$ are associate representations with respect to $\nu$. 
But since there is no self-associate irreducible representation of $B_1$ with respect to $\delta_{B_1}$, there is no self-associate irreducible representation of $B_{n-1}\times B_1$ with respect to $\nu$.

Let $V$ be an irreducible representation of $Z$. 
Then for some irreducible representation $S^{\lambda,\mu}$ of $B_{n-1}$, we have 
\[
\ind^{B_{n-1}\times B_1}_Z V = S^{\lambda,\mu}\boxtimes (\mathbf{1};\mathbf{1}) \oplus S^{\lambda', \mu'}\boxtimes (\mathbf{\epsilon};\mathbf{1}).
\]
By Lemma~\ref{L:Branching}, $\ind^{B_n}_Z V$ is
\begin{align}\label{A:bothBZ}
\bigoplus_{\tau \in \overline{\lambda}} S^{\tau, \mu}  \oplus  \bigoplus_{\rho \in \overline{\mu'}} S^{\lambda',\rho}.
\end{align}
It is easy to see that this representation is multiplicity-free.
Thus, $(B_n, Z)$ is a strong Gelfand pair.

\subsubsection{$(B_n, (D_{n-1}\times B_1)_{\varepsilon \delta})$ and $(B_n,(H_{n-1}\times B_1)_{\delta})$.}

To ease our notation, let us denote $(D_{n-1}\times B_1)_{\varepsilon \delta }$ by $K$. 
Let $\nu$ denote the linear character of $D_{n-1}\times B_1$ such that $\ker \nu= K$.
Since $K\leqslant D_{n-1}\times F$, the restrictions of $\nu$ to the factors are given by $\nu |_{D_{n-1}\times \{\id\}} = (\varepsilon \delta)_{B_{n-1}}$ 
and $\nu |_{\{\id\} \times B_1}= (\varepsilon \delta)_{B_1}= \delta_{B_1}$.

Let $U\boxtimes (D; \mathbf{1})$ be an irreducible representation of $D_{n-1}\times B_1$, where $D\in \{\mathbf{1}_F, \epsilon_F\}$. 
Since $(\mathbf{1}_F;\mathbf{1})$ and $(\epsilon_F;\mathbf{1})$ are $\delta_{B_1}$-associate representations, 
there are no self-associate representations of $D_{n-1}\times B_1$ with respect to $\nu$. 
In particular, every irreducible representation of $D_{n-1}\times B_1$ restricts to $K$ as an irreducible representation.
Thus, if $V$ is an irreducible representation of $K$, then 
$\ind^{D_{n-1}\times B_1}_K V$ is of the form $U\boxtimes (\mathbf{1}_F; \mathbf{1}) \oplus (\varepsilon \delta U)\boxtimes (\epsilon_F; \mathbf{1})$.
Let $S^{\lambda,\mu}$ be the irreducible representation of $B_{n-1}$ such that $\res^{B_{n-1}}_{D_{n-1}} S^{\lambda, \mu} = U$,
where $\lambda$ and $\mu$ are two partitions such that $|\lambda | + | \mu | = n-1$. 
First we assume that $\lambda \neq \mu$.
Then we have 
\begin{align*}
\ind^{B_n}_K V & = \ind^{B_n}_{B_{n-1}\times B_1} \ind^{B_{n-1}\times B_1}_K V \\
&=  \ind^{B_n}_{B_{n-1}\times B_1}  \ind^{B_{n-1}\times B_1} _{D_{n-1}\times B_1} ( U\boxtimes (\mathbf{1}_F; \mathbf{1}) \oplus (\varepsilon \delta U)\boxtimes (\epsilon_F; \mathbf{1}) ) \\
&=  \ind^{B_n}_{B_{n-1}\times B_1} ( 
S^{\lambda,\mu} \boxtimes (\mathbf{1}_F; \mathbf{1}) \oplus S^{\mu,\lambda} \boxtimes (\mathbf{1}_F; \mathbf{1}) 
\oplus S^{\lambda',\mu'} \boxtimes (\epsilon_F; \mathbf{1}) 
\oplus S^{\mu',\lambda'}\boxtimes (\epsilon_F; \mathbf{1}) ).
\end{align*}

Then by Lemma~\ref{L:Branching}, we find that 
\begin{align}\label{A:bothBU}
\ind^{B_n}_K V = \bigoplus_{\tau \in \overline{\lambda}} S^{\tau, \mu}  \oplus 
 \bigoplus_{\tau \in \overline{\mu}} S^{\tau, \lambda}  \oplus  
 \bigoplus_{\rho \in \overline{\mu'}} S^{\lambda',\rho} \oplus
 \bigoplus_{\rho \in \overline{\lambda'}} S^{\mu',\rho}.
\end{align}
It is easy to check that this is a multiplicity-free representation of $B_n$ if $n-1$ is even.
If $n-1$ is odd, then we choose $\lambda= (m)$ and $\mu= 1^{m+1}$.
Then $S^{(m+1),(1^{m+1})}$ appears with multiplicity 2 in (\ref{A:bothBU}).
Next, we assume that $\lambda = \mu$. Of course, this choice is available only when $n-1$ is even. 
Then we have 
\begin{align*}
\ind^{B_n}_K V & = \ind^{B_n}_{B_{n-1}\times B_1} \ind^{B_{n-1}\times B_1}_K V \\
&=  \ind^{B_n}_{B_{n-1}\times B_1}  \ind^{B_{n-1}\times B_1} _{D_{n-1}\times B_1} ( U\boxtimes (\mathbf{1}_F; \mathbf{1}) \oplus (\varepsilon \delta U)\boxtimes (\epsilon_F; \mathbf{1}) ) \\
&=  \ind^{B_n}_{B_{n-1}\times B_1} ( 
S^{\lambda,\lambda} \boxtimes (\mathbf{1}_F; \mathbf{1})
\oplus S^{\lambda',\lambda'}\boxtimes (\epsilon_F; \mathbf{1}) ).
\end{align*}
Clearly, this representation is multiplicity-free if and only if $\lambda \neq \lambda '$; we can find self-conjugate partitions of $n-1$ as long as 
$n-1> 2$. 
Therefore, if $n\geq3$, $(B_n, K)$ is a strong Gelfand pair if and only if $n$ is odd.

By a similar argument, one can also deduce that, if $n\geq 3$, $(B_n,(H_{n-1}\times B_1)_{\delta})$ is a strong Gelfand pair if and only if $n$ is odd.

\subsubsection{Summary for $\gamma_K = S_{n-1}\times S_1$.}

We now summarize the conclusions of the previous subsections in a single proposition. 
In particular, we maintain our notation from Lemma~\ref{L:11cases}. 

\begin{prop}\label{P:Summary3}
Let $n\geq 7$.
Let $K$ be a subgroup of $B_n$ such that $\gamma_K=S_{n-1}\times S_1$.
In this case, $(B_n,K)$ is a strong Gelfand pair if and only if $K$ is one of the following subgroups:
\begin{enumerate}
\item $K= B_{n-1}\times B_1$,
\item $K=B_{n-1} \times \{\id\}$, 
\item $K=D_{n-1}\times B_1$, 
\item $K=D_{n-1}\times \{\id\}$, if $n$ is odd,
\item $K=H_{n-1}\times B_1$, 
\item $K=H_{n-1}\times \{\id\}$, if $n$ is odd,

\item $(B_{n-1}\times B_1)_{\delta}$, if $n$ is odd,
\item $(B_{n-1}\times B_1)_{\varepsilon \delta}$, if $n$ is odd, 
\item $(B_{n-1}\times B_1)_{\varepsilon }$.
\item $(D_{n-1}\times B_1)_{\varepsilon \delta}$, if $n$ is odd,
\item $(H_{n-1}\times B_1)_{\delta}$, if $n$ is odd. 
\end{enumerate}
\end{prop}

\subsection{$\gamma_K = S_{n-2}\times S_2$.}\label{SS:4th}

Throughout this subsection, we will assume that $n\geq 8$, $K$ will be a subgroup of $B_n$ such that $\gamma_K = S_{n-2}\times S_2$, and
hence, $K\leqslant F \wr (S_{n-2}\times S_2)$.
The idea of our analysis in this subsection is the same as the one that we used in the previous subsection.

Let $\phi' : F\wr (S_{n-2}\times S_2) \to B_{n-2}\times B_2$ be the canonical splitting isomorphism.
For $\lambda \in F\wr (S_{n-2}\times S_2)$ and $\phi'(\lambda) = (a,b) \in B_{n-2}\times B_2$,
we denote by $\lambda_\alpha$ the element in $F\wr (S_{n-2}\times S_2)$ such that $\phi' (\lambda_\alpha) = (a,\id_{B_2})$.
Likewise, we denote by $\lambda_\beta$ the element in $F\wr (S_{n-2}\times S_2)$ such that $\phi' (\lambda_\beta) = (\id_{B_{n-2}},b)$. 
In this notation, we have $\lambda = \lambda_\alpha \lambda_\beta$.
As before, we have the following ``$K$-related'' subgroups of $B_{n-2}\times B_2$:
\begin{align}\label{A:newlambdas}
\Lambda^\alpha_K  := \{ \lambda_\alpha :\ \lambda \in K \} 
 \qquad \text{and}\qquad 
\Lambda^\beta_K  := \{ \lambda_\beta :\ \lambda \in K\}.
\end{align} 
It is easy to show that 
$K \leqslant \Lambda^\alpha_K \Lambda^\beta_K \cong \Lambda^\alpha_K \times \Lambda^\beta_K$, 
$\gamma_{\Lambda^\alpha_K} \cong S_{n-2}$, and that $\gamma_{\Lambda^\beta_K} \cong S_2$.
Hereafter, when it is convenient for our purposes, we will identify $\Lambda^\alpha_K$ with its isomorphic copy in $B_{n-2}$, 
and $\Lambda^\beta_K$ with its isomorphic copy in $B_2$. 
\begin{Remark}
It is worth noting here that if $K$ is a direct product subgroup of $\Lambda^\alpha_K\Lambda^\beta_K$, then 
$K=\Lambda^\alpha_K \Lambda^\beta_K$. 
\end{Remark}
The proof of our second reduction theorem is adaptable to the subgroups $\Lambda^\alpha_K$ and $\Lambda^\beta_K$ defined in (\ref{A:newlambdas}).

\begin{thm}\label{T:reductionIII}
If $(B_n, K)$ is a strong Gelfand pair, then 
so is $(B_{n-2}, \Lambda_K^\alpha)$.
\end{thm}

Also, the proof of the following corollary is similar to that of Corollary~\ref{C:reductionII}.

\begin{cor}\label{C:reductionIII}
If $(B_n, K)$ is a strong Gelfand pair, 
then $\Lambda^\alpha_K$ is one of the following subgroups: $B_{n-2}\times \{\id_{B_2}\}, D_{n-2}\times \{\id_{B_2}\}$, or $H_{n-2}\times \{\id_{B_2}\}$.
\end{cor}

In the following lemma, which is easy to verify, we examine all subgroups of $B_2$, thus giving us all possibilities for what the subgroup $\Lambda_K^\beta$ could be.

\begin{lem}\label{L:sgsB2}
If $G$ is a strong Gelfand subgroup of $B_2$ such that $\gamma_G = S_2$, 
then $G$ is one of the following subgroups:
\begin{enumerate}
\item $B_2$,
\item $D_2=\{ ((0,0),\id_{S_2}), ((1,1),\id_{S_2}),((0,0), (1\, 2)),((1,1), (1\,2))\}$,
\item $H_2=\{ ((0,0),\id_{S_2}), ((1,1),\id_{S_2}),((1,0), (1\, 2)),((0,1), (1\,2))\}$,
\item $\overline{S_2}:=\{ ((0,0),\id_{S_2}),((0,0), (1\, 2))\}$ or $\overline{S_2}':=\{ ((0,0),\id_{S_2}),((1,1), (1\, 2))\} = x \overline{S_2} x^{-1}$, where $x= ((0,1),(1,2))$.
\end{enumerate}
We have 3 more strong Gelfand subgroups $G$ with $\gamma_G = \{\id_{S_2}\}$:
\begin{enumerate}[resume]
\item $F\times 0 = \{ ((0,0),\id_{S_2}),((1,0), \id_{S_2})\}$, its conjugate $0 \times F = \{ ((0,0),\id_{S_2}),((0,1), \id_{S_2})\}$, and $F\wr \{\id_{S_2} \}$, with $\gamma_{F \times 0} = \gamma_{0\times F} = \gamma_{F\wr \{\id_{S_2} \}} = \{ \id_{S_2}\}$.
\end{enumerate}
Finally, $B_2$ has two more subgroups that are not strong Gelfand, namely the diagonal subgroup $\diag(F)$ and the trivial subgroup.
In both cases, we again have that $\gamma_K = \{\id_{S_2}\}$.
\end{lem}

We now introduce our auxiliary subgroup of $K$ to show that $K$ cannot be too small.

\begin{lem}\label{L:auxiliaryL2}
Let $L$ denote the following subgroup of $\Lambda^\alpha_K$:
\[
L := \{ \lambda_\alpha \in K :\ \lambda_\beta = \id_{\Lambda^\beta_K} \}.
\]
Then $L$ is a normal subgroup of $K$, that is nontrivial if $n\geq 5$.
Furthermore, the following hold:
\begin{enumerate}
\item $[K:L] \leq 8$ and $[\Lambda^\alpha_K\Lambda^\beta_K: K] \leq 8$.
\item If $L=\Lambda^\alpha_K$, then $K=\Lambda^\alpha_K\Lambda^\beta_K$. 
\end{enumerate}
\end{lem}

\begin{proof}
Let $\varDelta_2: K \to \Lambda^\beta_K$ denote the composition of the canonical injection of $K$ into $\Lambda^\alpha_K\times \Lambda^\beta_K$ 
and projection onto the second component. Then $L$ is precisely the kernel of $\varDelta_2$, hence, $L$ is a normal subgroup of $K$.
Next, we will show that $L$ is nontrivial. 
Since $\gamma_K = S_{n-2} \times S_2$, whenever $n\geq 5$, we can choose an element $\lambda$ of order 3. 
Note that the order of an element of $B_2$ is 1,2, or 4.
Therefore, $\lambda^4  = (\lambda_a \lambda_b)^4 = (\lambda_a)^4 (\lambda_b)^4 = (\lambda^4)_a(\lambda^4)_b = (\lambda^4)_a \in K$. 
Since $(\lambda^4)_a \neq \id_{B_{n-2}}$, we see that $L$ is a nontrivial normal subgroup of $K$.

Since $\Lambda^\beta_K$ is a subgroup of $B_2$ and since $B_2$ has order 8, we see that the index of $L$ in $K$ is bounded by 8.
The second bound follows from the following inequalities: 
\[
|\Lambda^\alpha_K | \leq |K| \leq |\Lambda^\alpha_K \Lambda^\beta_K |=|\Lambda^\beta_K | | \Lambda^\alpha_K | \leq 8 |\Lambda^\alpha_K|.
\]
For our final assertion, we observe that if $L=\Lambda^\alpha_K$, then we have $\Lambda^\alpha_K \leqslant K$.
It follows that $\Lambda^\beta_K \leqslant K$, hence that $\Lambda^\beta_K \Lambda^\alpha_K = K$.
This finishes the proof of our lemma.
\end{proof}

\begin{cor}\label{C:possibleL}
Let $L$ be the subgroup of $K$ that is defined in Lemma~\ref{L:auxiliaryL2}.
Then $\gamma_L = A_{n-2}\times \{\id_{S_2}\}$ or $S_{n-2}\times \{\id_{S_2}\}$.
\end{cor}

\begin{proof}
Since $L\leqslant B_{n-2}\times \{ \id_{B_2}\}$, the proof follows from the fact that $\gamma_L \leqslant \gamma_K = S_{n-2}\times S_2$
and that $A_{n-2}$ is the unique nontrivial normal subgroup of $S_{n-2}$.
\end{proof}

In our next lemma, we will narrow the choices for $L$.

\begin{prop}\label{L:thesecondcase}
Let $n\geq 8$ and let $K$ be a strong Gelfand subgroup of $B_n$ such that $\gamma_K = S_{n-2}\times S_2$,
and let $L$ be the subgroup of $K$ that is defined in Lemma~\ref{L:auxiliaryL2}.
Then we have
\begin{enumerate}
\item
$L\in \{ B_{n-2}, D_{n-2},H_{n-2}, F\wr A_{n-2}, J_{n-2} \}$, and in particular, $L$ is a normal subgroup of $B_{n-2}$ such that $[B_{n-2}:L] \leq 4$.

\item $K$ is a normal subgroup of $\Lambda^\alpha_K\Lambda^\beta_K$. 
Furthermore, the quotient group $\Lambda^\alpha_K\Lambda^\beta_K/ K$ is isomorphic to $\Lambda^\alpha_K / L$, and in particular, $[\Lambda^\alpha_K\Lambda^\beta_K:K] \leq 4$.
\end{enumerate}
\end{prop}

\begin{proof}

\begin{enumerate}
\item Let $L$ be the subgroup of $K$ that is defined in Lemma~\ref{L:auxiliaryL2}.
In Subsections~\ref{SS:1st} and~\ref{SS:2nd} we characterized subgroups $K' \leqslant B_{n-2}$ with $\gamma_{K'} \in \{ S_{n-2},A_{n-2}\}$.
In particular, we notice that the index of $K'$ in $B_{n-2}$ is one of $1,2,4,2^{n-3}, 2^{n-2}, 2^{n-1}$.
By \cite[Corollary 1.3]{BridsonMiller} and the argument after it, the index of $L$ in $\Lambda^\alpha_K$ is bounded above by the order of $\Lambda^\beta_K$, which is at most $2^3$.
Combined with the fact that $[B_{n-2} : \Lambda^\alpha_K] \leq 2$, and that $\gamma_L \in \{ S_{n-2},A_{n-2}\}$, we see that $L$ must have index $1$, $2$, or $4$ in $B_{n-2}$.
The result now follows from the description of those subgroups of index $1$, $2$, and $4$ in its statement.
\item Follows from part 1. and \cite[Corollary 1.3]{BridsonMiller}.\qedhere
\end{enumerate}
\end{proof}

In light of this proposition, let us organize the major cases that we will check for the strong Gelfand property. 

\begin{enumerate}

\item[] Case 1.~$K$ is a normal, index 4, non-direct product subgroup of $\Lambda^\alpha_K \times \Lambda^\beta_K$, 
where $\Lambda^\alpha_K = B_{n-2}$ and $\Lambda^\beta_K = B_2$.

\item[] Case 2.~$K$ is a normal, index 2, non-direct product subgroup of $\Lambda^\alpha_K \times \Lambda^\beta_K$, 
where $\Lambda^\alpha_K$ is one of the subgroups in Corollary~\ref{C:reductionIII}
and $\Lambda^\beta_K$ is one of the subgroups in Lemma~\ref{L:sgsB2}.

\item[] Case 3.~$K$ is equal to the direct product $\Lambda^\alpha_K \times \Lambda^\beta_K$, where 
$\Lambda^\alpha_K$ is one of the subgroups in Corollary~\ref{C:reductionIII}
and $\Lambda^\beta_K$ is one of the subgroups in Lemma~\ref{L:sgsB2}.

\end{enumerate}

Note that Cases 1.~and 2.~are not necessarily distinct.
For example, an index 4 subgroup of $B_{n-2}\times B_2$ might be an index 2 subgroup of $D_{n-2}\times B_2$.

We are now ready to determine all strong Gelfand subgroups of $K\leqslant B_n$ such that $\gamma_K = S_{n-2}\times S_2$.
We will first examine subgroups as in Case 3.

\begin{Notation}
For brevity, in the following subsections, we will denote the identity element $(0,\id_{B_2})$ of $B_2$ by 1, 
and we will denote the element $(0,(1\,2))$ by $-1$. 
\end{Notation}

\subsubsection{Strong Gelfand pairs of the form $(B_n,B_{n-2}\times G)$, where $G\leqslant B_2$.}

By Proposition~\ref{P:secondcase} we know that $(B_n,B_{n-2}\times \overline{S_2})$ is a strong Gelfand pair.
Therefore, for any $G$ such that $\overline{S_2}\leqslant G \leqslant B_2$, the pair $(B_n,B_{n-2}\times G)$ is a strong Gelfand pair. 
There is one more subgroup that we have to check, that is, $G= H_2$. 
In this case, we see from Corollary~\ref{A:C:H2} that $(B_n,B_{n-2}\times G)$ is a strong Gelfand pair. 
(Of course, we could have used the same method for the subgroups $G$, where $\overline{S_2} \leqslant G \leqslant B_2$.)

\subsubsection{Strong Gelfand pairs of the form $(B_n,D_{n-2}\times B_2)$.}

First, we assume that $n$ is an even number such that $n-2=2m$ for some $m\geq 3$. 
We let $\lambda$ denote the partition $(m-1,1)$, and let $\mu$ denote the partition $(m)$.
Then $V=\res^{B_{n-2}}_{D_{n-2}} S^{\lambda,\mu}$ is an irreducible representation of $D_{n-2}$.
Furthermore, we have $\ind^{B_{n-2}}_{D_{n-2}} V = S^{\lambda,\mu}\oplus S^{\mu,\lambda}$.
The tensor product $V \boxtimes S^{(1),(1)}$ is an irreducible representation of $D_{n-2}\times B_2$.
By transitivity of induction, Lemma~\ref{L:ellproduct}, and Lemma~\ref{L:Branching2}, part 3, we have 
\begin{align}
\ind^{B_n}_{D_{n-2}\times B_2} V\boxtimes S^{(1),(1)} &= 
\ind^{B_n}_{B_{n-2}\times B_2} (S^{\lambda,\mu}\oplus S^{\mu,\lambda})  \boxtimes S^{(1),(1)} \notag \\ 
&= \ind^{B_n}_{B_{n-2}\times B_2} S^{\lambda,\mu} \boxtimes S^{(1),(1)}\oplus \ind^{B_n}_{B_{n-2}\times B_2} S^{\mu,\lambda} \boxtimes S^{(1),(1)} \notag \\ 
&= \left(\bigoplus_{\tau \in \bar{\lambda},\rho \in \bar{\mu} } S^{\tau, \rho} \right) \oplus 
\left(\bigoplus_{\rho \in \bar{\mu},\tau \in \bar{\lambda} } S^{\rho, \tau} \right).\label{A:BnDn-2even}
\end{align}
Since the multiplicity of $S^{(m,1),(m,1)}$ in (\ref{A:BnDn-2even}) is two, we see that $(B_n,D_{n-2}\times B_2)$ is not a strong Gelfand pair. 

Next, we assume that $n$ is odd.
Let $(\lambda, \mu)$ be a pair of partitions such that $|\lambda |+|\mu|=n-2$. 
Then we have $|\lambda | \neq |\mu|$. 
Hence, $V=\res^{B_{n-2}}_{D_{n-2}} S^{\lambda,\mu}$ is an irreducible representation of $D_{m-2}$, and furthermore, 
we have $\ind^{B_{n-2}}_{D_{n-2}} V = S^{\lambda,\mu}\oplus S^{\mu,\lambda}$.
Let $(a,b)$ be a pair of partitions such that $|a|+|b|=2$. 
As before, by using the transitivity of the induction and Lemma~\ref{L:ellproduct}, we get  
\begin{align}\label{A:BnDn-2odd}
\ind^{B_n}_{D_{n-2}\times B_2} V\boxtimes S^{a,b} 
&= \ind^{B_n}_{B_{n-2}\times B_2} S^{\lambda,\mu} \boxtimes S^{a,b}\oplus \ind^{B_n}_{B_{n-2}\times B_2} S^{\mu,\lambda} \boxtimes S^{a,b}.
\end{align}
But since $|\lambda|$ and $|\mu|$ are not equal, we see from Lemma~\ref{L:Branching2} that (\ref{A:BnDn-2odd}) is multiplicity-free. 
In summary, we proved the following result
\begin{lem}\label{L:BnDn-2B2}
Let $n$ be an integer such that $n\geq 8$.
Then $(B_n,D_{n-2}\times B_2)$ is a strong Gelfand pair if and only if $n$ is odd.
\end{lem}

\subsubsection{Strong Gelfand pairs of the form $(B_n,H_{n-2}\times B_2)$.}

Let $n=2m$ for some $m\geq 4$. 
First, we assume that $m$ is an even integer as well; $m=2k$ with $k\geq2$. 
Let $\lambda=\mu=(k+1,1^{k-1})$, and note that $\lambda \neq \mu'$. 
Then $V=\res^{B_{n-2}}_{H_{n-2}} S^{\lambda,\mu}$ is an irreducible representation of $H_{n-2}$,
and therefore, $\ind^{B_{n-2}}_{H_{n-2}} V = S^{\lambda,\mu}\oplus S^{\mu',\lambda'}$.
The tensor product $V \boxtimes S^{(1),(1)}$ is an irreducible representation of $H_{n-2}\times B_2$.
By transitivity of induction, Lemma~\ref{L:ellproduct}, and Lemma~\ref{L:Branching2}, part 3, we have 
\begin{align}
\ind^{B_n}_{H_{n-2}\times B_2} V \boxtimes S^{(1),(1)} &= 
\ind^{B_n}_{B_{n-2}\times B_2} (S^{\lambda,\mu}\oplus S^{\mu',\lambda'})  \boxtimes S^{(1),(1)} \notag \\ 
&= \ind^{B_n}_{B_{n-2}\times B_2} S^{\lambda,\mu} \boxtimes S^{(1),(1)}\oplus \ind^{B_n}_{B_{n-2}\times B_2} S^{\mu',\lambda'} \boxtimes S^{(1),(1)} \notag \\ 
&= \left(\bigoplus_{\tau \in \bar{\lambda},\rho \in \bar{\mu} } S^{\tau, \rho} \right) \oplus 
\left(\bigoplus_{\rho \in \bar{\mu'},\tau \in \bar{\lambda'} } S^{\rho, \tau}\right).\label{A:BnHn-2even}
\end{align}
Since the multiplicity of $S^{(k+1,1^k),(k+1,1^k)}$ in (\ref{A:BnDn-2even}) is 2, we see that $(B_n,D_{n-2}\times B_2)$ is not a strong Gelfand pair. 
Now suppose that $m = 2k+1$ with $k\geq 2$, and set $\lambda := (k+1,1^{k})$ and $\mu := (k,1^{k+1})$. 
Clearly, $\lambda$ is a self-conjugate partition and $\lambda \neq \mu$. 
Then $V=\res^{B_{n-2}}_{H_{n-2}} S^{\lambda,\mu}$ is an irreducible representation of $H_{m-2}$.
It follows that $\ind^{B_{n-2}}_{H_{n-2}} V = S^{\lambda,\mu}\oplus S^{\mu',\lambda'}$.
The tensor product $V \boxtimes S^{(1),(1)}$ is an irreducible representation of $H_{n-2}\times B_2$.
Once again, we have 
\begin{align}
\ind^{B_n}_{H_{n-2}\times B_2} V \boxtimes S^{(1),(1)} 
&= \left(\bigoplus_{\tau \in \bar{\lambda},\rho \in \bar{\mu} } S^{\tau, \rho} \right) \oplus 
\left(\bigoplus_{\rho \in \bar{\mu'},\tau \in \bar{\lambda'} } S^{\rho, \tau}\right).\label{A:BnHn-2even2}
\end{align}
It is easy to check that the multiplicity of $S^{(k+2,1^k),(k+1,1^{k+1})}$ in (\ref{A:BnHn-2even2}) is 2.
Hence, if $n$ is even, then $(B_n,D_{n-2}\times B_2)$ is not a strong Gelfand pair.

Next, we assume that $n$ is odd. 
Then, by arguing as in the second part of the $(B_n, D_{n-2}\times B_2)$ case, 
it is easy to verify that, for every irreducible representation $W$ of $B_2$ 
and for every pair of partitions $(\lambda, \mu)$ such that $|\lambda |+|\mu|=n$, the induced representation
$\ind^{B_n}_{H_{n-2}\times B_2} S^{\lambda,\mu} \boxtimes W$ is multiplicity-free. 
In summary, similarly to the case of $(B_n,D_{n-2}\times B_2)$, we proved the following result. 
\begin{lem}\label{L:BnHn-2B2}
Let $n$ be an integer such that $n\geq 8$.
Then $(B_n,H_{n-2}\times B_2)$ is a strong Gelfand pair if and only if $n$ is odd.
\end{lem}

\subsubsection{Strong Gelfand pairs of the form $(B_n,D_{n-2}\times D_2)$.}

Since $D_{n-2}\times D_2$ is a subgroup of $D_{n-2}\times B_2$, 
if $n$ is even, then by Lemma~\ref{L:BnDn-2B2} $(B_n,D_{n-2}\times D_2)$ is not a strong Gelfand subgroup. 
So, we proceed with the assumption that $n=2m+1$ for some $m\geq 4$. 

Let $\lambda$ and $\mu$ be two partitions such that $|\lambda|+|\mu|=n-2$, 
and let $S^{\lambda,\mu}$ denote the corresponding irreducible representation of $B_{n-2}$.
Since $|\lambda | \neq |\mu|$, the restricted representation $V=\res^{B_{n-2}}_{D_{n-2}} S^{\lambda,\mu}$ is an irreducible representation of $D_{n-2}$.
Furthermore, we have $\ind^{B_{n-2}}_{D_{n-2}} V = S^{\lambda,\mu}\oplus S^{\mu,\lambda}$.
The tensor product $V \boxtimes S^{(2),\emptyset}$ is an irreducible representation of $D_{n-2}\times D_2$.
By transitivity of induction, Lemma~\ref{L:ellproduct}, and Lemma~\ref{L:Branching2}, parts 1 and 5, we have 
\begin{align}\label{A:differenceinsizesDn-2D2}
\ind^{B_n}_{D_{n-2}\times D_2} V\boxtimes S^{(2),\emptyset} &= \ind^{B_n}_{B_{n-2}\times B_2}\ind^{B_{n-2}\times B_2}_{D_{n-2}\times D_2}
 V\boxtimes S^{(2),\emptyset} \notag \\
&=\ind^{B_n}_{B_{n-2}\times B_2} (S^{\lambda,\mu}\oplus S^{\mu,\lambda})  \boxtimes (S^{(2),\emptyset} \oplus S^{\emptyset,(2)}) \notag \\ 
&= \left(\bigoplus_{\tau \in \bar{\bar{\lambda}}} S^{\tau, \mu} \right) \oplus \left(\bigoplus_{\rho \in \bar{\bar{\mu}} } S^{\lambda,\rho} \right)\oplus 
\left(\bigoplus_{\rho \in \bar{\bar{\mu}}} S^{\rho, \lambda}\right) \oplus  \left(\bigoplus_{\tau \in \bar{\bar{\lambda}} } S^{\mu,\tau}\right).
\end{align}
Since $||\lambda|-|\mu||$ is odd, the representation (\ref{A:differenceinsizesDn-2D2}) is multiplicity-free. 
By using similar arguments, we see that $\ind^{B_n}_{D_{n-2}\times D_2} V\boxtimes S^{\emptyset, (2)},
\ind^{B_n}_{D_{n-2}\times D_2} V\boxtimes S^{\emptyset, (1^2)}$, and $\ind^{B_n}_{D_{n-2}\times D_2} V\boxtimes S^{(1^2),\emptyset}$
are multiplicity-free representations of $B_n$. 
Finally, we notice that $\ind^{B_n}_{D_{n-2}\times D_2} V\boxtimes S^{(1),(1)}=\ind^{B_n}_{D_{n-2}\times B_2} V\boxtimes S^{(1),(1)}$,
hence, it is also multiplicity-free (by Lemma~\ref{L:BnDn-2B2}).
Therefore, we proved the following result. 

\begin{lem}\label{L:BnDn-2D2}
Let $n$ be an integer such that $n\geq 8$.
Then $(B_n,D_{n-2}\times D_2)$ is a strong Gelfand pair if and only if $n$ is odd.
\end{lem}

\subsubsection{Strong Gelfand pairs of the form $(B_n,H_{n-2}\times D_2)$.}

The proof of this case is similar to that of Lemma~\ref{L:BnDn-2D2}.
By Lemma~\ref{L:BnHn-2B2}, if $n$ is even, then we know that $(B_n,H_{n-2}\times D_2)$ is not a strong Gelfand pair. 
We proceed with the assumption that $n$ is an odd number of the form $n=2m+1$ for some $m\geq 4$. 
Let $\lambda$ and $\mu$ be two partitions such that $|\lambda|+|\mu|=n-2$, 
and let $S^{\lambda,\mu}$ denote the corresponding irreducible representation of $B_{n-2}$.
Since $\lambda \neq \mu'$, $V=\res^{B_{n-2}}_{H_{n-2}} S^{\lambda,\mu}$ is an irreducible representation of $H_{n-2}$, 
and furthermore, we have $\ind^{B_{n-2}}_{H_{n-2}} V = S^{\lambda,\mu}\oplus S^{\mu',\lambda'}$.
From this point on, we argue as in the proof of Lemma~\ref{L:BnDn-2D2}.
We omit the details but write the conclusion below.
\begin{lem}\label{L:BnHn-2D2}
Let $n$ be an integer such that $n\geq 8$.
Then $(B_n,H_{n-2}\times D_2)$ is a strong Gelfand pair if and only if $n$ is an odd number.
\end{lem}

\subsubsection{Strong Gelfand pairs of the form $(B_n,D_{n-2}\times H_2)$ and $(B_n,H_{n-2}\times H_2)$.}

Since $D_{n-2}\times H_2$ and $H_{n-2}\times H_2$ are subgroups of $D_{n-2}\times B_2$ and $H_{n-2}\times B_2$, respectively, 
if $n$ is even, then by Lemmas~\ref{L:BnDn-2B2} and~\ref{L:BnHn-2B2} $(B_n,D_{n-2}\times H_2)$ and $(B_n,H_{n-2}\times H_2)$ are not strong Gelfand pairs. 
So, we proceed with the assumption that $n=2m+1$ for some $m\geq 4$. 
In this case, the proofs of Lemmas~\ref{L:BnDn-2D2} and~\ref{L:BnHn-2D2} are easily modified, and we get the following result.
\begin{lem}\label{L:BnDHn-2H2}
Let $n\geq 8$.
Then $(B_n,D_{n-2}\times H_2)$ is a strong Gelfand pair if and only if $n$ is odd.
Likewise, $(B_n,H_{n-2}\times H_2)$ is a strong Gelfand pair if and only if $n$ is odd.
\end{lem}

\subsubsection{Strong Gelfand pairs of the form $(B_n,D_{n-2}\times \overline{S_2})$.}

\begin{lem}\label{L:BnDn-2S2}
If $n\geq 8$, then $(B_n,D_{n-2}\times \overline{S_2})$ is not a strong Gelfand pair.
\end{lem}
\begin{proof}

Suppose $n-2 = 2m+1$ for some $m\geq 4$. 
Let  $\lambda = (m)$ and $\mu=(m+1)$. 
Then $S^{\lambda,\mu}\oplus S^{\mu,\lambda}$ is a representation of $B_{n-2}$ that is induced from an irreducible representation $V$ of $D_{n-2}$. 
Let $W$ denote the trivial representation of $\overline{S_2}$. 
By Corollary~\ref{A:C:D21}, we see that $S^{(m+1,1),(m+1)}$ has multiplicity 2 in $\ind^{B_n}_{D_{n-2}\times \overline{S_2}} V\boxtimes W$.

Next, suppose that $n-2=2m$ for some $m\geq 4$.
Let  $\lambda = (m)$  and $\mu = (1^m)$. 
Then $S^{\lambda,\mu}\oplus S^{\mu,\lambda}$ is a representation of $B_{n-2}$ that is induced from an irreducible representation $V$ of $D_{n-2}$. 
Let $W$ denote the trivial representation of $\overline{S_2}$. 
By Corollary~\ref{A:C:D21}, we see that $S^{(m+1),(1^{m+1})}$ has multiplicity 2 in $\ind^{B_n}_{D_{n-2}\times \overline{S_2}} V\boxtimes W$.
This completes the proof.
\end{proof}

\subsubsection{Strong Gelfand pairs of the form $(B_n,H_{n-2}\times \overline{S_2})$.}

\begin{lem}\label{L:BnHn-2S2}
If $n\geq 8$, then $(B_n,H_{n-2}\times \overline{S_2})$ is not a strong Gelfand pair.
\end{lem}
\begin{proof}

Suppose $n-2 = 2m+1$ for some $m\geq 4$.
Let  $\lambda = (m)$ and $\mu = (1^{m+1})$. 
Then $S^{\lambda,\mu}\oplus S^{\mu',\lambda'}$ is a representation of $B_{n-2}$ that is induced from an irreducible representation $V$ of $H_{n-2}$. 
Let $W$ denote the trivial representation of $\overline{S_2}$. 
By Corollary~\ref{A:C:D21}, we see that $S^{(m+1,1),1^{(m+1)}}$ has multiplicity 2 in $\ind^{B_n}_{H_{n-2}\times \overline{S_2}} V\boxtimes W$.

Next, we assume that $n-2=2m$ for some $m\geq 4$.
Let  $\lambda=(m-1,1)$ and $\mu = (1^m)$. 
Then $S^{\lambda,\mu}\oplus S^{\mu',\lambda'}$ is a representation of $B_{n-2}$ that is induced from an irreducible representation $V$ of $H_{n-2}$. 
Let $W$ denote the trivial representation of $\overline{S_2}$. 
By Corollary~\ref{A:C:D21}, we see that $S^{(m,1),(2,1^{m-1})}$ has multiplicity 2 in $\ind^{B_n}_{H_{n-2}\times \overline{S_2}} V\boxtimes W$.
This completes the proof.
\end{proof}

\subsubsection{Non-direct product index 2 subgroups of $B_{n-2}\times\overline{S_2}$.}

There are two non-direct product index 2 subgroups $K\leqslant B_{n-2}\times \overline{S_2}$ such that $\gamma_K = S_{n-2}\times S_2$:
\begin{enumerate}
\item $K= (B_{n-2}\times \overline{S_2})_{\delta} := \{ (a,\delta(a)):\ a\in B_{n-2}\}$;
\item $K= (B_{n-2}\times \overline{S_2})_{\varepsilon\delta }  :=  \{ (a, (\varepsilon \delta) (a)):\ a\in B_{n-2}\}$.
\end{enumerate}

We begin with the case $K= (B_{n-2}\times \overline{S_2})_{\delta}$. 
Let $\nu$ denote the linear character of $B_{n-2}\times \overline{S_2}$ such that $\ker \nu= K$.
Then the restrictions of $\nu$ to the factors are given by $\nu |_{B_{n-2}\times \{1\}} = \delta$ and $\nu |_{\{\id\} \times \overline{S_2}}=\varepsilon$.
Let $W = S^{\lambda,\mu} \boxtimes D$ be an irreducible representation of $B_{n-2}\times \overline{S_2}$.
Since $\delta S^{\lambda, \mu}= S^{\mu, \lambda}$, $\varepsilon \mathbf{\epsilon} = \mathbf{1}$, and $\varepsilon \mathbf{1} = \mathbf{\epsilon}$, 
we have 
\[
\nu ( S^{\lambda,\mu} \boxtimes D) = S^{\mu, \lambda}\boxtimes \tilde{D},
\]
where $\{ D,\tilde{D} \} = \{\mathbf{1},\mathbf{\epsilon}\}$.
In particular, the representations $S^{\lambda,\mu} \boxtimes D$
and $S^{\mu, \lambda}\boxtimes \tilde{D}$ are inequivalent. 
Hence, there is no self-associate irreducible representation with respect to $\nu$.

We are now ready to describe the induction from $K$ to $B_{n-2}\times \overline{S_2}$ by using Frobenius reciprocity.
Let $V$ be an irreducible representation of $K$.
Then we have 
\begin{align}\label{A:KtoBn-2S2}
\ind^{B_{n-2}\times \overline{S_2}}_K V = (S^{\lambda,\mu}\boxtimes \mathbf{1}) \oplus (S^{\mu, \lambda}\boxtimes  \mathbf{\epsilon})
\end{align}
for some irreducible representation $S^{\lambda,\mu}$ of $B_{n-2}$. 
Here, $\lambda$ and $\mu$ may be any partitions with $|\lambda | + |\mu|= n-2$. 
It is now easy to see from Corollary~\ref{A:C:D21} that 
if we induce the representation in (\ref{A:KtoBn-2S2}), then we will get a non-multiplicity-free representation of $B_n$. 
Indeed, for $n=2m$, we can choose $\lambda = \mu$, and for $n=2m+1$ we can choose $\lambda = (m)$ and $\mu = (m+1)$. 
Therefore, $(B_n, (B_{n-2}\times \overline{S_2})_{\delta} )$ is not a strong Gelfand subgroup.

Next, we focus on the case $K=(B_{n-2}\times \overline{S_2})_{\varepsilon\delta }$. 
We know that $H_{n-2}\times \{1\}$ is an index 2 subgroup of $K$, and $K$ is an index 2 subgroup of $B_{n-2}\times \overline{S_2}$. 
We begin with describing the irreducible representations of $K$.
Let $\nu$ denote the linear character of $B_{n-2}\times \overline{S_2}$ such that $\ker \nu= K$.
Then the restrictions of $\nu$ to the factors are given by $\nu |_{B_{n-2}\times \{1\}} = \varepsilon \delta$ and $\nu |_{\{\id\} \times \overline{S_2}}=\varepsilon$.

Let $W = S^{\lambda,\mu} \boxtimes D$ be an irreducible representation of $B_{n-2}\times \overline{S_2}$.
Since $\varepsilon\delta S^{\lambda, \mu}= S^{\mu', \lambda'}$, $\varepsilon \mathbf{\epsilon} = \mathbf{1}$, and $\varepsilon \mathbf{1} = \mathbf{\epsilon}$, 
we have 
\[
\nu ( S^{\lambda,\mu} \boxtimes D) = S^{\mu', \lambda'}\boxtimes \tilde{D},
\]
where $\{ D,\tilde{D} \} = \{\mathbf{1},\mathbf{\epsilon}\}$.
Since $D \neq \tilde{D}$, the representations $S^{\lambda,\mu} \boxtimes D$
and $S^{\mu', \lambda'}\boxtimes \tilde{D}$ are inequivalent. 
Hence, we conclude that there is no self-associate irreducible representation with respect to $\nu$. 
We are now ready to describe the induction from $K$ to $B_{n-2}\times \overline{S_2}$ by using Frobenius reciprocity.
Let $V$ be an irreducible representation of $K$.
Then we have 
\begin{align}\label{A:2KtoBn-2S2}
\ind^{B_{n-2}\times \overline{S_2}}_K V = S^{\lambda,\mu}\boxtimes \mathbf{1}\oplus S^{\mu', \lambda'}\boxtimes  \mathbf{\epsilon}
\end{align}
for some irreducible representation $S^{\lambda,\mu}$ of $B_{n-2}$. 
Here, $\lambda$ and $\mu$ can be any partitions with $|\lambda | + |\mu|= n-2$. 
It is now easy to see from Corollary~\ref{A:C:D21} that 
if we induce the representation in (\ref{A:2KtoBn-2S2}), then we will get a non-multiplicity-free representation of $B_n$. 
Indeed, for $n=2m$, can choose $V$ with $\lambda = \mu'$, and for $n=2m+1$ we can choose $\lambda = (1^m)$ and $\mu = (m+1)$. 
Therefore, $(B_n,(B_{n-2}\times \overline{S_2})_{\varepsilon\delta })$ is not a strong Gelfand subgroup.

\begin{lem}\label{L:nondirectBn-2S2}
If $n\geq 8$, then there is no strong Gelfand pair of the form $(B_n,K)$, where 
$K$ is a non-direct product index 2 subgroup of $B_{n-2}\times\overline{S_2}$ such that $\gamma_K = S_{n-2}\times S_2$.
\end{lem}

\subsubsection{Non-direct product index 2 subgroups of $B_{n-2}\times D_2$.}

Since $D_2$ is isomorphic to $\Z/2\times \Z/2$, it has four linear characters $\chi_i$ ($i\in \{0,\dots, 3\}$)
with the corresponding irreducible representations denoted by $V_i$ ($i\in \{0,\dots, 3\}$).
These one-dimensional (inequivalent) representations can be obtained by restricting the irreducible representations from $B_2$:
\begin{enumerate}
\item $V_0:=\res^{B_2}_{D_2} S^{(2),\emptyset}$,
\item $V_1:=\res^{B_2}_{D_2} S^{(1^2),\emptyset}$,
\item $V_2\oplus V_3:=\res^{B_2}_{D_2} S^{(1),(1)}$.
\end{enumerate}
Then we know that $\ind^{B_2}_{D_2} V_0 = S^{(2),\emptyset} \oplus S^{\emptyset,(2)}$,
$\ind^{B_2}_{D_2} V_1 = S^{(1^2),\emptyset} \oplus S^{\emptyset,(1^2)}$, and that 
$\ind^{B_2}_{D_2} V_2=\ind^{B_2}_{D_2} V_3= S^{(1),(1)}$. 
Note that the character group of $D_2$, which is isomorphic to $D_2$, 
acts on the set of representations $\{V_i: 0\leq i\leq3\}$ as it does on itself by left multiplication. 
The character $\chi_0$ is the trivial character, and the other three characters $\chi_i$ ($i\in \{1,2,3\}$) have order 2 as associators, satisfying $\chi_i V_0 \cong V_i$.
Note also that, in the notation of Lemma~\ref{L:sgsB2}, the kernel of the character $\chi_1$ is 
the diagonal copy of $F$ in $D_2$, the kernel of $\chi_2$ is $\overline{S_2}$, and 
the kernel of $\chi_3$ is $\overline{S_2}'$.

Let $\nu$ denote the linear character of $B_{n-2}\times D_2$ such that $\ker \nu= K$, where $K$ is a non-direct product index 2 subgroup of $B_{n-2}\times D_2$. 
The restrictions $\nu |_{B_{n-2}\times \{1\}}$ and $\nu |_{\{\id\} \times D_2}$ are nontrivial linear characters.
In particular, the kernel of $\nu |_{B_{n-2}\times \{1\}}$ is one of the following groups: $D_{n-2}\times \{1\}$, $H_{n-2}\times \{1\}$, or $F\wr A_{n-2}\times \{1\}$.

We proceed with the assumption that $\ker \nu |_{B_{n-2}\times \{1\}}= D_{n-2}\times \{1\}$.
Let $W$ be an irreducible representation of $B_{n-2}\times D_2$ of the form 
$W = S^{\lambda,\mu} \boxtimes V$, where $V\in \{V_0,\dots, V_3\}$.
Let $\tilde{V}$ denote $\chi_i V$ for some $i\in \{1,2,3\}$.
Since the character group of $D_2$ is isomorphic to $D_2$, $\chi_i$ does not fix any of the representations, $V_0,\dots, V_3$,
so, we know that $\tilde{V} \neq V$. 
Since $\delta S^{\lambda, \mu}= S^{\mu, \lambda}$ and $\tilde{V} \neq V$, 
we have $\nu ( S^{\lambda,\mu} \boxtimes V) = S^{\mu, \lambda}\boxtimes \tilde{V}$, and furthermore, 
the representations $S^{\lambda,\mu} \boxtimes V$ and $S^{\mu, \lambda}\boxtimes \tilde{V}$ are inequivalent.
Therefore, the restrictions of both of these representations to $K$ give the same irreducible representation,
\[
C:= \res^{B_{n-2}\times D_2}_K S^{\lambda,\mu} \boxtimes V =  \res^{B_{n-2}\times D_2}_K S^{\mu,\lambda} \boxtimes \tilde{V}.
\]
By inducing it to $B_n$, we get 
\begin{align*}
\ind^{B_n}_K C &= \ind^{B_n}_{B_{n-2}\times D_2} \ind^{B_{n-2}\times D_2}_K C \\
&=  \ind^{B_n}_{B_{n-2}\times D_2} (S^{\lambda,\mu}\boxtimes V \oplus S^{\mu, \lambda}\boxtimes  \tilde{V}).
\end{align*}
As we have three possibilities for $\nu |_{\{\id\} \times D_2}$, which are given by $\chi_1,\chi_2$, and $\chi_3$, we proceed to analyze them separately.

First suppose that $\nu |_{\{\id\} \times D_2} = \chi_1$.
To distinguish it from the other two cases, let us denote $\nu$ by $\nu_1$.
Then we see that any irreducible $K$-module $C$ is of the form a) $C= \res^{B_{n-2}\times D_2}_K S^{\lambda,\mu} \boxtimes V_0 =  \res^{B_{n-2}\times D_2}_K S^{\mu,\lambda} \boxtimes V_1$ or b) $C= \res^{B_{n-2}\times D_2}_K S^{\lambda,\mu} \boxtimes V_2 =  \res^{B_{n-2}\times D_2}_K S^{\mu,\lambda} \boxtimes V_3$.
In the former case, using Corollary~\ref{A:C:D2} we have 
\begin{align*}
\ind^{B_n}_K C &=  \ind^{B_n}_{B_{n-2}\times D_2} (S^{\lambda,\mu}\boxtimes V_0 \oplus S^{\mu, \lambda}\boxtimes  V_1)\\
&= \bigoplus_{\tau \in \bar{\bar{\lambda}}} S^{\tau, \mu}
\oplus
\bigoplus_{\rho \in \bar{\bar{\mu}} } S^{\lambda,\rho}
\oplus
\bigoplus_{\alpha \in \tilde{\bar{\mu}}} S^{\alpha, \lambda}
\oplus
\bigoplus_{\beta \in \tilde{\bar{\lambda}} } S^{\mu, \beta}.
\end{align*}
If $n$ is even, then setting $\lambda=(m)$ and $\mu=(m+1,1)$, $S^{(m+1,1),(m+1,1)}$ appears in the above with multiplicity 2.

If $n$ is odd, then the sum is easily seen to be multiplicity-free, by considering parities of the partitions involved.
If $C$ is as in b), then
\[
\ind^{B_n}_K C = \bigoplus_{\tau \in \bar{\lambda},\rho \in \bar{\mu} } S^{\tau, \rho} \oplus \bigoplus_{\alpha \in \bar{\mu},\beta \in \bar{\lambda} } S^{\alpha, \beta}.
\]
which is easily seen to be multiplicity-free by considering parities of the partitions involved.
Therefore, we showed that $\ker \nu_1$ is a strong Gelfand subgroup if and only if $n$ is odd.

Now suppose that $\nu |_{\{\id\} \times D_2} = \chi_2$.
Then we will denote $\nu$ by $\nu_2$.
Then we see that any irreducible $K$-module $C$ is of the form a) $C= \res^{B_{n-2}\times D_2}_K S^{\lambda,\mu} \boxtimes V_0 =  \res^{B_{n-2}\times D_2}_K S^{\mu,\lambda} \boxtimes V_2$ or b) $C= \res^{B_{n-2}\times D_2}_K S^{\lambda,\mu} \boxtimes V_1 =  \res^{B_{n-2}\times D_2}_K S^{\mu,\lambda} \boxtimes V_3$.
In the former case, we have
\begin{align*}
\ind^{B_n}_K C &=  \ind^{B_n}_{B_{n-2}\times D_2} (S^{\lambda,\mu}\boxtimes V_0 \oplus S^{\mu, \lambda}\boxtimes  V_2)\\
&= \bigoplus_{\tau \in \bar{\bar{\lambda}}} S^{\tau, \mu} \oplus \bigoplus_{\rho \in \bar{\bar{\mu}} } S^{\lambda,\rho} \oplus \bigoplus_{\alpha \in \bar{\mu},\beta \in \bar{\lambda} } S^{\alpha, \beta},
\end{align*}
and the latter case is analogous.
If $n$ is odd, then we may set $\lambda=(m)$ and $\mu = (m+1)$, and see that $S^{(m+2),(m+1)}$ appears with multiplicity 2.

If $n$ is even, then the sum is easily seen to be multiplicity-free, by considering parities of the partitions involved.
Therefore, we showed that $\ker \nu_2$ is a strong Gelfand subgroup if and only if $n$ is even.

Now we will consider the final case where $\nu$ is such that $\nu |_{\{\id\} \times D_2} = \chi_3$.
Let us write $\nu_3$ instead of $\nu$. We claim that $\ker \nu_3$ is conjugate to the subgroup $\ker \nu_2$. 
Indeed, let $\eta : B_n \to B_n$ denote the inner automorphism defined by the element $(1,x)$ of $D_{n-2}\times B_2$, 
where $x$ is as defined in part 4 of Lemma~\ref{L:sgsB2}. We know that $x$ has the property that $x \overline{S_2} x^{-1} = \overline{S_2}'$ (but $x\notin D_2$). Therefore, we see that $\ker \nu_1$ and $\ker \nu_2$ are conjugate subgroups of $B_n$. 
In conclusion, as far as our classification up-to-conjugation concerned, in the case of $\nu_3$, we do not get a ``new'' strong Gelfand subgroup.

In the remaining two major cases, where $\ker \nu |_{B_{n-2}\times \{1\}}= H_{n-2}\times \{1\}$
or $\ker \nu |_{B_{n-2}\times \{1\}}= F\wr A_{n-2}\times \{1\}$, our analyses are almost identical
to the case of $\ker \nu |_{B_{n-2}\times \{1\}}= D_{n-2}\times \{1\}$.
In fact, in the former case, for every $n\geq 8$, we find the same number of strong Gelfand subgroups up to conjugacy as 
in the case of $\ker \nu |_{B_{n-2}\times \{1\}}= D_{n-2}\times \{1\}$. 
Nevertheless, in the latter case, we get only one strong Gelfand subgroup up to conjugacy for every $n\geq 8$.
Since all of our arguments in these cases are very similar to the arguments we had for the first case, we omit their details.
The summary of our results are as follows.

\begin{lem}\label{L:nondirectBn-2D2}
Let $n\geq 8$, and let $K$ be a non-direct product index 2 subgroup of $B_{n-2}\times D_2$ with $\gamma_K = S_{n-2}\times S_2$. 
Let $\nu$ denote the linear character of $B_{n-2}\times D_2$ such that $K =\ker \nu$.
If $\ker \nu|_{B_{n-2}\times \{\id\} }$ is equal to either $D_{n-2} \times \{1\}$ or $H_{n-2}\times \{1\}$, then 
\begin{enumerate}
\item  
If $n$ is odd, then there is one such strong Gelfand subgroup, with $\nu |_{\{\id\} \times D_2} = \chi_1$.

\item If $n$ is even, then there are two such strong Gelfand subgroups, with $\nu |_{\{\id\} \times D_2} = \chi_2$ or $\chi_3$, respectively.
These are conjugate to each other.
\end{enumerate}
If $\ker \nu|_{B_{n-2}\times \{\id\} } = F\wr A_{n-2}\times \{1\}$, then there are two such strong Gelfand subgroups, with $\nu |_{\{\id\} \times D_2} = \chi_2$ or $\chi_3$, respectively. These are conjugate to each other.
\end{lem}

\subsubsection{Non-direct product index 2 subgroups of $B_{n-2}\times H_2$.}

$H_2$ is isomorphic to $\Z/4$, so, it has four linear characters $\chi_i$ ($i\in \{0,\dots, 3\}$)
with the corresponding irreducible representations denoted by $V_i$ ($i\in \{0,\dots, 3\}$).
We denote by $\chi_0$ the trivial character, and we denote by $\chi_1$ a generator so that $\chi_i = \chi_1^i$ for $i\in \{1,2,3\}$. 
As in the case of $D_2$, we will express the (inequivalent) irreducible representations of $H_2$ 
by restricting the irreducible representations from $B_2$:
\begin{enumerate}
\item $V_0:=\res^{B_2}_{H_2} S^{(2),\emptyset}$,
\item $V_1\oplus V_3:=\res^{B_2}_{H_2} S^{(1),(1)}$,
\item $V_2:=\res^{B_2}_{H_2} S^{(1^2),\emptyset}$.
\end{enumerate}
Then we know that $\ind^{B_2}_{H_2} V_0 = S^{(2),\emptyset} \oplus S^{\emptyset,(1^2)}$,
$\ind^{B_2}_{H_2} V_2 = S^{(1^2),\emptyset} \oplus S^{\emptyset,(2)}$, and that 
$\ind^{B_2}_{H_2} V_1=\ind^{B_2}_{H_2} V_3= S^{(1),(1)}$. 
The character group of $H_2$, which is isomorphic to $H_2$, 
acts on the set of representations $\{V_i: i\in \{0,\dots, 3\}\}$ as it acts on itself by left multiplication.

Let $\nu$ denote the linear character of $B_{n-2}\times H_2$ such that $\ker \nu= K$, where $K$ is a non-direct product index 2 subgroup
of $B_{n-2}\times H_2$. 
The restrictions $\nu |_{B_{n-2}\times \{1\}}$ and $\nu |_{\{\id\} \times H_2}$ are nontrivial linear characters.
In particular, the kernel of $\nu |_{B_{n-2}\times \{1\}}$ is one of the following groups: $D_{n-2}\times \{1\}$, $H_{n-2}\times \{1\}$, or $F\wr A_{n-2}\times \{1\}$.
Let $\chi_i$ be the nontrivial character of $H_2$ such that $\nu |_{\{\id\} \times H_2} = \chi_i$. 
Since $\nu |_{\{\id\} \times H_2}$ has order 2, we have $\chi_i = \chi_2$.

First, let us assume that $\ker \nu |_{B_{n-2}\times \{1\}}= D_{n-2}\times \{1\}$.
Let $W$ be an irreducible representation of $B_{n-2}\times H_2$ of the form $W = S^{\lambda,\mu} \boxtimes V$, where $V\in \{V_0,\dots, V_3\}$.
Let $\tilde{V}$ denote $\chi_2 V$. 
Since $\chi_2$ does not fix any of the representations, $V_0,\dots, V_3$, $W$ is not self-associate representation with respect to $\nu$. 
In particular, we have $\nu ( S^{\lambda,\mu} \boxtimes V) = S^{\mu, \lambda}\boxtimes \tilde{V}$.
Furthermore, the representations $S^{\lambda,\mu} \boxtimes V$ and $S^{\mu, \lambda}\boxtimes \tilde{V}$ are inequivalent. 
Therefore, the restrictions of both of these representations to $K$ give the same irreducible representation,
\[
E:= \res^{B_{n-2}\times B_2}_K S^{\lambda,\mu} \boxtimes V =  \res^{B_{n-2}\times B_2}_K S^{\mu,\lambda} \boxtimes \tilde{V}.
\]
By inducing it to $B_n$, we get 
\begin{align*}
\ind^{B_n}_K E &= \ind^{B_n}_{B_{n-2}\times H_2} \ind^{B_{n-2}\times H_2}_K E  \\
&=  \ind^{B_n}_{B_{n-2}\times H_2} (S^{\lambda,\mu}\boxtimes V \oplus S^{\mu, \lambda}\boxtimes  \tilde{V}).
\end{align*}
For the action of $\chi_2$ on $\{V_0,\dots, V_3\}$ we have $\chi_2 V_0 \cong V_2$ and $\chi_2 V_1=V_3$.

We proceed with the assumption that $V= V_0$ and $\tilde{V} = V_2$ in $E$. 
Then, by Corollary~\ref{A:C:H2} we have 
\begin{align*}
\ind^{B_n}_K E &=  \ind^{B_n}_{B_{n-2}\times H_2} (S^{\lambda,\mu}\boxtimes V_0 \oplus S^{\mu, \lambda}\boxtimes  V_2)\\
&= \bigoplus_{\tau \in \bar{\bar{\lambda}}} S^{\tau, \mu}
\oplus
\bigoplus_{\rho \in \tilde{\bar{\mu}} } S^{\lambda,\rho}
\oplus
\bigoplus_{\alpha \in \tilde{\bar{\mu}}} S^{\alpha, \lambda}
\oplus
\bigoplus_{\beta \in \bar{\bar{\lambda}} } S^{\mu, \beta}.
\end{align*}
If $n$ is even, then we may set $\lambda=(m-1)$ and $\mu = (m,1)$, and see that $S^{(m,1),(m,1)}$ appears with multiplicity 2.
If $n$ is odd, then the sum is easily seen to be multiplicity-free, by considering parities of the partitions involved.

We now proceed with the assumption that $V= V_1$ and $\tilde{V} = V_3$ in $E$. 
Then, by Corollary~\ref{A:C:H2} we have 
\begin{align*}
\ind^{B_n}_K E &=  \ind^{B_n}_{B_{n-2}\times H_2} (S^{\lambda,\mu}\boxtimes V_1 \oplus S^{\mu, \lambda}\boxtimes  V_3)\\
&= \bigoplus_{\tau \in {\bar{\lambda}}, \rho \in {\bar{\mu}}} S^{\tau, \rho}
\oplus
 \bigoplus_{\rho \in {\bar{\mu}}, \tau \in {\bar{\lambda}}} S^{\rho,\tau}
\end{align*}
If $n$ is even, then we may set $\lambda=\mu$. Then the sum is not multiplicity-free. 
If $n$ is odd, then the sum is easily seen to be multiplicity-free, by considering parities of the partitions involved.

The case where $\ker \nu |_{B_{n-2}\times \{1\}}= H_{n-2}\times \{1\}$ is almost identical, as is the result in that case.
The case where $\ker \nu |_{B_{n-2}\times \{1\}}= F\wr A_{n-2}\times \{1\}$ is almost identical in proof, but the result in that case is that there are no such strong Gelfand subgroups.
We summarize these below.

\begin{lem}\label{L:nondirectBn-2H2}
Let $n\geq 8$, and let $K$ be a non-direct product index 2 subgroup of $B_{n-2}\times H_2$ with $\gamma_K = S_{n-2}\times S_2$.
If $n$ is even, then $K$ is not a strong Gelfand subgroup.
If $n$ is odd, there are two such subgroups $K$ that are strong Gelfand subgroups.
\end{lem}

\subsubsection{Non-direct product index 2 subgroups of $B_{n-2}\times B_2$.}

Let $\nu$ denote the linear character of $B_{n-2}\times B_2$ such that $\ker \nu= K$, where $K$ is a non-direct product index 2 subgroup
of $B_{n-2}\times B_2$. 
The restrictions $\nu |_{B_{n-2}\times \{1\}}$ and $\nu |_{\{\id\} \times B_2}$ are nontrivial linear characters.
Each factor can be one of the 3 nontrivial linear characters of the corresponding hyperoctahedral group.
Therefore, we have nine cases.

1) We start with the case $\ker \nu |_{B_{n-2}\times \{1\}}=D_{n-2}\times \{1\}$ and $\ker \nu |_{\{\id\} \times B_2}=\{\id\}\times D_2$. 
Since $K$ is an index 2 subgroup of $B_{n-2}\times B_2$, for an irreducible representation $W$ of $K$,
$\ind_K^{B_{n-2}\times B_2} W$ is either irreducible, or it is the direct sum of two inequivalent irreducible representations $V_1$ and $V_2$ 
such that $\res^{B_{n-2}\times B_2}_K V_1 = \res^{B_{n-2}\times B_2}_{B_2} V_2=W$.
Since $(B_n, B_{n-2}\times B_2)$ is a strong Gelfand pair (see Theorem~\ref{T:nonabelian}), 
only in the second case it may happen that $\ind_K^{B_n} W = \ind^{B_n}_{B_{n-2}\times B_2} V_1\oplus V_2$ is not multiplicity-free.
So, we will look more closely at the second case.
The irreducible representations $V_1$ and $V_2$ are associate representations with respect to $\nu$.
Let $V_1 = S^{\lambda,\mu}\boxtimes S^{\sigma, \tau}$, where $\lambda$ and $\mu$ are two partitions such that $|\lambda|+|\mu| = n-2$ and 
$\sigma,\tau$ are two partitions such that $|\sigma| + |\tau | =2$. 
Then $V_2 = \nu V_1 = S^{\mu,\lambda}\boxtimes S^{\tau, \sigma}$, and therefore, we have 
\begin{align}\label{A:howW}
\ind^{B_n}_K W = \ind^{B_n}_{B_{n-2}\times B_2} (S^{\lambda,\mu}\boxtimes S^{\sigma, \tau} \oplus S^{\mu,\lambda}\boxtimes S^{\tau, \sigma})
\end{align}

Then it follows from Lemma~\ref{L:Branching2} that (\ref{A:howW}) is multiplicity-free if and only if $n$ is odd.

In the following five cases,  
\begin{enumerate}
\item[] 2) $\ker \nu |_{B_{n-2}\times \{1\}}=D_{n-2}\times \{1\}$ and $\ker \nu |_{\{\id\} \times B_2}=\{\id\}\times H_2$,
\item[] 3) $\ker \nu |_{B_{n-2}\times \{1\}}=H_{n-2}\times \{1\}$ and $\ker \nu |_{\{\id\} \times B_2}=\{\id\}\times D_2$,
\item[] 4) $\ker \nu |_{B_{n-2}\times \{1\}}=H_{n-2}\times \{1\}$ and $\ker \nu |_{\{\id\} \times B_2}=\{\id\}\times H_2$,
\item[] 5) $\ker \nu |_{B_{n-2}\times \{1\}}=H_{n-2}\times \{1\}$ and $\ker \nu |_{\{\id\} \times B_2}=\{\id\}\times F \wr A_2$,
\item[] 6) 
$\ker \nu |_{B_{n-2}\times \{1\}}=D_{n-2}\times \{1\}$ and $\ker \nu |_{\{\id\} \times B_2}=\{\id\}\times F\wr A_2$,
\end{enumerate}
we arrive at the same conclusion by similar arguments, so, we omit their details. 
Also by using similar arguments, it is easily checked that the following three cases are not strong Gelfand, 
\begin{enumerate}
\item[] 7) $\ker \nu |_{B_{n-2}\times \{1\}}=F\wr A_{n-2}\times \{1\}$ and $\ker \nu |_{\{\id\} \times B_2}=\{\id\}\times H_2$,
\item[] 8) $\ker \nu |_{B_{n-2}\times \{1\}}=F\wr A_{n-2}\times \{1\}$ and $\ker \nu |_{\{\id\} \times B_2}=\{\id\}\times D_2$,
\item[] 9) $\ker \nu |_{B_{n-2}\times \{1\}}=F\wr A_{n-2}\times \{1\}$ and $\ker \nu |_{\{\id\} \times B_2}=\{\id\}\times F \wr A_2$.
\end{enumerate}

Here is the result of this subsection: 
\begin{lem}\label{L:nondirectBn-2B2}
Let $n\geq 8$.
If $n$ is odd, then there are six non-direct product, index 2, strong Gelfand subgroups of 
$B_{n-2}\times B_2$ such that $\gamma_K = S_{n-2}\times S_2$.
If $n$ is even, then there are no strong Gelfand subgroups.
\end{lem}

\subsubsection{Non-direct product index 4 normal subgroups of $B_{n-2}\times B_2$.}\label{S:index4BB}

Let $n\geq 8$ and let $K$ be a strong Gelfand subgroup of $B_{n-2}\times B_2$ such that $\gamma_K = S_{n-2}\times S_2$. 
If $K$ is a non-direct product index 4 subgroup of $B_{n-2}\times B_2$, then 
$K$ is a subgroup of index 2 of a group $K'$, where $K'$ is a subgroup of index 2 in $B_{n-2}\times B_2$.
Then $K'$ is either a direct product of the form 1) $D_{n-2}\times B_2$, 2) $H_{n-2}\times B_2$, 3) $B_{n-2}\times D_2$, 4) $B_{n-2}\times H_2$,
or 5) $K'$ is a non-direct product subgroup of index 2 in $B_{n-2}\times B_2$.
In cases 3) and 4), we have  already found all such strong Gelfand subgroups, in Lemmas~\ref{L:nondirectBn-2D2} and \ref{L:nondirectBn-2H2}, respectively.
In cases 1), 2), and 5), we determined in Lemmas~\ref{L:BnDn-2B2}, \ref{L:BnHn-2B2}, and \ref{L:nondirectBn-2B2}, respectively, that in each case, $K'$ is strong Gelfand if and only $n$ is odd.
Thus, it remains to check index 2 subgroups of these 3 groups when $n$ is odd.

We proceed with the assumption that $n$ is odd.

{\em 1) $K$ is a non-direct product index 2 subgroup of $K':=D_{n-2}\times B_2$.} 

Let $\nu$ be the linear character of $D_{n-2}\times B_2$ that defines $K$ in $K'$.
Then $\nu |_{D_{n-2}\times \{\id \}}$ and $\nu |_{\{\id\}\times B_2}$ are linear characters of $D_{n-2}$ and $B_2$, respectively. 
In particular, the linear character $\nu |_{D_{n-2}\times \{1 \}}$ is given by $\varepsilon |_{D_{n-2}\times \{1 \}}$.

Let $U\boxtimes S^{a,b}$ be an irreducible representation of $K'$, 
where $U$ is an irreducible representation of $D_{n-2}$ and $S^{a,b}$ is an irreducible representation of $B_2$; here, 
$a$ and $b$ are two integer partitions such that $|a| + |b| = 2$. 
Since $n$ is odd, $U$ is of the form $\res^{B_{n-2}}_{D_{n-2}} S^{\lambda,\mu}$, 
where $\lambda$ and $\mu$ are two (distinct) partitions such that $|\lambda|+|\mu| = n-2$.

Let $V$ be an irreducible constituent of $\res^{K'}_K S^{\lambda,\mu} \boxtimes S^{a,b}$. 
Every irreducible representation of $K$ arises this way. 
We proceed with the assumptions that $\lambda \notin \{\mu, \mu'\}$ and $a=b=(1)$.
We now look at the induced representation $\ind^{B_n}_K V =\ind^{B_n}_{K'} \ind^{K'}_K V$.
Since the restriction $\nu |_{D_{n-2}\times \{\id \}}$ is given by $\varepsilon$, we see that 
$\ind^{K'}_K V = {S^{\lambda,\mu}} \boxtimes S^{(1),(1)} \oplus {S^{\lambda',\mu'}} \boxtimes S^{(1),(1)}$.
Hence, by using the arguments that led us to (\ref{A:BnDn-2odd}), we obtain 
\begin{align}\label{A:BnDn-2oddsecondtime}
\ind^{B_n}_K V &= \ind^{B_n}_{K'} {S^{\lambda,\mu}} \boxtimes S^{(1),(1)} \oplus \ind^{B_n}_{K'} {S^{\lambda',\mu'}} \boxtimes S^{(1),(1)} \notag \\
&= \left(\bigoplus_{\tau \in \bar{\lambda},\rho \in \bar{\mu} } S^{\tau, \rho} \right) \oplus 
\left(\bigoplus_{\rho \in \bar{\mu},\tau \in \bar{\lambda} } S^{\rho, \tau} \right) \oplus 
 \left(\bigoplus_{\tau \in \bar{\lambda'},\rho \in \bar{\mu'} } S^{\tau, \rho} \right) \oplus 
\left(\bigoplus_{\rho \in \bar{\mu'},\tau \in \bar{\lambda'} } S^{\rho, \tau} \right).
\end{align}
For $\lambda := (r, 1^{r})$ and $\mu := (s,1^{s-1})$, where $2r+2s-1 = n-2$, we see that the multiplicity of 
$S^{(r+1,1^r),(s,1^{s})}$ in (\ref{A:BnDn-2oddsecondtime}) is 2. 
Hence, $K$ is not a strong Gelfand subgroup of $B_n$.\\

{\em 2) $K$ is a non-direct product index 2 subgroup of $K':=H_{n-2}\times B_2$.} 

This case develops essentially in the same way as the previous case does; $K$ is not a strong Gelfand subgroup of $B_n$. 
We omit the details for brevity.\\

{\em 3) $K$ is an index 2 subgroup of a non-direct product index 2 subgroup $K'$ of $B_{n-2}\times B_2$.} 

Let $\nu'$ be the linear character of $B_{n-2}\times B_2$ that defines $K'$. 
Since $K'$ is an index 2 subgroup of $B_{n-2}\times B_2$, 
the linear characters $\nu' |_{B_{n-2}}$ and $\nu'|_{B_2}$ are in $\{ \delta_{B_{n-2}}, (\varepsilon \delta)_{B_{n-2}}\}$ 
and $\{ \delta_{B_2}, (\varepsilon \delta)_{B_2}\}$, respectively. 
Let $\nu$ be the linear character of $K'$ such that $K=\ker \nu$.
Then we see that $\nu |_{K' \cap (B_{n-2}\times \{1\})}$ is obtained by restricting $\delta$ on the kernel of $(\varepsilon \delta)_{B_{n-2}}$, or vice versa. 
Likewise, we have $\nu |_{K' \cap (\{\id\}\times B_{2})}$ is the restriction of $\delta$ onto the kernel of $(\varepsilon \delta)_{B_{2}}$, or vice versa. 
The irreducible representations of $K$ are constructed in two stages: first, we describe the irreducible representations of $K'$
(by using Clifford theory applied to $B_{n-2}$), then we apply the same method (Clifford theory) to $K'$.
In particular, the decompositions of the induced irreducible representations of $K$ have the same description as those 
induced irreducible representations from the previous two cases. 
Therefore, $K$ is not a strong Gelfand subgroup. 

In summary, a non-direct product index 4 normal subgroup of $B_{n-2}\times B_2$ with $\gamma_K = S_{n-2}\times S_2$ is a strong Gelfand subgroup 
of $B_n$ only if it is among the strong Gelfand subgroups we described previously.

\subsubsection{Non-direct product index 2 subgroups of $D_{n-2}\times G$ and $H_{n-2}\times G$.}\label{S:DGandHG}

In this section, $G$ is one of the subgroups $B_2,D_2,H_2$, or $\overline{S_2}$. 
Let us first assume that $K$ is a non-direct product index 2 subgroup of $D_{n-2}\times B_2$.
Our analysis in the first paragraph and part 1) of the previous Subsection~\ref{S:index4BB} shows that 
$K$ cannot be a strong Gelfand subgroup. 
At the same time, since every non-direct product index 2 subgroup of the form $D_{n-2}\times G$, 
where $G\in \{ H_2,D_2,\overline{S_2}\}$ is contained in a non-direct product index 2 subgroup of $D_{n-2}\times B_2$, 
by the transitivity of the strong Gelfand subgroup property, we see that 
there is no non-direct product index 2 strong Gelfand subgroup of the form $D_{n-2}\times G$. 
The case where $K$ is a non-direct product index 2 subgroup of $H_{n-2}\times G$ can be handled in an entirely similar way;
there are no new strong Gelfand subgroups in this case also. 
In summary, a non-direct product index 2  subgroup of $D_{n-2}\times G$ or $H_{n-2} \times G$ with $\gamma_K = S_{n-2}\times S_2$ is a strong Gelfand subgroup of $B_n$ only if it is among the strong Gelfand subgroups we described previously.

\subsubsection{Summary for $\gamma_K = S_{n-2}\times S_2$.}

We now summarize the conclusions of the previous subsections in a single proposition. 

\begin{prop}\label{P:Summary4}
Let $n\geq8$ and let $K$ be a subgroup of $B_n$ such that $\gamma_K=S_{n-2}\times S_2$.
In this case, $(B_n,K)$ is a strong Gelfand pair if and only if $K$ is conjugate to one of the following subgroups:
\begin{enumerate}
\item $K= B_{n-2}\times B_2$,
\item $K= B_{n-2}\times D_2$,
\item $K= B_{n-2}\times \overline{S_2}$,
\item $K= B_{n-2}\times H_2$,
\item $K=D_{n-2}\times D_2$ if $n$ is odd,
\item $K=D_{n-2}\times B_2$ if $n$ is odd,
\item $K=H_{n-2}\times D_2$ if $n$ is odd,
\item $K=H_{n-2}\times B_2$ if $n$ is odd,
\item $K=D_{n-2}\times H_2$ if $n$ is odd,
\item $K=H_{n-2}\times H_2$ if $n$ is odd,
\item three non-direct product index 2 subgroup of $B_{n-2}\times D_2$,
\item two non-direct product index 2 subgroups of $B_{n-2}\times H_2$ if $n$ is odd,
\item six non-direct product index 2 subgroups of $B_{n-2}\times B_2$ if $n$ is odd.
\end{enumerate}
\end{prop}

\subsection{Exceptional cases}\label{S:exceptionals}

Finally, we add some closing remarks regarding the missing cases for small $n$.
Note that the strong Gelfand subgroups that we have found in Sections~\ref{S:Hyperoctahedral} and~\ref{S:lattertwo}, and collected in Table~\ref{F:thelist}, are all still strong Gelfand pairs in these small cases.
Our lower bounds on $n$ come in when reducing the possible cases to check, and are thus required only in order for the corresponding parts of Table~\ref{F:thelist} to be exhaustive.
For $\gamma_K = S_n$ or $A_n$, we have filled in these extra cases along the way, and they appear in Propositions~\ref{P:Summary1} and~\ref{P:Summary2}.
With more work, we could have considered these extra subgroups when $\gamma_K = S_{n-1}\times S_1$ or $S_{n-2}\times{S_2}$, for example extending Corollary~\ref{C:reductionII} and Lemma~\ref{L:11cases} to pick up extra possible subgroups for $n\leq 6$.

Further to these, there are also missing cases for $n=4$, $5$, and $6$ arising due to the extra possibilities for $\gamma_K$ -- see \cite[Theorem~4.13]{AHN}.

Instead of an extensive body of work to explicitly give a long list of all strong Gelfand subgroups in these few small cases, we have instead computed them in GAP, and the following proposition (along with \cref{L:sgsB2}) summarize the number of these in each case, with finer detail for $B_3$.

\begin{prop}\label{P:smallcases}
For $B_3$, Table~\ref{F:thelist}, Propositions~\ref{P:Summary1}, ~\ref{P:Summary2}, and~\ref{P:Summary3} give us 21 strong Gelfand subgroups.
In fact, there are 22 in total.
Up to conjugation, the only strong Gelfand subgroup we have not seen, which falls into the case $\gamma_K = S_{2} \times S_1$, is
\[
K = \overline{S_2} \times B_1=  \{((0,0,0),\id_{B_1}),((0,0,1),\id_{B_1}),((0,0,0),(1,2)),((0,0,1),(1,2))\}.
\]
Up to conjugation, the other small hyperoctahedral groups have the following numbers of strong Gelfand subgroups.

\begin{itemize}
\item $B_4$ has 32 strong Gelfand subgroups.

\item $B_5$ has 43 strong Gelfand subgroups.

\item $B_6$ has 20 strong Gelfand subgroups.

\item $B_7$ has 37 strong Gelfand subgroups.
\end{itemize}
\end{prop}

\bibliographystyle{amsplain}
\addcontentsline{toc}{section}{\refname}
\bibliography{References}

\providecommand{\bysame}{\leavevmode\hbox to3em{\hrulefill}\thinspace}
\providecommand{\MR}{\relax\ifhmode\unskip\space\fi MR }
\providecommand{\MRhref}[2]{%
  \href{http://www.ams.org/mathscinet-getitem?mr=#1}{#2}
}
\providecommand{\href}[2]{#2}
\begin{thebibliography}{10}

\bibitem{Aizenbudetal}
A.~Aizenbud, D.~Gourevitch, S.~Rallis, and G.~Schiffmann,
  \emph{\href{https://doi.org/10.4007/annals.2010.172.1413}{Multiplicity one
  theorems}}, Ann.\ of Math.\ (2) \textbf{172} (2010), no.~2, 1407--1434.

\bibitem{AHN}
G.~Anderson, S.~P. Humphries, and N.~Nicholson,
  \emph{\href{https://doi.org/10.1142/S0219498821500547}{Strong {Gelfand} pairs
  of symmetric groups}}, J.~Algebra Appl. (to appear).

\bibitem{BensonRatcliff}
Chal Benson and Gail Ratcliff,
  \emph{\href{https://doi.org/10.4064/cm7249-8-2017}{A family of finite
  {G}elfand pairs associated with wreath products}}, Colloq. Math. \textbf{152}
  (2018), no.~1, 65--78.

\bibitem{BridsonMiller}
Martin~R. Bridson and Charles~F. Miller, III,
  \emph{\href{https://doi.org/10.1112/plms/pdn039}{Structure and finiteness
  properties of subdirect products of groups}}, Proc. Lond. Math. Soc. (3)
  \textbf{98} (2009), no.~3, 631--651.

\bibitem{CST09}
T.~Ceccherini-Silberstein, A.~Mach\'{\i}, F.~Scarabotti, and F.~Tolli,
  \emph{\href{https://doi.org/10.1007/s10958-008-9254-5}{Induced
  representations and {Mackey} theory}}, J.\ Math.\ Sci.\ (N.Y.) \textbf{156}
  (2009), no.~1, 11--28.

\bibitem{CST10}
T.~Ceccherini-Silberstein, F.~Scarabotti, and F.~Tolli,
  \emph{\href{https://doi.org/10.1017/CBO9781139192361}{Representation theory
  of the symmetric groups: The Okounkov-Vershik approach, character formulas,
  and partition algebras}}, Cambridge Studies in advanced mathematics,
  Cambridge Univ.\ Press, 2010.

\bibitem{CST14}
\bysame, \emph{\href{https://doi.org/10.1017/CBO9781107279087}{Representation
  theory and harmonic analysis of wreath products of finite groups}}, London
  Math.\ Soc.\ Lecture Note Ser., vol. 410, Cambridge Univ.\ Press, 2014.

\bibitem{CurtisReiner}
C.~W. Curtis and I.~Reiner,
  \emph{\href{https://doi.org/10.1090/chel/356}{Representation theory of finite
  groups and associative algebras}}, AMS Chelsea Publishing, Providence, RI,
  2006, Reprint of the 1962 original.

\bibitem{GT1}
I.~M. Gel{'}fand and M.~L. Cetlin, \emph{Finite-dimensional representations of
  groups of orthogonal matrices}, Doklady Akad.\ Nauk SSSR (N.S.) \textbf{71}
  (1950), 1017--1020.

\bibitem{GT2}
\bysame, \emph{Finite-dimensional representations of the group of unimodular
  matrices}, Doklady Akad.\ Nauk SSSR (N.S.) \textbf{71} (1950), 825--828.

\bibitem{GodsilMeagher}
C.~Godsil and K.~Meagher,
  \emph{\href{https://doi.org/10.1007/s00026-009-0035-8}{Multiplicity-free
  permutation representations of the symmetric group}}, Ann.\ Comb. \textbf{13}
  (2010), no.~4, 463--490.

\bibitem{JamesKerber}
G.~D. James and A.~Kerber, \emph{The representation theory of the symmetric
  group}, Encyclopedia of Mathematics and its Applications, vol.~16,
  Addison-Wesley, 1981.

\bibitem{Jantzen}
J.~C. Jantzen, \emph{\href{http://dx.doi.org/10.1090/surv/107}{Representations
  of algebraic groups}}, second ed., Mathematical Surveys and Monographs, vol.
  107, American Mathematical Society, Providence, RI, 2003.

\bibitem{KobayashiMatsuki}
T.~Kobayashi and T.~Matsuki,
  \emph{\href{https://doi.org/10.1007/s00031-014-9265-x}{Classification of
  finite-multiplicity symmetric pairs}}, Transform.\ Groups \textbf{19} (2014),
  no.~2, 457--493.

\bibitem{Kramer}
M.~Kr\"{a}mer, \emph{\href{https://doi.org/10.1007/BF01224637}{Multiplicity
  free subgroups of compact connected {Lie} groups}}, Arch.\ Math. \textbf{27}
  (1976), no.~1, 28--36.

\bibitem{Macdonald}
I.~G. Macdonald, \emph{Symmetric functions and {Hall} polynomials}, second ed.,
  Oxford Classic Texts in the Physical Sciences, The Clarendon Press, Oxford
  University Press, New York, 2015.

\bibitem{Pushkarev}
I.~A. Pushkarev, \emph{\href{https://doi.org/10.1007/BF02175835}{On the
  representation theory of wreath products of finite groups and symmetric
  groups}}, J.\ Math.\ Sci.\ (N.Y.) \textbf{96} (1999), no.~5, 3590--3599.

\bibitem{Saxl75}
J.~Saxl, \emph{\href{https://doi.org/10.1016/0021-8693(75)90174-X}{Characters
  of multiply transitive permutation groups}}, J.\ Algebra \textbf{34} (1975),
  no.~3, 528--539.

\bibitem{Saxl81}
\bysame, \emph{\href{https://doi.org/10.1017/CBO9781107325579.034}{On
  multiplicity-free permutation representations}}, Finite Geometries and
  Designs: Proceedings of the Second Isle of Thorns Conference 1980, London
  Math.\ Soc.\ Lecture Note Ser., vol.~49, Cambridge Univ. Press, 1981,
  pp.~337--353.

\bibitem{Stembridge89}
J.~R. Stembridge,
  \emph{\href{https://doi.org/10.1016/0001-8708(89)90005-4}{Shifted tableaux
  and the projective representations of symmetric groups}}, Adv.\ Math.
  \textbf{74} (1989), no.~1, 87--134.

\bibitem{Stembridge92}
\bysame, \emph{\href{https://doi.org/10.1016/0021-8693(92)90110-8}{The
  projective representations of the hyperoctahedral group}}, J.\ Algebra
  \textbf{145} (1992), no.~2, 396--453.

\bibitem{SunZhu}
B.~Sun and C.-B. Zhu,
  \emph{\href{https://doi.org/10.4007/annals.2012.175.1.2}{Multiplicity one
  theorems: the {Archimedean} case}}, Ann.\ of Math.\ (2) \textbf{175} (2012),
  no.~1, 23--44.

\bibitem{Tout}
O.~Tout, \emph{\href{https://doi.org/10.4153/S0008439520000259}{Gelfand pairs
  involving the wreath product of finite abelian groups with symmetric
  groups}}, Canad.~Math.~Bull. (to appear).

\end{thebibliography}

\end{document}